\documentclass[10pt]{amsart}
      \usepackage[mathscr]{eucal}
      \usepackage{amsmath,amsfonts}

\usepackage{mathtools}
\usepackage{mathrsfs}
\usepackage{stackengine,scalerel}

\newcommand\obullet[1]{\ThisStyle{\ensurestackMath{%
  \stackon[1pt]{\SavedStyle#1}{\SavedStyle\kern.6\LMpt\bullet}}}}
\newcommand\ocirc[1]{\ThisStyle{\ensurestackMath{%
  \stackon[1pt]{\SavedStyle#1}{\SavedStyle\kern.6\LMpt\circ}}}}

      \parskip\smallskipamount
          \newtheorem{theorem}{Theorem}[section]
      \newtheorem{definition}[theorem]{Definition}
      \newtheorem{proposition}[theorem]{Proposition}
      \newtheorem{corollary}[theorem]{Corollary}
      \newtheorem{lemma}[theorem]{Lemma}
      \newtheorem{example}[theorem]{Example}
      
      \newtheorem{remark}[theorem]{Remark}
      \makeatletter
      \@addtoreset{equation}{section}
      \makeatother

      \newcommand{\BB}{{\mathbb B}}
      \newcommand{\CC}{{\mathbb C}}
      \newcommand{\NN}{{\mathbb N}}
      
      \newcommand{\ZZ}{{\mathbb Z}}
      \newcommand{\DD}{{\mathbb D}}
      \newcommand{\RR}{{\mathbb R}}
      \newcommand{\FF}{{\mathbb F}}
      \newcommand{\TT}{{\mathbb T}}

\DeclareMathOperator{\Span}{span}

      \newcommand{\cA}{{\mathcal A}}
      
      \newcommand{\cC}{{\mathcal C}}
      \newcommand{\cD}{{\mathcal D}}
      \newcommand{\cE}{{\mathcal E}}
      
      \newcommand{\cG}{{\mathcal G}}
      \newcommand{\cH}{{\mathcal H}}
      
      \newcommand{\cJ}{{\mathcal J}}
      \newcommand{\cK}{{\mathcal K}}
      \newcommand{\cL}{{\mathcal L}}
      \newcommand{\cM}{{\mathcal M}}
       \newcommand{\cO}{{\mathcal O}}
      \newcommand{\cQ}{{\mathcal Q}}
      \newcommand{\cN}{{\mathcal N}}
      \newcommand{\cP}{{\mathcal P}}
      
      \newcommand{\cS}{{\mathcal S}}
      \newcommand{\cT}{{\mathcal T}}

      \newdimen\expt
      \def\boxit#1{\setbox0\hbox{$\displaystyle{#1}$}
            \hbox{\lower.4\expt
       \hbox{\lower3\expt\hbox{\lower\dp0
            \hbox{\vbox{\hrule height.4\expt
       \hbox{\vrule width.4\expt\hskip3\expt
            \vbox{\vskip3\expt\box0\vskip2\expt}%
       \hskip3\expt\vrule width.4\expt}\hrule height.4\expt}}}}}}
       
\begin{document}
     
       \pagestyle{myheadings}
      \markboth{ Gelu Popescu}{ Noncommutative domains, universal operator models, and operator algebras }

      \title [ Noncommutative domains, universal operator models, and operator algebras  ]
      {   Noncommutative domains, universal operator models, and operator algebras    }
        \author{Gelu Popescu}
\date{February 28, 2023 (revised version)}
     \thanks{Research supported in part by  NSF grant DMS 1500922}
      \address{Department of Mathematics, The University of Texas
      at San Antonio \\ San Antonio, TX 78249, USA}
      \email{\tt gelu.popescu@utsa.edu}

\begin{abstract}    Let $B(\cH)$ be the algebra of all bounded linear operators on a Hilbert space $\cH$.  The main goal of the  paper  is to find   large classes of noncommutative   domains   in $B(\cH)^n$  with prescribed  universal operator models, acting on the full Fock space  with $n$ generators, and to study these domains and their universal models  in connection with the Hardy algebras and the $C^*$-algebras  they generate.
 While the class  of these domains contains the  regular noncommutative domains previously studied in the literature, the main  focus of the present paper is on  the {\it non-regular}  domains. The multi-variable operator theory of these domains is developed throughout the paper. 
\end{abstract}

      \maketitle

\section*{Contents}
{\it

\quad Introduction
\begin{enumerate}
   \item[1.]     Universal operator  models associated with  free holomorphic functions
   \item[2.]     Noncommutative Hardy algebras and functional calculus
   \item[3.]     Multipliers and multi-analytic  operators
   \item[4.]      Admissible free holomorphic functions for operator model theory
   \item[5.]     $C^*$-algebras associated with noncommutative domains and classification 
  
   \end{enumerate}
    
    \quad References

}

\section*{Introduction}

Let  $S$ be the unilateral shift  defined by $(S\varphi) (z):=zf(z)$ on the   Hardy space $H^2(\DD)$  of all analytic functions $f(z)=\sum_{k=0}^\infty c_k z^k$ on the open unit disc  $\DD:=\{z\in\CC: \ |z|<1\}$  with square-summable coefficients. It is well-known that  a bounded linear operator $T\in B(\cH)$   has its adjoint unitarily equivalent to $(S^*\otimes I_\cE)|_\cN$, where $\cN$ is a co-invariant subspace of $S\otimes I_\cE$, if and only if  $T$ is a pure contraction, i.e. $\|T\|\leq1$ and ${T^*}^n\to 0$ strongly, as $n\to\infty$.  In this case, we say that $S$ is a universal model for $T$. The closure  (in the operator norm) of the set of all operators which admit $S$ as universal model    coincides with the operator unit ball
$[B(\cH)]_1:=\{X\in B(\cH):\ XX^*\leq I\}$. The study of this  ball has generated the  Sz.-Nagy--Foia\c s   \cite{SzFBK-book}, \cite{S}  theory of contractions   which had profound implications in operator theory, function theory, interpolation and system theory, prediction theory, etc.

A free analogue of Sz.-Nagy--Foia\c s theory  has been developed
 for  the closed unit ball
$$
[B(\cH)^n]_1:=\{(X_1,\ldots, X_n)\in B(\cH)^n:\ X_1 X_1^*+\cdots +X_nX_n^*\leq I\}
$$  
(see \cite{F}, \cite{B}, \cite{Po-isometric}, \cite{Po-charact}, \cite{Po-von}, \cite{Po-funct}, \cite{Po-analytic},  \cite{Po-poisson}, \cite{APo1}, \cite{APo2}, \cite{Po-curvature}, \cite{Po-Nehari}, \cite{Po-entropy}, \cite{Po-varieties}, \cite{Po-unitary}, \cite{Po-automorphisms},   \cite{BV}, \cite{DP1}, \cite{DP2}, \cite{DKP}, \cite{SSS1}, \cite{SSS2} and the reference there in). The corresponding universal model is the $n$-tuple of left creation operators  ${\bf S}=(S_1,\ldots, S_n)$ acting on the full Fock space with $n$ generators.
We know that an $n$-tuple $T=(T_1,\ldots, T_n)\in B(\cH)^n$ admits ${\bf S}$ as universal model, i.e. there is a Hilbert space $\cG$  such that $T_i^*=(S^*_i\otimes I_\cG)|_{\cH}$ for every $i\in \{1,\ldots, n\}$, if and only if 
$T$ is a  pure   row contraction, i.e.  $T\in [B(\cH)^n]_1$ and   SOT-$\lim_{k\to \infty} \sum_{|\alpha|=k} T_\alpha T_\alpha^*=0$.
The closure  (in the operator norm) of the set of all  $n$-tuples  $T\in B(\cH)^n$ which admit
 ${\bf S}$ as universal model   coincides with   the closed ball $[B(\cH)^n]_1$.
 
 Let $\FF_n^+$ be the unital free semigroup on $n$ generators
$g_1,\ldots, g_n$ and the identity $g_0$. The length of $\alpha\in
\FF_n^+$ is defined by $|\alpha|:=0$ if $\alpha=g_0$  and
$|\alpha|:=k$ if
 $\alpha=g_{i_1}\cdots g_{i_k}$, where $i_1,\ldots, i_k\in \{1,\ldots, n\}$.  If $Z_1,\ldots, Z_n $ are noncommuting indeterminates   and $\alpha=g_{i_1}\cdots g_{i_k}$,    we
denote $Z_\alpha:= Z_{i_1}\cdots Z_{i_k}$  and $Z_{g_0}:=1$.
Let $p:=\sum_{\alpha\in \FF_n^+} a_\alpha
Z_\alpha$ be a positive regular noncommutative polynomial,  i.e. 
$a_\alpha\geq 0$ for every $\alpha\in \FF_n^+$, \ $a_{g_0}=0$,
 \ and  $a_{g_i}>0$ for all $i\in \{1,\ldots, n\}$.
 Given $m\in \NN:=\{1,2,\ldots \}$, we define  the  noncommutative domain
$$
{\cD}_p^m(\cH):=\left\{ X=(X_1,\ldots, X_n)\in B(\cH)^n: \ \Phi_{p,X}(I)\leq I \text{ and } 
( id-\Phi_{p,X})^m(I)\geq 0  \right\},
$$
where $\Phi_{p,X}:B(\cH)\to B(\cH)$ is the completely positive linear map
given by $\Phi_{p,X}(Y):=\sum_{\alpha\in \FF_n^+} a_\alpha X_\alpha YX_\alpha^*$.
  When $m\geq 1$, $n\geq 1$,   the noncommutative domain ${\cD}_p^m(\cH)$ was studied  in   \cite{Po-domains} (when $m=1$), and in \cite{Po-Berezin} (when $m\geq 2$). In this case, the corresponding universal model is an $n$-tuple of weighted left creation operators ${\bf W}=(W_1,\ldots, W_n)$  acting on the full Fock space.
  Even in this case, the noncommutative domain ${\cD}_p^m(\cH)$ coincides with the closure of the set of all $n$-tuples   of operators $T\in B(\cH)^n$ which admit
 ${\bf W}$ as universal model.  
  
  We should mention that,  in the commutative setting (i.e. $X_iX_j=X_jX_i$) 
  if  $m\geq 1$ and  $p=Z_1+\cdots + Z_n$  the corresponding commutative domain was studied by  Drury \cite{Dru}, Arveson \cite{Arv3-acta}, Bhattacharyya-Eschmeier-Sarkar \cite{BES}, Timotin \cite{T},  and by the author \cite{Po-poisson}, \cite{Po-varieties} (when $m=1$), and  by Athavale \cite{At}, M\" uller \cite{M}, M\" uller-Vasilescu \cite{MV},
   Vasilescu \cite{Va},  and Curto-Vasilescu \cite{CV1} (when $m\geq 2$).
    If
$m\geq 2$ and    $p=Z$,  the
    corresponding domain coincides with the set of all
    $m$-hypercontractions  studied by Agler  in \cite{Ag1}, \cite{Ag2},
    and  recently by
    Olofsson \cite{O1}, \cite{O2}, and by Ball and Bolotnikov \cite{BB}.  
 The commutative case when $m\geq 1$ and   $p$ is a positive regular   polynomial
  was considered  by  S. Pott \cite{Pot}.        

Inspired by the above-mentioned work in multi-variable operator theory  and by the more  recent papers  by
  Clou\^ atre and Hartz  \cite{CH}, and  by Bickel, Hartz, and McCarthy \cite{BHM}  on various aspects of the multiplier algebras associated with reproducing kernel Hilbert spaces, we were led to writing the present paper, which has the roots in our recent work \cite{Po-invariant}, \cite{Po-Bergman}, \cite{Po-wold} in noncommutative multivariable operator theory.

  The main goal of the paper  is to find   large classes of noncommutative   domains   in $B(\cH)^n$  with prescribed  universal models, acting on the full Fock space $F^2(H_n)$ with $n$ generators, and to study these domains and their universal models  in connection with the Hardy algebras and $C^*$-algebras  they generate. The emphasize  is on the multi-variable operator theory  of the {\it non-regular noncommutative domains} in $B(\cH)^n$.  
   Given a formal power series $g=1+\sum_{|\alpha|\geq 1}b_\alpha Z_\alpha$ with strictly positive coefficients, we associate
an $n$-tuple ${\bf W}=(W_1,\ldots, W_n)$ of weighted left creation operators acting on the full Fock space. The  first goal  is to find  minimal conditions on the coefficients $\{b_\alpha\}$ to ensure that ${\bf W}$ is the {\it universal operator  model}  for   the $n$-tuples of operators in a noncommutative set $\cD(\cH)\subset B(\cH)^n$ which contains a ball $[B(\cH)^n]_\epsilon$ for some $\epsilon>0$, i.e. for each $T\in \cD(\cH)$, there is a Hilbert space $\cG$ such that
$$
T_i^*=(W_i^*\otimes I_\cG)|_{\cH},\qquad i\in \{1,\ldots, n\}.
$$
Moreover, we want to  find the maximal noncommutative set $\cD(\cH)$ so that each $T\in \cD(\cH)$ has ${\bf W}$ as universal operator model.
This goal is achieved  by  introducing the class of  {\it admissible free holomorphic functions  for operator model theory} and by showing that it  is in one-to-one  correspondence with a class of universal models ${\bf W}=(W_1,\ldots, W_n)$ which is   in one-to-one  correspondence with a class of  maximal noncommutative domains in $B(\cH)^n$.  
We  study these noncommutative domains, their universal models and the algebras they generate. While the class  of these domains contains the regular domains previously studied in \cite{Po-domains}, the main  focus of the present paper will be on  non-regular noncommutative domains in $B(\cH)^n$  and their  connection with the regular ones. The operator theory of these  domains is developed throughout the paper.

Now, let us discuss the structure of the paper.
 For each  formal power series $g= 1+\sum_{ |\alpha|\geq 1} b_\alpha
Z_\alpha$ with $b_\alpha>0$,
 we  define, in Section 1,  an $n$-tuple ${\bf W}=(W_1,\ldots, W_n)$
of weighted left creation operators acting on the full Fock space  $F^2(H_n)$. We provide necessary and sufficient conditions  so that   an $n$-tuple $T=(T_1,\ldots, T_n)\in B(\cH)^n$  admits ${\bf W}$ as universal model (see Theorem \ref{intertwining}).  This leads to the introduction of  the class of free holomorphic functions admissible for multi-variable 
operator model theory.  Here are the questions that we  answer in Section 1.
Under what conditions on the weights  $\{b_\alpha\}_{\alpha\in \FF_n^+}$  is the associated $n$-tuple 
${\bf W}=(W_1,\ldots, W_n)$   of weighted left creation operators on the full Fock space $F^2(H_n)$   a
{\it universal model } for the $n$-tuples  of operators in a noncommutative set  $\cD(\cH)\subset B(\cH)^n$ which contains a ball $[B(\cH)^n]_\epsilon$  for some $\epsilon>0$ $?$  What is the maximal noncommutative set $\cD(\cH)$ so that each $n$-tuple $T\in \cD(\cH)$ has ${\bf W}$ as universal model $?$
The main result of this section asserts   that, given an admissible  free holomorphic function  $g$  and  an $n$-tuple of operators $T=(T_1,\ldots, T_n)\in B(\cH)^n$, there is a Hilbert spaces $\cG$ such that 
$$
T_i^*=(W_i^*\otimes I_\cG)|_{\cH},\qquad i\in \{1,\ldots, n\},
$$
where $\cH$ is  identified with a coinvariant subspace  for the operators $W_1\otimes I_\cG,\ldots, W_n\otimes I_\cG$, if and only  if $T$  is    a {\it  pure $n$-tuple} in the noncommutative set  $\cD_{g^{-1}}(\cH)$, which is precisely described. In addition, we  introduce the {\it noncommutative domain algebra} $\cA(g)$ as the norm-closed non-self-adjoint algebra generated by $W_1,\ldots, W_n$ and the the identity and prove that, for any  
$T=(T_1,\ldots, T_n)\in \overline{\cD^{pure}_{g^{-1}}(\cH)}$,  there is a unital completely contractive homomorphism 
$$
\Psi_T:\cA(g)\to B(\cH),\quad \Psi_T(q({\bf W})):=q(T),
$$
for any polynomial $q$  in $n$ noncommutative indeterminates.

In Section 2,  we introduce the noncommutative Hardy algebras $F^\infty(g)$ associated with  any
 admissible free holomorphic function  $g$ such that 
$$
\sup_{\alpha\in \FF_n^+} \frac{b_\alpha}{b_{\alpha g_i}}<\infty,\qquad i\in \{1,\ldots, n\}.
$$
If  ${\bf W}=(W_1,\ldots, W_n)$ is  the universal model of the noncommutative domain $\cD_{g^{-1}}$, 
  we show that  $F^\infty(g)$ is the sequential SOT-(resp. WOT-, w*-) closure of   polynomials in $W_1,\ldots, W_n$ and the identity,  and 
     each $\varphi({\bf W})\in F^\infty(g)$ has a unique   Fourier representation $\sum_{\alpha\in \FF_n^+} c_\alpha W_\alpha$.
 These results are necessary in order to  
   provide a $w^*$-continuous $F^\infty(g)$-functional calculus for pure (resp. completely non-coisometric) $n$-tuples of operators in the noncommutative domain $\cD_{g^{-1}}(\cH)$. More precisely, we show that if   $T\in \overline{\cD^{pure}_{g^{-1}}(\cH)}$ is completely non-coisometric with respect to $\cD_{g^{-1}}(\cH)$, then the completely contractive  linear map  $\Psi_T$, described above, 
 has a unique extension to a $w^*$-continuous  homomorphism
$\Psi_T:F^\infty(g)\to B(\cH)$ 
 which is  completely  contractive and  sequentially WOT-continuous (resp. SOT-continuous).
 Moreover,  if $\varphi({\bf W})\in F^\infty(g)$ has the Fourier representation $\sum_{\alpha\in \FF_n^+} c_\alpha W_\alpha$, then
 $$
 \Psi_T(\varphi({\bf W}))=\text{\rm SOT-} \lim_{N\to\infty}\sum_{s\in \NN, s\leq N} \left(1-\frac{s}{N+1}\right)\sum_{|\alpha |=s} c_\alpha T_\alpha.
 $$

Section 3  is devoted to the multipliers on the Hilbert space  
$$
F^2(g)=\left\{ \zeta=\sum_{\alpha\in \FF_n^+}c_\alpha Z_\alpha: \   \|\zeta\|_g^2:=\sum_{\alpha\in \FF_n^+}\frac{1}{b_\alpha} |c_\alpha|^2<\infty, \ c_\alpha\in \CC\right\}
$$
of formal power series in indeterminates $Z_1,\ldots, Z_n$ with orthogonal basis $\{Z_\alpha:\ \alpha\in \FF_n^+\}$ such that $\|Z_\alpha\|_g:=\frac{1}{\sqrt{b_\alpha}}$. 
In this setting, the {\it left multiplication} operators  $L_{Z_i}$ are  defined by  $L_{Z_i}\zeta:=Z_i\zeta$ for all $\zeta\in F^2(g)$.  Due to the fact that the $n$-tuple $L_{\bf Z}:=(L_{Z_1},\ldots,  L_{Z_n})$  is unitarily equivalent to the the universal model ${\bf W}:=(W_1,\ldots, W_n)$ associated with $g$, we prove that 
the left multiplier algebra $\cM^\ell(F^2(g)\otimes \cK)$ is completely isometrically  isomorphic  to the noncommutative Hardy algebra $F^\infty(g)\bar\otimes_{min} B(\cK)$. We also discuss the connection between the multipliers and  the multi-analytic operators on Fock spaces.

  In Section 4,  we introduce
  several classes of admissible free holomorphic functions  for  operator model theory and the associated noncommutative domains. The examples presented  here will be referred to throughout the paper. What stands out are the new classes of noncommutative domains which, as far as we know,  have not been studied before. 
  For example, the  scale of   {\it weighted noncommutative Bergman spaces}    $\{F^2(g_s):  s\in(0,\infty)\}$, where
  $$
g_s:=1+\sum_{k=1}^\infty\left(\begin{matrix} s+k-1 \\k \end{matrix}\right)(Z_1+\cdots  +Z_n)^k,
$$ 
and the associated domains $\cD_{g_s^{-1}}(\cH)$ were never considered and studied  in the  multivariable noncommutative setting, except the case when $s\in \NN$ (see \cite{Po-Berezin}).    We remark that, when $s\in (0,1]$,   the Hilbert  spaces $F^2(g_s)$  are  noncommutative generalizations of the classical   Besov-Sobolev spaces on the unit ball $\BB_n\subset \CC^n$ with representing kernel $\frac{1}{(1-\left<z,w\right>)^s}$, $z,w\in \BB_n$.
 Another example, is the scale of spaces $\{F^2(\xi_s):  s\in \RR\}$, where
 $$
\xi_s:=1+\sum_{\alpha\in \FF_n^+, |\alpha|\geq 1} (|\alpha|+1)^s Z_\alpha,
$$
 which 
 contains  the {\it noncommutative Dirichlet space}  $F^2(\xi_{-1})$ over the unit ball
 $[B(\cH)^n]_1$. When $s=0$, we recover again the full Fock space with $n$ generators.
 
Finally, if    $s\in [1,\infty)$ and  $\varphi=\sum_{|\alpha|\geq 1} d_\alpha Z_\alpha$ is a noncommutative polynomial such that $d_\alpha \geq 0$ and $d_{g_i}>0$ for $i\in \{1,\ldots, n\}$, then    the  formal power series  
$\psi_s =(1-\varphi)^{-s}$  is an admissible free holomorphic function.  The scale of noncommutative domains $\{\cD_{\psi_s^{-1}}:  s\in [1,\infty)\}$ contains the particular cases when $s\in \NN$, which were considered in \cite{Po-Berezin}.

In Section 5, we associate with each universal model ${\bf W}=(W_1,\ldots, W_n)$ the $C^*$-algebra  $C^*({\bf W})$ generated by $W_1,\ldots, W_n$ and identity. Under the assumption that  $\Delta_{g^{-1}}({\bf W},{\bf W}^*):=\sum_{k=0}^\infty\sum_{|\alpha|=k} a_\alpha W_\alpha W_\alpha^*$ is convergent in the operator norm topology, where $g^{-1}=1+\sum_{|\alpha|\geq 1} a_\alpha Z_\alpha$, condition that is satisfied by a large class of universal models,
we obtain   Wold  \cite{W} decompositions for the  unital $*$-representations of $C^*({\bf W})$.   More precisely, if  $\pi$ is     a unital
$*$-representation  of the $C^*$-algebra $C^*({\bf W})$  on a separable Hilbert
space  $\cK$ and
$V_i:=\pi(W_i)$, then the  {\it noncommutative Berezin kernel} ${ K_{g,V}}$ associated with $g$ and $V$    is a   partial isometry. Setting
$$
\cK^{(0)}:=\text{\rm range}\, K_{g,V}^* \ \text{ and }  \ \cK^{(1)}:= \ker K_{g,V},
$$
the orthogonal decomposition  
 $\cK=\cK^{(0)}\bigoplus \cK^{(1)}$  has  the following properties.
\begin{enumerate}
\item[(i)] $\cK^{(0)}$ and $ \cK^{(1)}$ are reducing  subspaces for each operator $V_i$.

\item[(ii)]  $V|_{\cK^{(0)}}:=(V_1|_{\cK^{(0)}},\ldots, V_n|_{\cK^{(0)}})$ is  a {\it pure tuple} in
 $\cD_{g^{-1}}(\cK^{(0)})$.

\item[(iii)]   $V|_{\cK^{(1)}}:=(V_1|_{\cK^{(1)}},\ldots, V_n|_{\cK^{(1)}})$  is a {\it Cuntz tuple} in
$\cD_{g^{-1}}(\cK^{(1)})$.

\end{enumerate}
We introduce 
  the algebra  $\cO(g)$  as the universal $C^*$-algebra  generated by $\pi(W_1),\ldots, \pi(W_n)$ and the identity, where 
  $\pi$ is a  {\it Cuntz type unital $*$-representation} of   $C^*({\bf W})$, i.e.
  $\Delta_{g^{-1}}(\pi({\bf W}),\pi({\bf W})^*)=0$.
    We prove that
   the sequence
of $C^*$-algebras
$$
0\to \boldsymbol\cK(F^2(H_n))\to C^*({\bf W})\to \cO(g)\to 0
$$
is exact,  where  $\boldsymbol\cK(F^2(H_n))$ stands for the ideal of all compact
operators  in $B(F^2(H_n))$.
This extends  the corresponding results obtained by Coburn  \cite{Co} and Cuntz \cite{Cu}. We should mention that  these results carry over
  to  $C^*$-algebras   $C^*(\pi({\bf W}))$, where $\pi:C^*({\bf W})\to B(\cK)$ is an arbitrary unital  $*$-representation on a Hilbert space $\cK$  which is not a Cuntz type $*$-representation.
 
 In a sequel  \cite{Po-Noncommutative domains II} to the present paper, we present a Beurling   type characterization  of the joint invariant subspaces of $W_1,\ldots, W_n$ and   develop
 a  dilation theory for not necessarily pure  $n$-tules of operators  in noncommutative domains associated with admissible free holomorphic functions  for operator model theory.     
  To what extent the results  of the present paper   and those from \cite{Po-Noncommutative domains II}  extend    to  noncommutative varieties  in non-regular noncommutative domains in $B(\cH)^n$ is discussed  in  a forthcoming paper.

\bigskip

\section{Universal operator  models associated with  free holomorphic functions}

In this section, we introduce the class of free holomorphic functions admissible for multi-variable 
operator model theory. Each such a function uniquely defines an $n$-tuple ${\bf W}=(W_1,\ldots, W_n)$
of weighted left creation operators acting on the full Fock space with $n$ generators, which turns out to be the universal model for a certain maximal noncommutative domain in $B(\cH)^n$. The operator theory of these noncommutative domains is developed throughout the paper.

 Let $H_n$ be an $n$-dimensional complex  Hilbert space with orthonormal
      basis
      $e_1,\dots,e_n$, where $n\in\NN$.        
      The full Fock space  of $H_n$ is defined by
      $$F^2(H_n):=\bigoplus_{k\geq 0} H_n^{\otimes k},$$
      where $H_n^{\otimes 0}:=\CC 1$ and $H_n^{\otimes k}$ is the (Hilbert)
      tensor product of $k$ copies of $H_n$.
      Define the {\it left creation
      operators} $S_i:F^2(H_n)\to F^2(H_n), \  i\in\{1,\dots, n\}$,  by
      $$
       S_i\varphi:=e_i\otimes\varphi, \quad  \varphi\in F^2(H_n),
      $$
      and  the {\it right creation operators}
      $R_i:F^2(H_n)\to F^2(H_n)$  by
      $
       R_i\varphi:=\varphi\otimes e_i$, \ $ \varphi\in F^2(H_n)$. The noncommutative disk algebra \cite{Po-von} is the norm-closed non-self-adjoint algebra generated by $S_1,\ldots, S_n$ and the identity.
       
Let $\FF_n^+$ be the unital free semigroup on $n$ generators
$g_1,\ldots, g_n$ and the identity $g_0$.  The length of $\alpha\in
\FF_n^+$ is defined by $|\alpha|:=0$ if $\alpha=g_0$  and
$|\alpha|:=k$ if
 $\alpha=g_{i_1}\cdots g_{i_k}$, where $i_1,\ldots, i_k\in \{1,\ldots, n\}$. We set  $e_\alpha:=e_{g_{i_1}}\otimes \cdots \otimes e_{g_{i_k}}$ and  $e_{g_0}=1$, and note that $\{e_\alpha\}_{\alpha\in \FF_n^+}$ is an othonormal basis for $F^2(H_n)$.
 Given a formal power series  $f:= \sum_{\alpha\in \FF_n^+} a_\alpha
Z_\alpha$, \ $a_\alpha\in \CC$, in indeterminates $Z_1,\ldots, Z_n$, we say that $f$ is
a {\it free holomorphic function} in a neighborhood of the origin if there is $\delta>0$ such that  the series
$\sum_{k=0}^\infty \sum_{|\alpha|=k} a_\alpha X_\alpha$ is convergent in norm  for any
$(X_1,\ldots, X_n)\in [B(\cH)^n]_\rho$ and any Hilbert space $\cH$,  where
$$[B(\cH)^n]_\rho:=\{(X_1,\ldots, X_n)\in B(\cH)^n: \
\|X_1X_1^*+\cdots + X_nX_n^*\|^{1/2}\leq \rho\}.
$$

\begin{lemma}  \label{lem1}  A formal power series  $g:= 1+\sum_{ |\alpha|\geq 1} b_\alpha
Z_\alpha$   is  a free holomorphic function in a neighborhood of the origin if and only if there is  a 
free holomorphic function  $f=1+\sum_{ |\gamma|\geq 1} a_\gamma
Z_\gamma$ in a neighborhood of the origin such that $fg=gf=1$. In this case, we have 
\begin{equation*} 
b_\alpha= \sum_{j=1}^{|\alpha|}
\sum_{{\gamma_1\cdots \gamma_j=\alpha }\atop {|\gamma_1|\geq
1,\ldots, |\gamma_j|\geq 1}}  (-1)^ja_{\gamma_1}\cdots a_{\gamma_j}   \quad
\text{ and  } \quad a_\alpha= \sum_{j=1}^{|\alpha|}
\sum_{{\gamma_1\cdots \gamma_j=\alpha }\atop {|\gamma_1|\geq
1,\ldots, |\gamma_j|\geq 1}}  (-1)^jb_{\gamma_1}\cdots b_{\gamma_j}
\end{equation*}
for every  $\alpha\in \FF_n^+$ with $ |\alpha|\geq 1$.
\end{lemma} 
\begin{proof}  The case when $g=1$ is trivial. Assume that $g$ is  a non-constant  free holomorphic function in a neighborhood of the origin. According to \cite{Po-holomorphic},   $
\limsup_{k\to\infty} \left( \sum_{|\alpha|=k}
|b_\alpha|^2\right)^{1/2k}<\infty 
$  
and there is $r\in(0,1)$ such that   
 $
 \varphi(rS_1,\ldots, rS_n):=\sum_{p=1}^\infty\sum_{ |\alpha|= p} r^{|\alpha|}b_\alpha S_\alpha
 $  is convergent in the operator norm topology and  $\| \varphi(rS_1,\ldots, rS_n)\|<1$.
Therefore, the operator $g(rS_1,\ldots, rS_n)=I+\varphi(rS_1,\ldots, rS_n)$ is invertible, different   from the identity, and 
its inverse $$g(rS_1,\ldots, rS_n)^{-1}= I- \varphi(rS_1,\ldots, rS_n)+ \varphi(rS_1,\ldots, rS_n)^2-\cdots $$
is  in the noncommutative disk algebra $\cA_n$. Moreover,  $g(rS_1,\ldots, rS_n)^{-1}$
  has  a unique  ``Fourier representation''
$  \sum_{\alpha\in \FF_n^+} a_\alpha
r^{|\alpha|} S_\alpha $ for some constants $a_\alpha\in \CC$. Consequently,
  using the fact that $r^{|\alpha|}a_\alpha=P_{\CC 1} S_\alpha^*
g(rS_1,\ldots, rS_n)^{-1}(1)$, we  deduce that
 $g(rS_1,\ldots, rS_n)^{-1}$ has the Fourier representation 
$$ I+\sum_{m=1}^\infty \sum_{|\alpha|=m}\left(\sum_{j=1}^{|\alpha|}
\sum_{{\gamma_1\cdots \gamma_j=\alpha }\atop {|\gamma_1|\geq 1,\ldots,
 |\gamma_j|\geq 1}} (-1)^jb_{\gamma_1}\cdots b_{\gamma_j} \right) r^{|\alpha|}
  S_\alpha.
$$
Due to the uniqueness of the Fourier representation of the elements
in $\cA_n$,  we have
\begin{equation*}
a_{g_0}=1 \quad \text{ and }\quad a_\alpha= \sum_{j=1}^{|\alpha|}
\sum_{{\gamma_1\cdots \gamma_j=\alpha }\atop {|\gamma_1|\geq
1,\ldots, |\gamma_j|\geq 1}}(-1)^j b_{\gamma_1}\cdots b_{\gamma_j}   \quad
\text{ if } \ |\alpha|\geq 1.
\end{equation*}
Moreover, taking $0<r_0<r$, we have 
$g(r_0S_1,\ldots, r_0S_n)^{-1}=1+\sum_{m=1}^\infty \sum_{|\alpha|=m} a_\alpha r_0^{|\alpha|} S_\alpha
$
where the convergence is  in the operator norm topology.
Due to the noncommutative von Neumann inequality \cite{Po-von}, we deduce that 
  $f:=1+\sum_{ |\gamma|\geq 1} a_\gamma
Z_\gamma$   is  a 
free holomorphic function in  the noncommutative ball $[B(\cH)^n]_{r_0}$  and, taking into account  the relations above, we also have $fg=gf=1$.  The converse can be proved in a similar manner.
The proof is complete.
\end{proof}

Let $\{b_\alpha\}_{\alpha\in \FF_n^+}$ be  a collection of strictly positive  real numbers such that
\begin{equation*}
b_{g_0}=1\quad \text{ and }  \quad
\sup_{\alpha\in \FF_n^+} \frac{b_\alpha}{b_{g_i \alpha}}<\infty,\qquad i\in \{1,\ldots, n\},
\end{equation*}
and let  $g= 1+\sum_{ |\alpha|\geq 1} b_\alpha
Z_\alpha$ be the  associated formal power series.
We  define the {\it weighted left creation  operators}
$W_i:F^2(H_n)\to F^2(H_n)$, $i\in \{1,\ldots, n\}$,  associated with  $g$     by setting $W_i:=S_iD_i$, where
 $S_1,\ldots, S_n$ are the left creation operators on the full
 Fock space $F^2(H_n)$ and the diagonal operators $D_i:F^2(H_n)\to F^2(H_n)$,
 are defined by setting
$$
D_ie_\alpha=\sqrt{\frac{b_\alpha}{b_{g_i \alpha}}} e_\alpha,\qquad
 \alpha\in \FF_n^+,
$$
where  $\{e_\alpha\}_{\alpha\in \FF_n^+}$ is the othonormal basis for $F^2(H_n)$.
If   $\{b_\alpha^\prime\}_{\alpha\in \FF_n^+}$  is another sequence  of positive real  numbers  with similar properties as $\{b_\alpha\}_{\alpha\in \FF_n^+}$ and   $(W_1',\ldots,
W_n')$ are  the associated  weighted left creation operators, then
there exists a unitary operator $U\in B(F^2(H_n))$ such that
$
UW_i=W_i'U,\  i\in\{1,\ldots, n\},
$
if and only if $b_\alpha=b_\alpha^\prime$ for any $\alpha\in \FF^+$. Since the proof is similar to that of Theorem 1.2 from \cite{Po-domains}, we omit it.

\begin{definition}  Let $T=(T_1,\ldots, T_n)\in B(\cH)^n$.
If there is    a Hilbert space $\cE$ and an isometric operator $V:\cH\to F^2(H_n)\otimes \cE$ such that 
$$
VT_i^*=(W_i^*\otimes I_\cE)V,\qquad i\in \{1,\ldots,n\},
$$ 
we say that ${\bf W}=(W_1,\ldots, W_n)$ is a {\it universal model} for $T$ . 
\end{definition}

\begin{theorem} \label{intertwining}
Let $\{b_\alpha\}_{\alpha\in \FF_n^+}$ be  a collection of strictly positive  real numbers such that
\begin{equation*}
b_{g_0}=1\quad \text{ and }  \quad
\sup_{\alpha\in \FF_n^+} \frac{b_\alpha}{b_{g_i \alpha}}<\infty,\qquad i\in \{1,\ldots, n\},
\end{equation*}
and let ${\bf W}=(W_1,\ldots, W_n)$ be the associated $n$-tuple of weighted left creation operators on the full Fock space $F^2(H_n)$. If  $T=(T_1,\ldots, T_n)\in B(\cH)^n$,  then the following statements are equivalent.
\begin{enumerate}
\item[(i)]
There is    a Hilbert space $\cE$ and an operator $V:\cH\to F^2(H_n)\otimes \cE$ such that 
$$
VT_i^*=(W_i^*\otimes I_\cE)V,\qquad i\in \{1,\ldots,n\}.
$$
\item[(ii)] There is a  Hilbert space $\cE$  and  an operator $\Theta_0\in B(\cH, \cE)$, 
$\Theta_0\neq 0$,  such that the series 
$$
\sum_{\alpha\in \FF_n^+} b_\alpha T_\alpha \Theta_0^*\Theta_0 T_\alpha^*
$$
converges in the weak operator topology.
\item[(iii)]
There is a positive operator $D(T)\in B(\cH)$ such that the series 
$$
\sum_{\alpha\in \FF_n^+} b_\alpha T_\alpha D(T) T_\alpha^*
$$
converges in the weak operator topology.
\end{enumerate}
If condition (iii) holds,    then the operator  $\widetilde V:\cH\to F^2(H_n)\otimes \cD$  defined  by setting
$$
\widetilde Vh:=\sum_{\alpha\in \FF_n^+} \sqrt{b_\alpha} e_\alpha\otimes D(T)^{1/2} T_\alpha^*h,\qquad h\in \cH,
$$
where $\cD:=\overline{D(T) \cH}$, satisfies the relation
$
\widetilde VT_i^*=(W_i^*\otimes I_\cD)\widetilde V$,  $i\in \{1,\ldots,n\}.
$
In addition, $\widetilde V$ is an isometry if and only if 
$$
\sum_{\alpha\in \FF_n^+} b_\alpha T_\alpha D(T) T_\alpha^*=I_\cH.
$$
where the convergence is  in the weak operator topology.
\end{theorem}
\begin{proof}
Assume that condition (i) holds. For each $\alpha\in \FF_n^+$,  define  the bounded  operator 
$\Theta_{(\alpha)}:=uP_{e_\alpha\otimes \cE} V$,  where $u: e_\alpha\otimes \cE\to \cE$ is the unitary operator given by $u(e_\alpha\otimes x):=x$, $x\in \cE$. It is clear that 
$Vh=\sum_{\alpha\in \FF_n^+} e_\alpha\otimes \Theta_{(\alpha)} h$  for all $h\in \cH$ and 
\begin{equation}\label{bound}
\sum_{\alpha\in \FF_n^+} \Theta_{(\alpha)}^*\Theta_{(\alpha)}\leq \|V\|^2 I.
\end{equation}
For each $i\in \{1,\ldots, n\}$, we have
$$
VT_i^*h=\sum_{\beta\in \FF_n^+} e_\beta\otimes \Theta_{(\beta)} T_i^*h, \qquad h\in \cH,
$$
and, due to the definition of the weighted shift $W_i$, we have  
\begin{equation*}
\begin{split}
(W_i^*\otimes I_\cE)Vh&= (W_i^*\otimes I_\cE)\left(\sum_{\beta\in \FF_n^+} e_{g_i\beta} \otimes \Theta_{(g_i\beta)}h\right)
=\sum_{\beta\in \FF_n^+}  \sqrt{\frac{b_\beta}{b_{g_i \beta}}} e_\beta\otimes \Theta_{(g_i\beta)}h.
\end{split}
\end{equation*}
Since $VT_i^*=(W_i^*\otimes I_\cE)V$, we deduce that
\begin{equation}\label{TT}
\frac{1}{\sqrt{b_\beta}} \Theta_{(\beta)} T_i^*h=\frac{1}{\sqrt{b_{g_i\beta}}} \Theta_{(g_i\beta)}
\end{equation}
for any $\beta\in \FF_n^+$ and any $i\in \{1,\ldots, n\}$.
Let $\gamma=g_{i_1}\cdots g_{i_k}\in \FF_n^+$. Using relation \eqref{TT} repeatedly, we obtain 
\begin{equation*}
\begin{split}
\frac{1}{\sqrt{b_{g_{i_2}\cdots g_{i_k}}}} \Theta_{(g_{i_2}\cdots g_{i_k})} T_{i_1}^*h
&=\frac{1}{\sqrt{b_{g_{i_1}\cdots g_{i_k}}}} \Theta_{(g_{i_1}\cdots g_{i_k})}\\
\frac{1}{\sqrt{b_{g_{i_3}\cdots g_{i_k}}}} \Theta_{(g_{i_3}\cdots g_{i_k})} T_{i_2}^*T_{i_1}^*h
&=\frac{1}{\sqrt{b_{g_{i_2}\cdots g_{i_k}}}} \Theta_{(g_{i_2}\cdots g_{i_k})}T_{i_1}^*h\\
&\  \ \vdots\\
\frac{1}{\sqrt{b_{g_0}}} \Theta_{(g_0)} T_{i_k}^*\cdots T_{i_1}^*h &=\frac{1}{\sqrt{b_{g_k}}} \Theta_{(g_k)}T_{i_{k-1}}^*\cdots T_{i_1}^*h.
\end{split}
\end{equation*}
Using these relations and taking into account that $b_{g_0}=1$, we deduce that
$$
\Theta_{(g_0)} T_{i_k}^*\cdots T_{i_1}^*h=\frac{1}{\sqrt{b_{g_{i_1}\cdots g_{i_k}}}} \Theta_{(g_{i_1}\cdots g_{i_k})}
$$
for any $\gamma=g_{i_1}\cdots g_{i_k}\in \FF_n^+$. Hence, we have
$
\Theta_{(\gamma)}=\sqrt{b_\gamma} \Theta_0 T_\gamma^*,\  \gamma\in \FF_n^+,
$
where $\Theta_0:=\Theta_{(g_0)}$, and relation \eqref{bound} shows that
the series 
$
\sum_{\alpha\in \FF_n^+} b_\alpha T_\alpha \Theta_0^*\Theta_0 T_\alpha^*
$
converges in the weak operator topology. Therefore, item (ii) holds.
If item (ii) holds, then  relation (iii) follows if we take  $D(T):=\Theta_0^* \Theta_0$.

Now, assume the item (iii) holds .
Set $\cD:=\overline{D\cH}$ and define 
the operator  $\widetilde V:\cH\to F^2(H_n)\otimes \cD$   by  relation
$$
 \widetilde Vh:=\sum_{\alpha\in \FF_n^+} \sqrt{b_\alpha} e_\alpha\otimes D(T)T_\alpha^*h,\qquad h\in \cH.
$$
Note that  $\widetilde V$ is a bounded operator and
 $
\widetilde VT_i^*=(W_i^*\otimes I_\cD)\widetilde V,\  i\in \{1,\ldots,n\}.
$
Indeed, we have 
$$
\left< \widetilde VT_i^*h, e_\beta\otimes h'\right>=\sqrt{b_\beta}\left<DT_\beta^* T_i^*h, h'\right>,\qquad h,h'\in \cH, \beta\in \FF_n^+,
$$
and 
\begin{equation*}
\begin{split}
\left< (W_i^*\otimes I_\cD)\widetilde Vh, e_\beta\otimes h'\right>&= \left< \widetilde Vh, (W_i e_\beta\otimes h'\right>\\
&=
\left<\sqrt{b_{g_i\beta}}e_{g_i\beta}\otimes DT_{g_i \beta}^*h, \frac{\sqrt{b_\beta}}{\sqrt{b_{g_i\beta}}}
e_{g_i\beta}\otimes h'\right>
=\sqrt{b_\beta}\left<DT_{g_i\beta}h, h'\right>
\end{split}
\end{equation*}
for any $h,h'\in \cH$, $\beta\in \FF_n^+$, which proves our assertion.
Consequently, item (i) holds. On the other hand, since 
$\widetilde V \widetilde V^*=
\sum_{\alpha\in \FF_n^+} b_\alpha T_\alpha D(T) T_\alpha^*,
$
it is clear that the last part of the theorem follows.
 The proof is completed. 
\end{proof}

The questions that  we would like to answer are the following. Under what conditions on the weights  $\{b_\alpha\}_{\alpha\in \FF_n^+}$,  is the associated $n$-tuple 
${\bf W}=(W_1,\ldots, W_n)$   of weighted left creation operators on the full Fock space $F^2(H_n)$   a
{\it universal model } for the $n$-tuples  of operators in a noncommutative set  $\cD(\cH)\subset B(\cH)^n$ which contains a ball $[B(\cH)^n]_\epsilon$  for some $\epsilon>0$ $?$  What is the maximal noncommutative set $\cD(\cH)$ so that each $n$-tuple $T\in \cD(\cH)$ has ${\bf W}$ as universal model $?$

According to Theorem \ref{intertwining}, we need to impose  appropriate conditions on  $\{b_\alpha\}_{\alpha\in \FF_n^+}$ so that, for each $T\in [B(\cH)^n]_\epsilon$,  there is a positive operator $D(T)\in B(\cH)$ such that
$$
\sum_{\alpha\in \FF_n^+} b_\alpha T_\alpha D(T) T_\alpha^*=I_\cH,
$$
where the convergence is  in the weak operator topology.
Note that in the scalar case when $\cH=\CC$ and  $T=\lambda:=(\lambda_1,\ldots, \lambda_n)\in [\CC^n]_\epsilon$, we should have
\begin{equation}\label{scalar}
\left(\sum_{\alpha\in \FF_n^+} b_\alpha |\lambda_\alpha|^2\right)D(\lambda)=1.
\end{equation}
Consequently, the series  $\sum_{\alpha\in \FF_n^+} b_\alpha |\lambda_\alpha|^2$ needs to be convergent for any $\lambda \in [\CC^n]_\epsilon$ and $D(\lambda)$ must be  its inverse.
This leads us to impose on the coefficients $\{b_\alpha\}_{\alpha\in \FF_n^+}$ the condition that 
$g:=\sum_{k=0}^\infty\sum_{|\alpha|=k} b_\alpha Z_\alpha$  be a free holomorphic function on  $[B(\cH)^n]_\epsilon$. According to   Lemma \ref{lem1}, the later condition implies  that  $g^{-1}=1+\sum_{ |\gamma|\geq 1} a_\gamma
Z_\gamma$  is also a free holomorphic function in a neighborhood of the origin.
In order to match the scalar  relation  \eqref{scalar} in the noncommutative setting, it seems reasonable to define 
$D(T):=\sum_{p=0}^\infty\sum_{ |\alpha|= p} a_\alpha T_\alpha T_\alpha^*$ as long as the series is convergent in the weak operator topology,  $D(T)\geq 0$,  and 
$$
\sum_{\alpha\in \FF_n^+} b_\alpha T_\alpha D(T)T_\alpha^*=I_\cH,
$$
where the convergence is  in the weak operator topology.   We  prove later (see Proposition \ref{lemm2}) that  all the  conditions above, concerning the operator  $D(T)$,  hold true for any $T\in [B(\cH)^n]_\epsilon$ and some   
$\epsilon>0$.
On the other hand, since Theorem \ref{intertwining} holds when $T$ is equal to the $n$-tuple ${\bf W}=(W_1,\ldots, W_n)$, one expects  to require that the above conditions on the existence of $D(T)$ be satisfied when  $T$ is replaced by ${\bf W}$.

Summing up these considerations,  we are led to  consider the following conditions on the weights  $\{b_\alpha\}_{\alpha\in \FF_n^+}$ and on the associated formal power series  $g:= 1+\sum_{ |\alpha|\geq 1} b_\alpha
Z_\alpha$.
\begin{enumerate}
\item[(a)] The collection of weights $\{b_\alpha\}_{\alpha\in \FF_n^+}$  consists  of strictly positive  real numbers such that
\begin{equation*}
b_{g_0}=1\quad \text{ and }  \quad
\sup_{\alpha\in \FF_n^+} \frac{b_\alpha}{b_{g_i \alpha}}<\infty,\qquad i\in \{1,\ldots, n\}.
\end{equation*}
\item[(b)] The formal power series $g:= 1+\sum_{ |\alpha|\geq 1} b_\alpha
Z_\alpha$  is  a free holomorphic function in a neighborhood of the origin, which is  equivalent  to the relation
$
\limsup_{k\to\infty} \left( \sum_{|\alpha|=k}
b_\alpha^2\right)^{1/2k}<\infty. 
$    

\item[(c)]  The condition 
\begin{equation} \label{ubound}
\sup_{k\in \NN}\left\|\sum_{p=0}^k\sum_{ |\alpha|= p} a_\alpha W_\alpha W_\alpha^*\right\|<\infty
\end{equation}
holds, where  $g^{-1}=1+\sum_{|\alpha|\geq 1} a_\alpha Z_\alpha $ is the inverse of $g$ and ${\bf W}=(W_1,\ldots, W_n)$ is the $n$-tuple of   the  weighted left creation operators associated with $\{b_\alpha\}_{\alpha\in \FF_n^+}$.
\end{enumerate}

 Note that if $g^{-1}$ is a noncommutative polynomial the condition (c)   is automatically satisfied.
 In what follows, we need several definitions.
\begin{definition}
A free holomorphic function  $g= 1+\sum_{ |\alpha|\geq 1} b_\alpha
Z_\alpha$ satisfying the conditions $(a), (b)$, and $(c)$ is  said to be  {\it admissible} for  operator model theory, and the $n$-tuple ${\bf W}=(W_1,\ldots, W_n)$ of weighted left creation operators is called universal model  associated with $g$.
\end{definition}

\begin{definition}
 Let  $g= 1+\sum_{ |\alpha|\geq 1} b_\alpha
Z_\alpha$ be  a free holomorphic function in a neighborhood of the origin, with $b_\alpha>0$, if $|\alpha|\geq 1$, and with   inverse  $g^{-1}=1+\sum_{ |\gamma|\geq 1} a_\gamma
Z_\gamma$.
The noncommutative   set  $\cD_{g^{-1}}(\cH)$ is defined as the set  of all $n$-tuples  of operators
$X=(X_1,\ldots, X_n) \in B(\cH)^n$ satisfying the following conditions.
\begin{enumerate}
\item[(i)]  $\Delta_{g^{-1}}(X,X^*)\geq 0$, where
$
\Delta_{g^{-1}}(X,X^*):= \sum_{p=0}^\infty\sum_{ |\alpha|= p} a_\alpha  X_\alpha X_\alpha^*
$
is convergent in the   weak operator topology. 
\item[(ii)]  
$\sum_{\alpha \in \FF_n^+} b_\alpha X_\alpha \Delta_{g^{-1}}(X,X^*) X_\alpha^*\leq I$, where the convergence is in the weak operator topology.
\end{enumerate}
\end{definition}

\begin{definition} Let  $X\in \cD_{g^{-1}}(\cH)$.
\begin{enumerate}
\item[(i)]  $X$  is said to be a {\it Cuntz tuple} in 
$\cD_{g^{-1}}(\cH)$ if  $\Delta_{g^{-1}}(X,X^*)=0$.
\item[(ii)]  $X$ is said to be a {\it pure}  element in  
$\cD_{g^{-1}}(\cH)$,
if 
$
\sum_{\alpha \in \FF_n^+} b_\alpha X_\alpha \Delta_{g^{-1}}(X,X^*) X_\alpha^*=I.
$
 \item[(iii)]  
 The {\it pure part} of   $\cD_{g^{-1}}(\cH)$ is defined by setting
$$
\cD^{pure}_{g^{-1}}(\cH):=\{X\in \cD_{g^{-1}}(\cH): X  \text{ is pure}\}.
$$
\end{enumerate}
\end{definition}

\begin{definition}  
We say that $X=(X_1,\ldots, X_n)\in B(\cH)^n$ is  a {\it radial}  $n$-tuple with respect to the set $\cD_{g^{-1}}(\cH)$ if 
  there is  $\delta\in (0,1)$ such  that $rX:=(rX_1,\ldots, rX_n)\in \cD_{g^{-1}}(\cH)$ for every $r\in (\delta, 1)$.  If, in addition, $rX$ is pure for every $r\in (\delta, 1)$, we say that $X$ is  a radially pure $n$-tuple with respect to  $\cD_{g^{-1}}(\cH)$. The noncommutative set $ \cD_{g^{-1}}(\cH)$ is called  {\it radially pure} domain if   any $X\in \cD_{g^{-1}}(\cH)$ is  a radially pure $n$-tuple.
  \end{definition}

Let $g$ be an admissible  free holomorphic function.
In what follows, we will show that
\begin{enumerate}
\item[(i)] the conditions  (a) and (c) are necessary  to guarantee that ${\bf W}=(W_1,\ldots, W_n)$ is a pure $n$-tuple  of bounded operators in the noncommutative set $\cD_{g^{-1}}(F^2(H_n))$;
\item[(ii)] the condition (b) ensures that there is  $\epsilon>0$ such that the noncommutative  ball
$[B(\cH)^n]_\epsilon$ is included in $\cD_{g^{-1}}(\cH)$ for any Hilbert space $\cH$.
\end{enumerate}

We remark that the 
 relation  
$fg=gf=1$, in Lemma \ref{lem1},
implies  
\begin{equation}\label{sum1}
b_\gamma+\sum_{\beta \alpha=\gamma, |\beta|\geq 1}a_\beta b_\alpha=0
\quad \text{ if } \ |\gamma|\geq 1
\end{equation}
and 
 \begin{equation}\label{sum2}
b_\gamma+\sum_{ \alpha\beta=\gamma, |\beta|\geq 1}a_\beta b_\alpha=0
 \quad \text{ if } \ |\gamma|\geq 1.
\end{equation}

\begin{proposition}   \label{W}  Let $g= 1+\sum_{ |\alpha|\geq 1} b_\alpha
Z_\alpha$ be an admissible free holomorphic function with inverse $g^{-1}=1+\sum_{ |\gamma|\geq 1} a_\gamma
Z_\gamma$ and let ${\bf W}=(W_1,\ldots, W_n)$ be  the coresponding  weighted left creation operators associated with $g$.
Then 
$$
\Delta_{g^{-1}}({\bf W},{\bf W}^*):=\sum_{p=0}^\infty\sum_{ |\alpha|= p} a_\alpha  W_\alpha W_\alpha^*
$$
converges in the  strong operator topology and  $\Delta_{g^{-1}}({\bf W},{\bf W}^*)=P_{\CC 1}$, the orthogonal projection of $F^2(H_n)$ onto the constants $\CC 1$. Moreover, we have
$$
\sum_{\alpha \in \FF_n^+} b_\alpha W_\alpha \Delta_{g^{-1}}({\bf W},{\bf W}^*) W_\alpha^*=I,
$$
where the convergence is in the strong operator topology.
\end{proposition}
\begin{proof}  First, note that the operators $W_1,\ldots, W_n$  are bounded if and only if  the condition (a) holds. 
Since 
$W_i e_\alpha=\sqrt{\frac {{b_\alpha}}{ {b_{g_i \alpha}}}}
e_{g_i \alpha}$,  $\alpha\in \FF_n^+$,
we deduce that  
\begin{equation}\label{WbWb}
W_\beta e_\gamma=
\frac {\sqrt{b_\gamma}}{\sqrt{b_{\beta \gamma}}}
e_{\beta \gamma} \quad
\text{ and }\quad
W_\beta^* e_\alpha =\begin{cases}
\frac {\sqrt{b_\gamma}}{\sqrt{b_{\alpha}}}e_\gamma& \text{ if }
\alpha=\beta\gamma \\
0& \text{ otherwise }
\end{cases}
\end{equation}
 for any $\alpha, \beta \in \FF_n^+$.
 Using  relation \eqref{WbWb}, we deduce that
$$
W_\beta W_\beta^* e_\alpha =\begin{cases}
\frac {{b_\gamma}}{{b_{\alpha}}} e_\alpha & \text{ if } \alpha=\beta\gamma\\
0& \text{ otherwise, }
\end{cases}
$$
which implies
$$\left(I+\sum\limits_{1\leq |\beta|\leq k} a_\beta W_\beta W_\beta^*\right)
 e_\alpha= \begin{cases} \left( 1+\sum\limits_{{\beta\gamma=\alpha}\atop { 1\leq |\beta|\leq k}}
\frac{a_\beta b_\gamma}{b_\alpha}\right)  e_\alpha,& \quad\text{ if }  |\alpha|\geq 1 \\
1,&\quad \text{ if } \alpha=g_0.
\end{cases}
$$
  Due to relation \eqref{sum1}, if $ k\geq |\alpha|\geq 1$, we have $1+\sum_{\beta\gamma=\alpha, 1\leq |\beta|\leq k}
\frac{a_\beta b_\gamma}{b_\alpha}=0$.
Consequently,   we deduce that
  $$\lim_{k\to \infty}\left(I+\sum\limits_{1\leq |\beta|\leq k} a_\beta W_\beta W_\beta^*\right)
 e_\alpha=0,\qquad  \alpha \in\FF_n^+.
 $$
  Now, condition (c)   implies 
$\Delta_{g^{-1}}(W,W^*)=P_{\CC 1}$.
To prove the last part of the proposition, note that  
\begin{equation*} 
P_{\CC 1} W_\beta^* e_\alpha =\begin{cases}
\frac {1}{\sqrt{b_{\beta}}}   & \text{ if } \alpha=\beta\\
0& \text{ otherwise}
\end{cases}
\end{equation*}
and the operator  $  b_\beta W_\beta P_{\CC 1} W_\beta^*$ is equal to the orthogonal projection of $F^2(H_n)$ on the subspace $\CC e_\beta$. Now, it is clear that
$$
\sum_{\alpha \in \FF_n^+} b_\alpha W_\alpha \Delta_{g^{-1}}({\bf W},{\bf W}^*) W_\alpha^*=
\sum_{\alpha \in \FF_n^+} b_\alpha W_\alpha  P_{\CC 1} W_\alpha^*
=I,
$$
where the convergence is in the strong operator topology.
 The proof is complete.   
\end{proof}

\begin{proposition}  \label{abb} The condition  \eqref{ubound}    is equivalent to the condition
$$
\sup_{ {N\in \NN},{\alpha\in \FF_n^+}}
\left| \sum_{{0\leq |\beta|\leq N}\atop{\beta\gamma=\alpha}}\frac{a_\beta b_\gamma}{b_\alpha} 
\right|<\infty.
$$
\end{proposition}
\begin{proof} According to the proof of Proposition \ref{W}, for each $N\in \NN$, the operator 
$\sum\limits_{0\leq |\beta|\leq N} a_\beta W_\beta W_\beta^*$ is diagonal and
$$
\left\|\sum\limits_{0\leq |\beta|\leq N} a_\beta W_\beta W_\beta^*\right\|
=\sup_{\alpha\in \FF_n^+} 
\left|  \sum\limits_{{\beta\gamma=\alpha}\atop { 0\leq |\beta|\leq N}}
\frac{a_\beta b_\gamma}{b_\alpha}
 \right|.
$$
The proof is complete.
\end{proof}

Following Theorem \ref{intertwining}, we  are led to define the
{\it noncommutative Berezin kernel} associated with an arbitrary  $n$-tuple $T=(T_1,\ldots,
T_n)\in \cD_{g^{-1}}(\cH)$ to be  the operator $K_{g,T}:\cH\to
F^2(H_n)\otimes \overline{\Delta_{g^{-1}} (T,T^*) (\cH)}$  given by
\begin{equation*}
 K_{g,T}h:=\sum_{\alpha\in \FF_n^+} \sqrt{b_\alpha}
e_\alpha\otimes \Delta_{g^{-1}}(T,T^*)^{1/2} T_\alpha^* h,\quad h\in \cH.
\end{equation*}

An important consequence of  Theorem \ref{intertwining} is the following.
\begin{corollary} \label{Berez}
If $g= 1+\sum_{ |\alpha|\geq 1} b_\alpha
Z_\alpha$ is  a free holomorphic function in a neighborhood of the origin, with $b_\alpha>0$, if $|\alpha|\geq 1$, and with   inverse  $g^{-1}=1+\sum_{ |\gamma|\geq 1} a_\gamma
Z_\gamma$,  and  $T=(T_1,\ldots,
T_n)\in \cD_{g^{-1}}(\cH)$ , then
$$
 K_{g,T} T_i^*=(W_i^*\otimes I_\cD)K_{g,T},\qquad i\in \{1,\ldots, n\},
 $$
 where   the noncommutative Berezin kernel  $K_{g,T} $ is a contraction.
\end{corollary}

\begin{proposition} \label{lemm2} Let  $g= 1+\sum_{ |\alpha|\geq 1} b_\alpha
Z_\alpha$ be  a free holomorphic function in a neighborhood of the origin, with $b_\alpha>0$, if $|\alpha|\geq 1$, and with   inverse  $g^{-1}=1+\sum_{ |\gamma|\geq 1} a_\gamma
Z_\gamma$. If  $X=(X_1,\ldots, X_n)\in B(\cH)^n$ is such that
$$\sum_{p=1}^\infty\sum_{ |\alpha|= p} |a_\alpha|  X_\alpha X_\alpha^* \leq  c I
$$
for some $c\in (0,1)$, then $X\in \cD_{g^{-1}}(\cH)$  is a pure   element.  Moreover,  there exists 
  $\epsilon>0$ such that   $[B(\cH)^n]_\epsilon\subset \cD_{g^{-1}}(\cH)$ and  any $n$-tuple in
  $[B(\cH)^n]_\epsilon$ is  pure and radial in $\cD_{g^{-1}}(\cH)$. 

\end{proposition}
\begin{proof}
First, note that 
$$
\Delta_{g^{-1}}(X,X^*):= I+\sum_{p=1}^\infty\sum_{ |\alpha|= p} a_\alpha  X_\alpha X_\alpha^*\geq I-\sum_{p=1}^\infty\sum_{ |\alpha|= p} |a_\alpha|  X_\alpha X_\alpha^*\geq 0.
$$
 Set   $\Phi_{X}(Y):= \sum_{p=1}^\infty\sum_{ |\alpha|= p} a_\alpha X_\alpha YX_\alpha^*$  for $Y\in B(\cH)$ and $\hat\Phi_{X}(Y):= \sum_{p=1}^\infty\sum_{ |\alpha|= p} |a_\alpha| X_\alpha YX_\alpha^*$.   If $Y\geq 0$, then we have
 \begin{equation*}
 \begin{split}
 |\left<\Phi_X^k(Y)h,h\right>|&\leq \left |\left<\hat\Phi_X^k(Y)h,h\right>\right|
 \leq \|Y\| \left |\left<\hat\Phi_X^k(I)h,h\right>\right|\leq c^k\|Y\|\|h\|^2
 \end{split}
 \end{equation*}
 for any $h\in \cH$
 and, hence, $\sum_{k=1}^\infty|\left< \Phi_X^k(Y)h,h\right>|<\infty$.
As a consequence,  for each $h\in \cH$, we can rearrange   the convergent series  of real numbers
$\sum_{k=1}^\infty \left<(-1)^k\Phi_{X}^k(\Delta_{g^{-1}}(X,X^*))h,h\right>$ and show that it is equal to
$$
\sum_{m=1}^\infty
\sum_{|\beta|=m}\left<\left(
\sum_{j=1}^{|\beta|}\sum_{{\gamma_1\cdots
\gamma_j=\beta}\atop{|\gamma_1|\geq 1,\ldots, |\gamma_j|\geq 1}}
(-1)^ja_{\gamma_1}\cdots a_{\gamma_j}\right) X_{\gamma_1\cdots \gamma_j}
\Delta_{g^{-1}}(X,X^*) X_{\gamma_1\cdots \gamma_j}^*h,h\right>.
$$
Now, 
 using Lemma \ref{lem1} and the fact that
 $$
 \left<\Delta_{g^{-1}}(X,X^*) h,h\right>=\|h\|^2+\left<\Phi_X(I)h,h\right>,
 $$
  we deduce that
\begin{equation*}
\begin{split}
&\left< \sum_{\beta\in \FF_n^+} b_\beta X_\beta \Delta_{g^{-1}}(X,X^*)X_\beta^*h,
h\right>
 =
\left<\Delta_{g^{-1}}(X,X^*) h,h\right>+
\sum_{m=1}^\infty \sum_{|\beta|=m}\left<
b_\beta X_\beta \Delta_{g^{-1}}(X,X^*)X_\beta^*h, h\right>\\
& =\left<\Delta_{g^{-1}}(X,X^*) h,h\right>+ \sum_{m=1}^\infty
\sum_{|\beta|=m}\left<\left(
\sum_{j=1}^{|\beta|}\sum_{{\gamma_1\cdots
\gamma_j=\beta}\atop{|\gamma_1|\geq 1,\ldots, |\gamma_j|\geq 1}}
(-1)^ja_{\gamma_1}\cdots a_{\gamma_j}\right) X_{\gamma_1\cdots \gamma_j}
\Delta_{g^{-1}}(X,X^*) X_{\gamma_1\cdots \gamma_j}^*h,h\right>\\
& =\left<\Delta_{g^{-1}}(X,X^*) h,h\right>+
\sum_{k=1}^\infty \left<(-1)^k\Phi_{X}^k(\Delta_{g^{-1}}(X,X^*))h,h\right>\\
& =\|h\|^2 -\lim_{m\to\infty} \left<(-1)^m\Phi_{X}^m(I)h,h\right>
\end{split}
\end{equation*}
for any $h\in \cH$.  Consequently,   since $\lim_{m\to\infty} \left<(-1)^m\Phi_{X}^m(I)h,h\right>=0$,  we obtain
$$
K_{g,X}^* K_{g,X}=\sum_{\beta\in \FF_n^+} b_\beta X_\beta \Delta_{g^{-1}}(X,X^*)X_\beta^*= I.  
$$
 To prove the last part of the proposition,
 let $\rho$ be the radius of convergence of the free holomorphic function $g^{-1}$, i.e.
$
\frac{1}{\rho}:= \limsup_{k\to\infty} \left( \sum_{|\alpha|=k}
|a_\alpha|^2\right)^{1/2k}.
$
Due to Lemma \ref{lem1}, we have $\rho>0$.
If $0<t<\frac{\rho}{\sqrt{n}}$, then the $n$-tuple $(t I_\cH,\ldots, t I_\cH)$ is in the ball $[B(\cH)^n]_\rho$, which  implies 
$
M:=\sum_{k=1}^\infty \sum_{|\alpha|=k} |a_\alpha|t^k<\infty.
$
Let $0<\omega <1$ be such that $\omega M<c<1$ and set $\epsilon:=\sqrt{\omega t}$. Note that, if $X\in [B(\cH)^n]_\epsilon$, then
\begin{equation*} 
 \sum_{p=1}^\infty\sum_{ |\alpha|= p} |a_\alpha|  \|X_\alpha X_\alpha^*\|\leq 
 \sum_{k=1}^\infty \sum_{|\alpha|=k} |a_\alpha| \epsilon^{2k} \leq  \omega \sum_{k=1}^\infty \sum_{|\alpha|=k} |a_\alpha|t^k =\omega M<c.
\end{equation*}
Applying the first part of the proposition, we deduce that
   $[B(\cH)^n]_\epsilon\subset \cD_{g^{-1}}(\cH)$ and  any $n$-tuple in
  $[B(\cH)^n]_\epsilon$ is  pure and radial with respect to $\cD_{g^{-1}}(\cH)$.
The proof is complete.
\end{proof}

\begin{theorem} \label{model}  Let $g$ be an admissible  free holomorphic function and let $T=(T_1,\ldots, T_n)\in B(\cH)^n$. Then  there is a Hilbert spaces $\cD$ such that 
$$
T_i^*=(W_i^*\otimes I_\cD)|_{\cH},\qquad i\in \{1,\ldots, n\},
$$
where $\cH$ is  identified with a coinvariant subspace  for the operators $W_1\otimes I_\cD,\ldots, W_n\otimes I_\cD$, if and only  if $T$  is    a  pure $n$-tuple in the noncommutative set  $\cD_{g^{-1}}(\cH)$.
In this case, we have
 $$
 K_{g,T} T_i^*=(W_i^*\otimes I_\cD)K_{g,T},\qquad i\in \{1,\ldots, n\},
 $$
 and the noncommutative Berezin kernel $K_{g,T}$ is an isometry.
\end{theorem}
\begin{proof} To prove the direct implication, assume that 
$T_i^*=(W_i^*\otimes I_\cD)|_{\cH}$ for  every $ i\in \{1,\ldots, n\}$. Then we have
$$
\sup_{k\in \NN}\left\|\sum_{p=0}^k\sum_{ |\alpha|= p} a_\alpha T_\alpha T_\alpha^*\right\|\leq
\sup_{k\in \NN}\left\|\sum_{p=0}^k\sum_{ |\alpha|= p} a_\alpha W_\alpha W_\alpha^*\right\|<\infty.
$$
Due to Proposition \ref{W},
$$
\Delta_{g^{-1}}({\bf W},{\bf W}^*):=\sum_{p=0}^\infty\sum_{ |\alpha|= p} a_\alpha  W_\alpha W_\alpha^*
$$
converges in the  strong operator topology and  $\Delta_{g^{-1}}(W,W^*)=P_{\CC 1}$, the orthogonal projection of $F^2(H_n)$ onto the constants $\CC1$.
Consequently,
$$
\Delta_{g^{-1}}(T,T^*):=\sum_{p=0}^\infty\sum_{ |\alpha|= p} a_\alpha  T_\alpha T_\alpha^*=P_\cH\sum_{p=0}^\infty\sum_{ |\alpha|= p} a_\alpha  W_\alpha W_\alpha^*|_\cH
$$
converges in the  strong operator topology to a positive operator. On the other hand, Proposition \ref{W} shows that
$$
\sum_{\alpha \in \FF_n^+} b_\alpha T_\alpha \Delta_{g^{-1}}(T,T^*) T_\alpha^*
=
P_\cH\sum_{\alpha \in \FF_n^+} b_\alpha W_\alpha \Delta_{g^{-1}}({\bf W},{\bf W}^*) W_\alpha^*|_\cH=I_\cH,
$$
where the convergence is in the strong operator topology. Hence, $T$ is a pure $n$-tuple in 
$\cD_{g^{-1}}(\cH)$.
 
 Conversely, assume that $T$  is    a  pure $n$-tuple in the noncommutative set  $\cD_{g^{-1}}(\cH)$.
 In this case,  $\Delta_{g^{-1}}(T,T^*)\geq 0$ and  $\sum_{\alpha \in \FF_n^+} b_\alpha T_\alpha \Delta_{g^{-1}}(T,T^*) T_\alpha^*=I_\cH$, where the convergence is in the weak operator topology. Applying Theorem
 \ref{intertwining}, we deduce that the Berezin kernel  $K_{g,T}$ is an isometry and 
 $
 K_{g,T} T_i^*=(W_i^*\otimes I_\cD)K_{g,T}$,  $i\in \{1,\ldots, n\}$,
 where $\cD:=\overline{\Delta_{g^{-1}}(T,T^*)\cH}$.
Identifying $\cH$ with $K_{g,T}\cH$, we complete the proof.
\end{proof}
We refer the reader to \cite{Pa-book}, \cite{Pi-book} for basic facts concerning completely contractive (resp. positive) maps.
\begin{theorem}\label{pure}
If   $g$ is an admissible  free holomorphic function and $T\in \overline{\cD^{pure}_{g^{-1}}(\cH)}$, then there is a unital completely contractive linear map
$$
\Psi_T:\cS:=\overline{span}\{W_\alpha W_\beta^*: \ \alpha,\beta\in \FF_n^+\}\to B(\cH),
$$
such that $\Psi_T (W_\alpha W_\beta^*)=T_\alpha T_\beta^*$  for any $\alpha,\beta\in \FF_n^+$.
If $T\in B(\cH)^n$ is a radially pure  $n$-tuple with respect to $\cD_{g^{-1}}(\cH)$, then
$$
\Psi_T(A)=\lim_{r\to 1}  K_{g,rT}^*(A\otimes I) K_{g,rT},\qquad A\in \cS,
$$
where the limit is in the operator norm.
\end{theorem}
\begin{proof}
Fix  $T\in \overline{\cD^{pure}_{g^{-1}}(\cH)}$ and let $\{T^{(p)}\}_{p\in \NN}\subset \cD^{pure}_{g^{-1}}(\cH)$ be a sequence such that $\|T-T^{(p)}\|\to 0$ as $p\to \infty$. Due to Theorem \ref{model}, we have
$$
T^{(p)}_\alpha T^{(p)*}_\beta=K_{g,T^{(p)}}^*(W_\alpha W_\beta^*\otimes I_\cH)K_{g,T^{(p)}},\qquad \alpha,\beta\in \FF_n^+.
$$
Since  $K_{g,T^{(p)}}$ is an isometry, we deduce that
$$
\left\|\sum_{|\alpha|,|\beta|\leq m}d_{\alpha,\beta} T^{(p)}_\alpha T^{(p)*}_\beta\right\|
\leq 
\left\|\sum_{|\alpha|,|\beta|\leq m}d_{\alpha,\beta}  W_\alpha W_\beta^*\right\|,\qquad m\in \NN, d_{\alpha,\beta}\in \CC,
$$
which implies 
\begin{equation}
\label{vN}
\left\|\sum_{|\alpha|,|\beta|\leq m}d_{\alpha,\beta} T_\alpha T^*_\beta\right\|
\leq 
\left\|\sum_{|\alpha|,|\beta|\leq m}d_{\alpha,\beta}  W_\alpha W_\beta^*\right\|,\qquad m\in \NN, d_{\alpha,\beta}\in \CC.
\end{equation}
Let $A\in \cS$ and consider a sequence $p_s({\bf W},{\bf W}^*):=\sum_{|\alpha|,|\beta|\leq m_s}d^{(s)}_{\alpha,\beta}  W_\alpha W_\beta^*$ such that $\|A-p_s({\bf W},{\bf W}^*)\|\to 0$  as $s\to \infty$.
Due to relation \eqref{vN},  the operator $\Phi_T(A):=\lim_{s\to \infty} p_s(T,T^*)$ is well-defined and 
$$
\|\Phi_T(A)\|=\lim_{s\to \infty} \|p_s(T,T^*)\|\leq \lim_{s\to \infty} \|p_s({\bf W},{\bf W}^*))\|=\|A\|.
$$
Similarly, one can prove that
$
\|[\Phi_T(A_{ij})]_{q\times q}\|\leq  \|[A_{ij}]_{q\times q}\|
$
for any $q\times q$ matrix $[A_{ij}]_{q\times q}$.

Now, let $T\in B(\cH)^n$  be  a radially pure  $n$-tuple with respect to $\cD_{g^{-1}}(\cH)$.  Then there is $\delta\in (0,1)$ such that $rT\in  \cD^{pure}_{g^{-1}}(\cH)$ for $r\in  (\delta,1)$. Using the first part of the proof, we deduce that
$\Psi_{rT}(A):=\lim_{s\to \infty} p_s(rT,rT^*)$ exists in norm, 
$$
\|\Psi_{rT}(A)-p_s(rT,rT^*)\|\leq \|A-p_s({\bf W},{\bf W}^*)\|,\qquad r\in (\delta,1),
$$
and 
$$
p_s(rT,rT^*)=K_{g, rT}^*(A\otimes I_\cH) K_{g, rT},\qquad r\in (\delta,1).
$$
Due to the fact that    $\|A-p_s(W,W^*)\|\to 0$,    $\|\Psi_T(A)-p_s(T,T^*)\|\to 0$, and $\|\Psi_{rT}(A)-p_s(rT,rT^*)\|\to 0$ uniformly with respect to  $r\in (\delta, 1)$,  it is clear that relation
\begin{equation*}
\begin{split}
\|\Psi_T(A)-K_{g, rT}^*(A\otimes I_\cH) K_{g, rT}\|&\leq 
\|\Psi_T(A)-p_s(T,T^*)\|\\
\qquad &+\|p_s(T,T^*)-p_s(rT,rT^*)\|+\|p_s(rT,rT^*)-\Psi_{rT}(A)\|
\end{split}
\end{equation*}
can be used to complete the proof.
\end{proof}
We remark  the map $\Psi_T$ in Theorem \ref{pure} is also completely positive  due to the fact that $\cS$ is an operator system.

We  introduce the {\it noncommutative domain algebra} $\cA(g)$ as the norm-closed non-self-adjoint algebra generated by $W_1,\ldots, W_n$ and the the identity. The following result provides a von Neumann \cite{vN} type inequality for the noncommutative domain $\overline{\cD^{pure}_{g^{-1}}(\cH)}$.

\begin{corollary}  \label{vN-disc}
If  $g$ is an admissible  free holomorphic function  and $T=(T_1,\ldots, T_n)\in \overline{\cD^{pure}_{g^{-1}}(\cH)}$, then
\begin{equation*}
\left\|\sum_{|\alpha|,|\beta|\leq m}d_{\alpha,\beta} T_\alpha T^*_\beta\right\|
\leq 
\left\|\sum_{|\alpha|,|\beta|\leq m}d_{\alpha,\beta}  W_\alpha W_\beta^*\right\|,\qquad m\in \NN, d_{\alpha,\beta}\in \CC.
\end{equation*}
 Moreover, there is a unital completely contractive homomorphism 
$$
\Psi_T:\cA(g)\to B(\cH),\quad \Psi_T(p(W)):=p(T),
$$
for any polynomial $p$  in $n$ noncommutative indeterminates.
\end{corollary}

\section{Noncommutative Hardy algebras and functional calculus}

In this section, we introduce the noncommutative Hardy algebras $F^\infty(g)$ associated with formal power series $g$ with strictly positive coefficients. We present some of their basic properties  and provide a $w^*$-continuous $F^\infty(g)$-functional calculus for pure (resp. completely non-coisometric) $n$-tuples of operators in the noncommutative domain $\cD_{g^{-1}}(\cH)$ associated with an admissible free holomorphic function $g$.

 Let $g= 1+\sum_{ |\alpha|\geq 1} b_\alpha
Z_\alpha$ be a formal power series such  that  the coefficients $\{b_\alpha\}_{\alpha\in \FF_n^+}$  are strictly positive  real numbers such that
 $\sup_{\alpha\in \FF_n^+} \frac{b_\alpha}{b_{ g_i\alpha}}<\infty$,
$ i\in \{1,\ldots, n\},
$ 
and let ${\bf W}=(W_1,\ldots, W_n)$ be the associated weighted left creation operators.
Let $\varphi({\bf W})=\sum\limits_{\beta\in \FF_n^+} c_\beta
W_\beta $ be a formal series such that,  
for each $\gamma\in \FF_n^+$,   
$
\sum_{\beta\in \FF_n^+}|c_\beta|^2 \frac{b_\gamma}{b_{\beta\gamma}}
 <\infty$
 and 
 $
 \sup_{p\in\cP, \|p\|\leq 1} \left\|\sum\limits_{\beta\in \FF_n^+}
c_\beta W_\beta (p)\right\|<\infty,
$
where $\cP$  is the linear span of $\{e_\alpha\}_{\alpha\in \FF_n^+}$.
In this case,  there is a unique bounded linear operator acting on the Fock space $F^2(H_n)$, which
we denote by $\varphi({\bf W})$, such that
$
\varphi({\bf W})p=\sum\limits_{\beta\in \FF_n^+} c_\beta
W_\beta (p)\quad \text{ for any } \ p\in \cP.
$
The set of all operators $\varphi({\bf W})\in B(F^2(H_n))$
satisfying the above-mentioned properties is denoted by
$F^\infty(g)$.  
One can prove that $F^\infty(g)$  is a Banach algebra, which we
call noncommutative Hardy algebra associated with  the formal power series $g$.

Let $g= 1+\sum_{ |\alpha|\geq 1} b_\alpha
Z_\alpha$ be a formal power series such  that   the coefficients $\{b_\alpha\}_{\alpha\in \FF_n^+}$  are strictly positive  real numbers   and 
 $\sup_{\alpha\in \FF_n^+} \frac{b_\alpha}{b_{ \alpha g_i}}<\infty$,
$i\in \{1,\ldots, n\}.
$ 
Define the {\it weighted right creation operators}
$\Lambda_i:F^2(H_n)\to F^2(H_n)$ by setting
$\Lambda_i:= R_i G_i$, $i\in\{1,\ldots, n\}$,  where $R_1,\ldots, R_n$ are
 the right creation operators on the full Fock space $F^2(H_n)$ and
 each  diagonal operator $G_i$  is defined by
$$
G_ie_\alpha=\sqrt{\frac{b_\alpha}{b_{ \alpha g_i}}} e_\alpha,\quad
 \alpha\in \FF_n^+.
$$
In this case, we have
\begin{equation}\label{WbWb-r}
\Lambda_\beta e_\gamma=
\frac {\sqrt{b_\gamma}}{\sqrt{b_{ \gamma \tilde\beta}}}
e_{ \gamma \tilde \beta} \quad
\text{ and }\quad
\Lambda_\beta^* e_\alpha =\begin{cases}
\frac {\sqrt{b_\gamma}}{\sqrt{b_{\alpha}}}e_\gamma& \text{ if }
\alpha=\gamma \tilde \beta \\
0& \text{ otherwise }
\end{cases}
\end{equation}
 for any $\alpha, \beta \in \FF_n^+$, where $\widetilde\beta:=g_{i_p}\cdots g_{i_1}$  is the reverse of  $\beta=g_{i_1}\cdots g_{i_p}\in \FF_n^+$. We use the notation ${\bf \Lambda}:=(\Lambda_1,\ldots, \Lambda_n)$. 
 If we assume that $\tilde{g}:=1+\sum_{ |\gamma|\geq 1} b_{\widetilde\gamma} Z_\gamma$ is an admissible  free holomorphic function, one can use relation \eqref{sum2} to show
that \begin{equation*} \Delta_{\widetilde g^{-1}}({\bf \Lambda},{\bf \Lambda}^*) =P_{\CC1}
  \quad \text{ and } \quad 
\sum_{\beta\in \FF_n^+}b_{\tilde\beta} \Lambda_\beta
\Delta_{\widetilde g^{-1}}({\bf \Lambda},{\bf \Lambda}^*)   \Lambda_\beta^* =I,
\end{equation*}
where the convergence is in the strong operator topology. the proof is similar to  that of Proposition \ref{W}.
As a consequence,  one can easily see that
$(\Lambda_1,\ldots,
\Lambda_n)\in \cD_{{\widetilde g}^{-1}}(F^2(H_n))$ and
$
U^* \Lambda_i U=W_i^{(\widetilde g)}$,
for $i\in\{1,\ldots,n\}$,
where $(W_1^{(\widetilde g)},\ldots, W_n^{(\widetilde g)})$ is the universal model associated with $\widetilde g$ and 
$U\in B(F^2(H_n))$ is the unitary operator
defined by $U e_\alpha:= e_{\tilde \alpha}$, $\alpha\in \FF_n^+$.

 Using the weighted right creation
operators ${\bf \Lambda}=(\Lambda_1,\ldots, \Lambda_n)$
 associated with the formal power series $g$ with the properties mentioned above, we  can define    the corresponding
     noncommutative  Hardy algebra $R^\infty(g)$.
Indeed, if $g({\bf \Lambda})=\sum\limits_{\beta\in \FF_n^+} c_{\tilde\beta
}\Lambda_\beta $ is a formal sum  such that,
for each $\gamma\in \FF_n^+$,   
$
\sum_{\beta\in \FF_n^+}|c_\beta|^2 \frac{b_\gamma}{b_{\gamma\beta}}
 <\infty$
and 
$
\sup_{p\in\cP, \|p\|\leq 1} \left\|\sum\limits_{\beta\in \FF_n^+}
 c_{\tilde\beta} \Lambda_\beta (p)\right\|<\infty,
$
then there is a unique bounded operator on $F^2(H_n)$, which we
 denote by $g({\bf \Lambda})$, such that
$
g({\bf \Lambda})p=\sum\limits_{\beta\in \FF_n^+}
c_{\tilde\beta} \Lambda_\beta (p)$ for any  $ p\in \cP.
$
The set of all operators $g({\bf \Lambda})\in
B(F^2(H_n))$
 satisfying the above-mentioned properties is denoted by $R^\infty(g)$. Note that $R^\infty(g)$ is  a Banach algebra.

\begin{proposition}\label{tilde-f2}    Let $g= 1+\sum_{ |\alpha|\geq 1} b_\alpha
Z_\alpha$ be a formal power series such  that   the coefficients
 $\{b_\alpha\}_{\alpha\in \FF_n^+}$  are strictly positive  real numbers and such that
$$
\sup_{\alpha\in \FF_n^+} \frac{b_\alpha}{b_{g_i \alpha}}<\infty, \qquad  \sup_{\alpha\in \FF_n^+} \frac{b_\alpha}{b_{ \alpha g_i}}<\infty,
\qquad i\in \{1,\ldots, n\},
$$
and let ~${\bf W}=(W_1,\ldots, W_n)$  $($resp.
${\bf \Lambda}:=(\Lambda_1,\ldots, \Lambda_n))$ be the  associated weighted left
(resp. right) creation operators.  Then  
\begin{enumerate}
\item[(i)] $F^\infty(g)=\{ \Lambda_1,\ldots, \Lambda_n\}^\prime =R^\infty(g)^\prime$, where $'$ stands for the commutant,
\item[(ii)] $F^\infty(g)''=F^\infty(g)$ and
$R^\infty(g)''=R^\infty(g)$.
\end{enumerate}
\end{proposition}
\begin{proof}   First, note that the conditions  in the hypothesis are necessary and sufficient  to ensure that 
the operators  $W_i$ and $ \Lambda_i$ are bounded.
Since
$W_i\Lambda_j=\Lambda_j W_i$ for any
$i,j\in\{1,\ldots,n\}$, it is clear that $F^\infty(g)\subset\{ \Lambda_1,\ldots, \Lambda_n\}^\prime$. To prove the reverse inclusion, let $A\in
\{ \Lambda_1,\ldots, \Lambda_n\}^\prime$. Since $A(1)\in F^2(H_n)$, we have
$A(1)=\sum_{\beta\in \FF_n^+} c_{
\beta}\frac{1}{\sqrt{b_{ \beta}}} e_{ \beta}$ for some
coefficients $\{c_\beta\}_{\FF_n^+}\subset \CC$ with $\sum_{\beta\in\FF_n^+}
|c_\beta|^2 \frac{1}{b_\beta}<\infty$. On the other hand, since $
A\Lambda_i=\Lambda_iA$ for $i\in\{1,\ldots,n\}$, relations \eqref{WbWb}
and \eqref{WbWb-r} imply
\begin{equation*}
\begin{split}
Ae_\alpha &=\sqrt{b_\alpha}A\Lambda_{\tilde\alpha}(1)=\sqrt{b_\alpha}\Lambda_{\tilde\alpha}
A(1)
=\sum_{\beta\in \FF_n^+} c_{ \beta}
\frac{\sqrt{b_\alpha}}{\sqrt{b_{ \beta\alpha}}} e_{
 \beta\alpha}=\sum_{\beta\in \FF_n^+} c_{ \beta} W_\beta(e_\alpha).
\end{split}
\end{equation*}
Therefore,  for each $\alpha\in \FF_n^+$,   
$
\sum_{\beta\in \FF_n^+}|c_\beta|^2 \frac{b_\alpha}{b_{\beta\alpha}}
 <\infty$
and  $A(q)=\sum_{\beta\in \FF_n} c_{ \beta}
W_\beta(q)$ for any polynomial $q\in \cP$. Since $A$ is a bounded operator, we must have  $A\in F^\infty(g)$, which shows that $F^\infty(g)=\{ \Lambda_1,\ldots, \Lambda_n\}^\prime$. Similarly, one can show that $R^\infty(g)=\{ W_1,\ldots, W_n\}^\prime$.
The rest of the proof  is straightforward.
\end{proof}
 
\begin{remark} Each $\varphi\in F^\infty(g)$ has a unique   Fourier representation $\sum_{\alpha\in \FF_n^+} c_\alpha W_\alpha$.
\end{remark}

If $g= 1+\sum_{ |\alpha|\geq 1} b_\alpha
Z_\alpha$  is  a formal power series  satisfying the conditions of Proposition \ref{tilde-f2},
   one can  prove the following properties for the noncommutative Hardy algebra $F^\infty(g)$.   Since the proof  is similar to the corresponding result obtained by  Davidson and Pitts \cite{DP2}, in the particular case when $g^{-1}=1-Z_1-\cdots -Z_n$,  and the corresponding result from  \cite{Po-domains}, we omit it.
 
\begin{enumerate}
\item[(i)]The Hardy algebra $F^\infty(g)$ is inverse closed.
\item[(ii)] The only normal elements in $F^\infty(g)$ are the
scalars.
\item[(iii)] Every element  $A\in F^\infty(g)$
has
 its spectrum
$\sigma (A)\neq \{0\}$ and it is  injective.
\item[(iv)] The algebra $F^\infty(g)$  contains no non-trivial
idempotents.
\item[(v)] If  $A\in F^\infty(g)$, $n\geq 2$, then
$\sigma (A)=\sigma_e (A)$.
\end{enumerate}

Let  $\Gamma:\TT\to B(F^2(H_n))$, where $\TT:=\{z\in \CC:\  |z|=1\}$,  be the strongly continuous unitary representation  of the torus 
defined by
$
\Gamma(e^{i\theta})f:=\sum_{\alpha\in \FF_n^+} e^{i |\alpha| \theta} c_\alpha e_\alpha
$
for any $f=\sum_{\alpha\in \FF_n^+}  c_\alpha e_\alpha$ in $F^2(H_n)$. We have the orthogonal decomposition
$F^2(H_n)=\bigoplus_{p\in \ZZ} \cE_p$, where the spectral subspace $\cE_p$ is the image of the orthogonal projection
$Q_p\in B(F^2(H_n))$ defined by
$
Q_p:=\frac{1}{2\pi} \int_0^{2\pi} e^{-ip\theta} \Gamma(e^{i\theta}) d\theta,
$
where the integral is defined as a weak integral and the integrant is a continuous function in the strong operator topology. We note that $Q_p=0$ and $\cE_p=\{0\}$ whenever $p<0$. The spectral subspaces of $\Gamma$ are
$$
\cE_p=\{f\in F^2(H_n):\ \Gamma(e^{i\theta})f=e^{ip \theta } f\}=\text{\rm span}\{ e_\alpha:\ \alpha\in \FF_n^+, |\alpha|=p\}
$$
for $p\in \NN\cup\{0\}$.
Let $T\in B(F^2(H_n)\otimes \cK)$ and $s\in \ZZ$. The $s$-homogeneous part of $T$ is defined as the operator $T_s\in B(F^2(H_n)\otimes \cK)$ given by
\begin{equation}
\label{Ts}
T_s:=\frac{1}{2\pi} \int_0^{2\pi} e^{-is\theta}( \Gamma(e^{i\theta})\otimes I_\cK)T( \Gamma(e^{i\theta})\otimes I_\cK)^* d\theta.
\end{equation}
Hence $\|T_s\|\leq \|T\|$.
Note that $(T^*)_s=(T_{-s})^*$ and   $ \Gamma(e^{i\theta}) ^*|_{\cE_p}=e^{-ip\theta} I_{\cE_p}$ for every $p\in \ZZ$. On the other hand, for every $f\in \cE_p\otimes  \cK$,
$$
T_sf =\frac{1}{2\pi} \int_0^{2\pi} e^{-i(s+p)\theta}(\Gamma(e^{i\theta})\otimes I_\cK)T fd\theta
=(Q_{s+p}\otimes I_\cK)Tf.
$$
Consequently, $T_s (\cE_p\otimes \cK)\subset  \cE_{s+p}\otimes \cK$ for all $s,p\in \ZZ$.
An operator $A\in B( F^2(H_n)\otimes \cK)$ is said to be {\it homogeneous} of degree $s\in \ZZ$ if
$A( \cE_p\otimes \cK)\subset  \cE_{s+p}\otimes \cK$ for every $p\in \ZZ$.

We recall, from \cite{K}, that if $X$ is a Banach space, $\varphi:\TT\to X$ is  a continuous function , and $\kappa_N$ is a summability kernel, then
$$
\varphi(0)=\lim_{N\to\infty} \frac{1}{2\pi} \int_0^{2\pi} \kappa_N(e^{i\theta})\varphi(e^{i\theta}) d\theta.
$$

\begin{proposition}  \label{homo} If  $T\in B( F^2(H_n)\otimes \cK)$ and $\{T_s\}_{s\in \ZZ}$ are the homogeneous parts of $T$, then
$$
Tf= \lim_{N\to\infty}\sum_{s\in \ZZ, |s|\leq N} \left(1-\frac{|s|}{N+1}\right) T_s f
$$
for any $f\in F^2(H_n)\otimes  \cK$  and
$$
\left\|\sum_{s\in \ZZ, |s|\leq N} \left(1-\frac{|s|}{N+1}\right) T_s \right\|\leq \|T\|,\qquad N\in \NN.
$$
\end{proposition}
\begin{proof} Let $f\in  F^2(H_n)\otimes \cK$  and $\psi:\RR\to F^2(H_n)\otimes  \cK$  be the function defined by
$$
\psi(\theta):=(  \Gamma(e^{i\theta}) \otimes I_\cK)T( \Gamma(e^{i\theta})\otimes I_\cK)^* )f.
$$
Considering the Fej\' er kernel 
$\kappa_N(e^{i\theta}):=\sum_{s\in \ZZ, |s|\leq N} \left(1-\frac{|s|}{N+1}\right)  e^{is\theta}$ and using the result preceding the proposition, we deduce that
\begin{equation*}
\begin{split}
Tf&=\psi(0)=\lim_{N\to\infty} \frac{1}{2\pi} \int_0^{2\pi}  \kappa_N(e^{i\theta})\psi(\theta)d\theta\\
&=\lim_{N\to \infty} \sum_{s\in \ZZ, |s|\leq N}\left(1-\frac{|s|}{N+1}\right)  \frac{1}{2\pi} \int_0^{2\pi} 
e^{-is\theta}( \Gamma(e^{i\theta}) \otimes I_\cK)T( \Gamma(e^{i\theta})\otimes I_\cK)^* f d\theta\\
&=\lim_{N\to \infty} \sum_{s\in \ZZ, |s|\leq N}\left(1-\frac{|s|}{N+1}\right)T_sf
\end{split}
\end{equation*}
for any $f\in  F^2(H_n)\otimes \cK$.
Due to  relation \eqref{Ts}, we obtain
$$
s_k(f):=\sum_{s=-k}^k T_sf= \frac{1}{2\pi} \int_0^{2\pi} \left(\sum_{s=-k}^k e^{-is\theta}\right) ( \Gamma(e^{i\theta}) \otimes I_\cK)T( \Gamma(e^{i\theta})\otimes I_\cK)^*fd\theta
$$
and
\begin{equation*}
\begin{split}
\sigma_N(f)&:=\frac{1}{N}(s_0(f)+\cdots +s_{N-1}(f))
=\sum_{s\in \ZZ, |s|\leq N} \left(1-\frac{|s|}{N+1}\right) T_s f\\
&=\frac{1}{2\pi} \int_0^{2\pi} \kappa_N(e^{i\theta}) (  \Gamma(e^{i\theta})\otimes I_\cK )T( \Gamma(e^{i\theta})\otimes I_\cK)^*fd\theta
\end{split}
\end{equation*}
for any $f\in  F^2(H_n)\otimes \cK$,
where  $\kappa_N(e^{i\theta})$ is the Fej\' er kernel. Hence, we deduce that
$$
\left\|\sum_{s\in \ZZ, |s|\leq N} \left(1-\frac{|s|}{N+1}\right) T_s f \right\|\leq \frac{1}{2\pi} \int_0^{2\pi} \kappa_N(e^{i\theta}) \left\|(  \Gamma(e^{i\theta}) \otimes I_\cK)T( \Gamma(e^{i\theta})\otimes I_\cK)^*f\right\|d\theta
\leq 
\|T\|\|f\|
$$
for any $N\in \NN$ and $f\in  F^2(H_n)\otimes \cK$. This completes the proof.
\end{proof}

We remark that, as in the proof of Proposition  \ref{tilde-f2}, one can show that any operator in the commutant 
$\{ \Lambda_1\otimes I_\cK,\ldots,  \Lambda_n\otimes I_\cK\}'$ has a unique Fourier representation 
$\sum_{\alpha\in \FF_n^+}  W_\alpha\otimes C_{(\alpha)}$, where  $\{C_{(\alpha)}\}\in B(\cK)$.

\begin{theorem} \label{analytic} Let $\{b_\alpha\}_{\alpha\in \FF_n^+}$ be  a collection of strictly positive  real numbers such that
$$
\sup_{\alpha\in \FF_n^+} \frac{b_\alpha}{b_{g_i \alpha}}<\infty,\quad \text{ and }\quad  \sup_{\alpha\in \FF_n^+} \frac{b_\alpha}{b_{ \alpha g_i}}<\infty,
\qquad i\in \{1,\ldots, n\},
$$
and let ~$(W_1,\ldots, W_n)$  $($resp.
$(\Lambda_1,\ldots, \Lambda_n)$ be the  associated weighted left
(resp. right) creation operators.
Let $A\in B( F^2(H_n)\otimes \cK)$ be such that 
$
A( \Lambda_i\otimes I_\cK)=( \Lambda_i\otimes I_\cK)A$,  $i\in \{1,\ldots, n\}$.
If $A$ has the Fourier representation $\sum_{\alpha\in \FF_n^+}  W_\alpha\otimes C_{(\alpha)}$, then
$$
A= \text{\rm SOT-}\lim_{N\to\infty}\sum_{ |\alpha|\leq N} \left(1-\frac{|\alpha|}{N+1}\right) ( W_\alpha\otimes C_{(\alpha)}), 
$$
$$
\left\|\sum_{\alpha\in \FF_n^+, |\alpha|=k}  W_\alpha\otimes C_{(\alpha)}\right\|\leq \|A\|,\qquad k\in \NN,
$$
   and
$$
\left\|\sum_{ |\alpha|\leq N} \left(1-\frac{|\alpha|}{N+1}\right) ( W_\alpha\otimes C_{(\alpha)})  \right\|\leq \|A\|,\qquad N\in \NN.
$$

\end{theorem}
\begin{proof}
First, we prove that 
 the $s$-homogeneous part of $A$ is 
$
A_s=\sum_{|\alpha|=s}  W_\alpha\otimes C_{(\alpha)},\ \text{if } \ s\geq 0,
$
and $A_s=0$ if $s<0$.
 Let $f\in \cE_p\otimes  \cK$, $p\in \ZZ$, and note that
 \begin{equation*}
 \begin{split}
 A_s f&=\frac{1}{2\pi} \int_0^{2\pi} e^{-i(s+p)}( \Gamma(e^{i\theta})\otimes I_\cK)Af d\theta
 =( Q_{s+p}\otimes I_\cK)Af\\
 &= ( Q_{s+p}\otimes I_\cK)\left(\sum_{k=0}^\infty \sum_{|\alpha|=k} ( W_\alpha\otimes C_{(\alpha)}) f\right).
 \end{split}
 \end{equation*}
If $f= e_\beta\otimes h$ with $|\beta|=p\geq 0$ and $s\geq0$, then 
$$
A_s( e_\alpha\otimes h)=(Q_{s+p}\otimes I_\cK)\left(\sum_{k=0}^\infty \sum_{|\alpha|=k} ( W_\alpha \times C_{(\alpha)}) f\right)=\sum_{|\alpha|=s}  W_\alpha e_\beta \otimes C_{(\alpha)}h.
$$
If $s<0$, then  $s+p<k+p$ for any $k\geq 0$ and, consequently, $A_s( e_\alpha\otimes h)=0$.  
Now, using Proposition  \ref{homo}, we complete the proof. \end{proof}

\begin{theorem}  \label{closure}  Let  $g= 1+\sum_{ |\alpha|\geq 1} b_\alpha
Z_\alpha$ be  a free holomorphic function  with $b_\alpha>0$, if $|\alpha|\geq 1$, and 
  such that
$$
\sup_{\alpha\in \FF_n^+} \frac{b_\alpha}{b_{g_i \alpha}}<\infty,\quad \text{ and }\quad  \sup_{\alpha\in \FF_n^+} \frac{b_\alpha}{b_{ \alpha g_i}}<\infty,
\qquad i\in \{1,\ldots, n\},
$$
and let ~${\bf W}=(W_1,\ldots, W_n)$   be the  associated weighted left
  creation operators.  Then   the noncommutative  Hardy space $F^\infty(g)$  satisfies the relation
  $$
  F^\infty(g)=\overline{\cP({\bf W})}^{SOT}=\overline{\cP({\bf W})}^{WOT}=\overline{\cP({\bf W})}^{w*},
  $$
where $\cP({\bf W})$ stands for the algebra of all polynomials in $W_1,\ldots, W_n$ and the identity.
Moreover,  $F^\infty(g)$ is the sequential SOT-(resp. WOT-, w*-) closure of   $\cP({\bf W})$.
\end{theorem}
\begin{proof}  Due to Proposition \ref{tilde-f2},   the noncommutative Hardy algebra $F^\infty(g)$ is WOT-(resp. SOT-, w*-) closed. On the other hand, Theorem \ref{analytic} implies 
$F^\infty(g)\subset \overline{\cP({\bf W})}^{SOT}$ and $F^\infty(g)\subset \overline{\cP({\bf W})}^{w*}$.
Combining these results, we obtain
$$
F^\infty(g)\subset \overline{\cP({\bf W})}^{SOT}\subset \overline{\cP({\bf W})}^{WOT}\subset F^\infty(g)
\quad \text{
and  }\quad 
F^\infty(g)\subset \overline{\cP({\bf W})}^{w*}\subset F^\infty(g),
$$
which implies
$$
  F^\infty(g)=\overline{\cP({\bf W})}^{SOT}=\overline{\cP({\bf W})}^{WOT}=\overline{\cP({\bf W})}^{w*}.
  $$
Moreover,
according to Theorem \ref{analytic}, each element in $F^\infty(g)$ is the  SOT-(resp. WOT-, w*-) limit of a sequence of polynomials in $\cP({\bf W})$. Now, using the results above, we conclude that $F^\infty(g)$ is the sequential SOT-(resp. WOT-, w*-) closure of all polynomials  $\cP({\bf W})$. The proof is complete.
\end{proof}

An operator  $A\in B( F^2(H_n)\otimes \cK)$ is called {\it multi-analytic} with respect to  $ \Lambda_1\otimes I_\cK,\ldots,  \Lambda_n\otimes I_\cK$  if 
$$
A( \Lambda_i\otimes I_\cK)=( \Lambda_i\otimes I_\cK) A,\qquad i\in\{1,\ldots, n\}.
$$
An extension of Theorem \ref{closure} is the following 
\begin{proposition}  \label{multi} Under the hypothesis of Theorem \ref{closure}, the set of all  multi-analytic operators with respect to $ \Lambda_1\otimes I_\cK,\ldots,  \Lambda_n\otimes I_\cK$ coincides with
$ F^\infty(g)\bar{\otimes}_{min}B(\cK)$, the WOT-closure of the spatial tensor product $ F^\infty(g){\otimes}_{min}B(\cK)$. Moreover,
$$
 F^\infty(g)\bar{\otimes}_{min}B(\cK)=\overline{  \cP({\bf W}){\otimes}_{min}B(\cK)}^{SOT}=\overline{ \cP({\bf W}){\otimes}_{min}B(\cK)}^{w*}
$$
and each element in  $  F^\infty(g)\bar{\otimes}_{min}B(\cK)$ is the   sequential SOT-(resp. WOT-, w*-) closure of operator-valued  polynomials  of the form
$\sum_{|\alpha|\leq m}  W_\alpha\otimes C_{(\alpha)}$, where  $C_{(\alpha)}\in B(\cK)$.
\end{proposition}
\begin{proof}  Let $M$ be the set of all  multi-analytic operators with respect to $ \Lambda_1\otimes I_\cK,\ldots,  \Lambda_n\otimes I_\cK$ and note that   $M$ is WOT-(resp. SOT-, w*-) closed.   As in the proof of Theorem  \ref{closure}, one can use Theorem \ref{analytic} to prove that
$$
M=\overline{  \cP({\bf W}){\otimes}_{min}B(\cK)}^{SOT}=\overline{  \cP({\bf W}){\otimes}_{min}B(\cK)}^{WOT}=\overline{ \cP({\bf W}){\otimes}_{min}B(\cK)}^{w*}
$$
and that each element in  $M$ is the    sequential SOT-(resp. WOT-, w*-) limit  of operator-valued  polynomials  of the form
$\sum_{|\alpha|\leq m}  W_\alpha\otimes C_{(\alpha)}$, where  $C_{(\alpha)}\in B(\cK)$.
On the other hand, we have
$$
M=\overline{  \cP({\bf W}){\otimes}_{min}B(\cK)}^{WOT}\subset \overline{ F^\infty(g){\otimes}_{min}B(\cK)}^{WOT}=
F^\infty(g)\bar{\otimes}_{min} B(\cK).
$$
Since $F^\infty(g)=\{ \Lambda_1,\ldots, \Lambda_n\}^\prime$,  each element in  $ F^\infty(g)\otimes_{min}B(\cK)$ commutes with $ \Lambda_i\otimes I_\cK$ for $i\in \{1,\ldots, n\}$.
Hence $F^\infty(g)\bar{\otimes}_{min} B(\cK)\subset M$, which completes the proof.
\end{proof}
We remark that  a similar result holds for the Hardy algebra $R^\infty(g)$.
In what follows, we prove the existence of a $w^*$-continuous $F^\infty(g)$-functional calculus for the pure elements in the noncommutative domain $\cD_{g^{-1}}(\cH)$.
\begin{theorem}  \label{pure2} Let $g$ be an admissible free holomorphic function such that 
$
\sup_{\alpha\in \FF_n^+} \frac{b_\alpha}{b_{\alpha g_i}}<\infty$  for  any $i\in \{1,\ldots, n\}$,
and let ${\bf W}=(W_1,\ldots, W_n)$ be the universal model of the noncommutative domain $\cD_{g^{-1}}$.
If $T=(T_1,\ldots, T_n)\in \cD_{g^{-1}}(\cH)$ is a pure $n$-tuple of operators, then the map
$\Psi_T:F^\infty(g)\to B(\cH)$ defined by
$$
\Psi_T(\varphi({\bf W})):=K_{g,T}^*(\varphi({\bf W})\otimes I_\cH) K_{g,T}, \qquad \varphi({\bf W})\in F^\infty(g),
$$
has the following  properties:
\begin{enumerate}
\item[(i)]  $\Psi_T$ is a sequentially WOT-continuous (resp. SOT-continuous) map;

\item[(ii)] $\Psi_T$ is a unital completely contractive  homomorphism;
\item[(iii)] $\Psi_T$ is a $w^*$-continuous linear map such that
$$
\Psi_T\left(\sum_{|\alpha|\leq m} d_\alpha W_\alpha \right)=\sum_{|\alpha|\leq m} d_\alpha T_\alpha
$$
for any $m\in \NN$, $d_\alpha\in \CC$.
\end{enumerate}
If,  in addition,  ${\bf W}=(W_1,\ldots, W_n)$ is radially pure with respect to $\cD_{g^{-1}}$ and $\varphi({\bf W})\in F^\infty(g)$ has the Fourier representation $\sum_{\alpha\in \FF_n^+} c_\alpha W_\alpha$, then
   $\varphi_r(T):=\sum_{k=0}^\infty \sum_{|\alpha|=k} r^{|\alpha|} c_\alpha T_\alpha$ converges in the operator norm and 
  $$\Psi_T(\varphi({\bf W}))=\text{\rm SOT-}\lim_{r\to 1} \varphi_r(T).
  $$
\end{theorem}
\begin{proof} 
Let $\{q_k({\bf W})\}_k\subset \cP({\bf W})$ be a sequence such that 
WOT-$\lim_{k\to\infty} q_k({\bf W})=\varphi({\bf W})$.
Since the WOT and $w^*$-topologies concide on bounded set, we have 
$w^*$-$\lim_{k\to\infty} q_k({\bf W})=\varphi({\bf W})$ which implies 
$w^*$-$\lim_{k\to\infty} (q_k({\bf W})\otimes I_\cH)=\varphi({\bf W})\otimes I_\cH$.
Hence,  WOT-$\lim_{k\to\infty} (q_k({\bf W})\otimes I_\cH)=\varphi({\bf W})\otimes I_\cH$ and 
WOT-$\lim_{k\to\infty}\Psi_T(q_k({\bf W}))=\Psi_T(\varphi({\bf W}))$.
Similarly, if  SOT-$\lim_{k\to\infty} q_k({\bf W})=\varphi({\bf W})$, then 
SOT-$\lim_{k\to\infty} (q_k({\bf W})\otimes I_\cH)=\varphi({\bf W})\otimes I_\cH$ which implies that
SOT-$\lim_{k\to\infty}\Psi_T(q_k({\bf W}))=\Psi_T(\varphi({\bf W}))$.
According  to   Theorem \ref{model}, the Berezin kernel  $K_{g,T}$ is an isometry and 
 $
 K_{g,T} T_i^*=(W_i^*\otimes I_\cD)K_{g,T},\  i\in \{1,\ldots, n\}.
 $
Consequently, 
$$
\Psi_T\left(\sum_{|\alpha|\leq m} d_\alpha W_\alpha \right)=\sum_{|\alpha|\leq m} d_\alpha T_\alpha
$$
for any $m\in \NN$, $d_\alpha\in \CC$.  Since $\Psi_T$ is a homomorphism on the algebra $\cP({\bf W})$ and  $F^\infty(g)=\overline{\cP({\bf W})}^{WOT}$, one can easily show that $\Psi_T$ is a homomorphism on 
$F^\infty(g)$.
On the other hand, since
$$
[\Psi_T(\varphi_{ij})]_{p\times p}
=\left(\oplus_1^p K_{g,T}^*\right)[\varphi_{ij}(W)\otimes I_\cH]_{p\times p}\left(\oplus_1^p K_{g,T}\right),\qquad p\in \NN,
$$
and $K_{g,T}$ is an isometry, it is clear that $\Psi_T$ is completely contractive. This proves part (ii).
To prove part (iii), note that if $\{\varphi_\iota({\bf W})\}_{\iota}$ is a net in $F^\infty(g)$ such that 
$w^*$-$\lim_\iota \varphi_\iota({\bf W})=\varphi({\bf W})\in F^\infty(g)$, then 
$w^*$-$\lim_\iota \varphi_\iota({\bf W})\otimes I_\cH=\varphi({\bf W})\otimes I_\cH$, which, due to the definition of $\Psi_T$,  implies
$w^*$-$\lim_\iota \Psi_T(\varphi_\iota({\bf W}))=\Psi_T(\varphi({\bf W}))$. Therefore $\Psi_T$ is $w^*$-continuous.

To prove the last part of the theorem, suppose that ${\bf W}=(W_1,\ldots, W_n)$ is radially pure with respect to $\cD_{g^{-1}}$ and $\varphi({\bf W})\in F^\infty(g)$ has the Fourier representation $\sum_{\alpha\in \FF_n^+} c_\alpha W_\alpha$. According to Theorem \ref{analytic}, we have 
$
\left\|\sum_{ |\alpha|=k} c_\alpha  W_\alpha\right\|\leq \|\varphi(W)\|.
$
If $r\in [0,1)$, then we have
$$
\left\|\sum_{k=0}^\infty\sum_{|\alpha|=k} c_\alpha r^{|\alpha|}W_\alpha\right\|
\leq \sum_{k=0}^\infty r^k\left\|\sum_{ |\alpha|=k} c_\alpha  W_\alpha\right\|
\leq \frac{1}{1-r}\|\varphi({\bf W})\|.
$$
This  shows that  $\varphi_r({\bf W}):=\sum_{k=0}^\infty \sum_{|\alpha|=k} r^{|\alpha|} c_\alpha W_\alpha$ converges in norm.  
Since ${\bf W}=(W_1,\ldots, W_n)$ is radially pure with respect to $\cD_{g^{-1}}$, there is $\delta\in(0,1)$ such that $r{\bf W}\in \cD_{g^{-1}}(F^2(H_n))$  is pure for all $r\in (\delta,1)$.
This implies that  $rT\in \cD_{g^{-1}}(\cH)$ is pure as well.
Since  $T$ is pure, Theorem \ref{pure} implies
$
\left\|\sum_{ |\alpha|=k} c_\alpha  T_\alpha\right\| 
\leq \left\|\sum_{ |\alpha|=k} c_\alpha  W_\alpha\right\|.
$
Combining this relation   with the inequalities above, we deduce that
 $
 \varphi_r(T):=\sum_{k=0}^\infty \sum_{|\alpha|=k} r^{|\alpha|} c_\alpha T_\alpha
 $
 converges in the operator norm. Moreover, due to Theorem \ref{pure}, we have
 \begin{equation}
 \label{fir}
 \varphi_r(T)=K_{g,T}^*(\varphi_r({\bf W})\otimes I_\cH)K_{g,T},\qquad r\in(\delta,1),
 \end{equation}
which implies $\|\varphi_r(T)\|\leq \|\varphi_r({\bf W})\|$.
The next step is to prove that $\|\varphi_r({\bf W})\|\leq \|\varphi({\bf W})\|$  for $r\in (\delta,1)$ and 
$\text{\rm SOT-}\lim_{r\to 1} \varphi_r({\bf W})=\varphi({\bf W})$.
To this end,  note that 
$$
K_{g,r{\bf W}}p_r({\bf W})^*=[p({\bf W})^*\otimes
I_{F^2(H_n)}]K_{g,r{\bf W}}
$$
for any $r\in(\delta,1)$ and any polynomial $p({\bf W})\in \cP({\bf W})$.
  Let  $\gamma, \sigma, \epsilon\in \FF_n^+$ and set
$p({\bf W}):=\sum\limits_{\beta\in \FF_n^+, |\beta|\leq
|\gamma|} c_\beta W_\beta$. Since $W_\beta^* e_\gamma =0$ for any
$\beta\in \FF_n^+$ with $|\beta|>|\gamma|$,  we have
$
\varphi_r({\bf W})^* e_\alpha =p_r({\bf W})^* e_\alpha$
for any $\alpha\in \FF_n^+$ with $|\alpha|\leq |\gamma|$ and any
$r\in [0,1]$.
Consequently, straightforward  calculations
reveal that
\begin{equation*}
\begin{split}
\left<K_{f,r{\bf W}}\varphi_r({\bf W}))^*e_\gamma,
e_\sigma\otimes e_\epsilon\right> 
&=\left<K_{f,r{\bf W}}p_r({\bf W})^*e_\gamma,
e_\sigma\otimes e_\epsilon\right>\\
&=\left<[(p({\bf W})^*\otimes
I_{F^2(H_n)})]K_{g,r{\bf W}}e_\gamma,
e_\sigma\otimes e_\epsilon\right>\\
&=\sum_{\beta\in \FF_n^+} r^{|\beta|}
\sqrt{b_\beta}\left<p({\bf W})^* e_\beta,e_\sigma\right>
\left< W_\beta^*e_\gamma,  \Delta_{g^{-1},r{\bf W}} e_\epsilon\right>\\
&=\sum_{\beta\in \FF_n^+} r^{|\beta|}
\sqrt{b_\beta}\left<\varphi({\bf W})^* e_\beta,e_\sigma\right>
\left< W_\beta^*e_\gamma,  \Delta_{g^{-1},r{\bf W}} e_\epsilon\right>\\
&= \left<[\varphi({\bf W})^*\otimes I_{F^2(H_n)}] K_{g,r{\bf W}}
e_\gamma, e_\sigma\otimes e_\epsilon\right>
\end{split}
\end{equation*}
for any $r\in (\delta,1)$ and $\gamma, \sigma,\epsilon\in \FF_n^+$.
Consequently, since $\varphi_r({\bf W})$ and  $\varphi({\bf W})$  are
bounded operators, we deduce that
$$
K_{g,r{\bf W}}\varphi_r({\bf W})^*=[\varphi({\bf W})^*\otimes
I_{F^2(H_n)}]K_{g,r{\bf W}}.
$$
Since $K_{g,r{\bf W}}$ is an isometry, we deduce that  $\|\varphi_r({\bf W})\|\leq \|\varphi({\bf W})\|$  for $r\in (\delta,1)$.
Now,  using that fact that  $\lim_{r\to 1} \varphi_r({\bf W})e_\alpha=\varphi({\bf W})e_\alpha$ for all $\alpha\in \FF_n^+$, an approximation argument shows that 
$\text{\rm SOT-}\lim_{r\to 1} \varphi_r({\bf W})=\varphi({\bf W})$.
Since $Y\mapsto Y\otimes I$ is SOT-continuous  on bounded sets, we can pass to the limit in relation \eqref{fir} and deduce that
$
\text{\rm SOT-}\lim_{r\to 1} \varphi_r(T)=K_{g,T}^*(\varphi({\bf W})\otimes I_\cH)K_{g,T}.
$
This completes the proof.
\end{proof}

Our next goal is to extend the $w^*$-continuous $F^\infty(g)$-functional calculus calculus to completely
 non-coisometric $n$-tuples in 
$\cD_{g^{-1}}(\cH)$.  We define the domain ${\cD}^{cnc}_{g^{-1}}(\cH)$   of all $n$-tuples $T\in {\cD}_{g^{-1}}(\cH)$ which are completely non-coisometric and note that ${\cD}^{pure}_{g^{-1}}(\cH)\subset {\cD}^{cnc}_{g^{-1}}(\cH)$.

\begin{definition}   An $n$-tuple $T=(T_1,\ldots, T_n)\in \cD_{g^{-1}}(\cH)$ is said to be  {\it completely non-coisometric} with respect to $ \cD_{g^{-1}}(\cH)$ if 
there is no nonzero joint invariant  subspace  under $T_1^*,\ldots, T_n^*$ such that 
$
\Delta_{g^{-1}}(T,T^*)|_{\cH_0}=0.
$
\end{definition}
We remark that, in the single variable case of one contraction, i.e. $\|T\|\leq 1$, we recover the well-known definition of a completely non-coisometric  contraction.

\begin{proposition} \label{cnc} 
If $T=(T_1,\ldots, T_n)\in \cD_{g^{-1}}(\cH)$, then the following statements are equivalent.
\begin{enumerate}
  \item[(i)] There is no $h\in \cH$, $h\neq 0$    such that 
$
\Delta_{g^{-1}}(T,T^*)T_\alpha^*h=0$ for any $\alpha\in \FF_n^+.
$
\item[(ii)] $T$ is completely non-coisometric with respect  to $\cD_{g^{-1}}(\cH)$.
\item[(iii)] The noncommutative Berezin kernel  $K_{g,T}$ is a one-to-one operator.
\end{enumerate}
\end{proposition}
\begin{proof} 
Note that item (ii) is equivalent to the condition
$
 \bigcap_{\alpha\in \FF_n^+} \ker \Delta_{g^{-1}}(T,T^*)T_\alpha^*=\{0\},
$
which shows that  (i) is equivalent to (ii).  Recall that  the noncommutative Berezin kernel associated with 
 $T\in \cD_{g^{-1}}(\cH)$ is  the operator $K_{g,T}:\cH\to
F^2(H_n)\otimes \overline{\Delta_{g^{-1}} (T,T^*) (\cH)}$  given by
\begin{equation*}
 K_{g,T}h=\sum_{\alpha\in \FF_n^+} \sqrt{b_\alpha}
e_\alpha\otimes \Delta_{g^{-1}}(T,T^*)^{1/2} T_\alpha^* h,\quad h\in \cH.
\end{equation*}
Consequently, if $h\in \cH$, then
$$\left< K_{g,T}^*K_{g,T}h,h\right >=    \sum_{\alpha\in \FF_n^+}b_\alpha \left<T_\alpha\Delta_{g^{-1}}(T,T^*)T_\alpha^*h, h\right>=0
$$ 
if and only if $\Delta_{g^{-1}}(T,T^*)^{1/2}T_\alpha^*h=0$ for any $\alpha\in \FF_n^+$, if and only if
$h\in \bigcap_{\alpha\in \FF_n^+} \ker \Delta_{g^{-1}}(T,T^*)^{1/2}T_\alpha^*$.
Since 
$$ \bigcap_{\alpha\in \FF_n^+} \ker \Delta_{g^{-1}}(T,T^*)T_\alpha^*=\bigcap_{\alpha\in \FF_n^+} \ker \Delta_{g^{-1}}(T,T^*)^{1/2}T_\alpha^*,
 $$
 it is clear that (iii) is equivalent to (i). The proof is complete.
\end{proof}

\begin{corollary}  Any  $n$-tuple $T\in \cD_{g^{-1}}(\cH)$ which is either pure or has the defect operator
 $\Delta_{g^{-1}}(T,T^*)$ injective is completely non-coisometric with respect to  $\cD_{g^{-1}}(\cH)$.
\end{corollary}

Recall that $A\in B(\cH)$ is a trace class operator if $\|A\|_1:=\text{\rm trace} (A^*A)^{1/2}<\infty$. We denote by $\cT(\cH)$ the ideal of trace class operators in $B(\cH)$. It is well-known that $\cT(\cH)$ is a Banach space with respect to the trace norm $\|\cdot \|_1$ and the dual space  $\cT(\cH)^*$ is isometrically isomorphic to $B(\cH)$.

\begin{lemma} \label{dual} Let  $g= 1+\sum_{ |\alpha|\geq 1} b_\alpha
Z_\alpha$ be  a free holomorphic function  with $b_\alpha>0$, if $|\alpha|\geq 1$, and 
  such that
$$
\sup_{\alpha\in \FF_n^+} \frac{b_\alpha}{b_{g_i \alpha}}<\infty\quad \text{ and }\quad  \sup_{\alpha\in \FF_n^+} \frac{b_\alpha}{b_{ \alpha g_i}}<\infty,
\qquad i\in \{1,\ldots, n\},
$$
and let $F^\infty(g)$ be the associated noncommutative Hardy algebra. If $\Phi:F^\infty(g)\to B(\cH)$ is a linear map, then the following are equivalent:
\begin{enumerate}
\item[(i)]  $\Phi: (F^\infty(g), w^*)\to (B(\cH), w^*)$ is continuous;
\item[(ii)]  $\Phi: (F^\infty(g), w^*)\to (B(\cH), WOT)$  is sequential continuous.
\end{enumerate}

\end{lemma}
\begin{proof} First, note that   Theorem \ref{closure}  ensures that   $F^\infty(g)$ is   a $w^*$-closed  subspace in $B(F^2(H_n))$. 
Since  $\cT(F^2(H_n))^*=B(F^2(H_n))$,  the duality theory for Banach spaces shows that
$$
F^\infty(g)=\left(\cT(F^2(H_n))/_{{}^\perp F^\infty(g)}\right)^*,
$$
where 
$${}^\perp F^\infty(g):=\left\{ T\in \cT(F^2(H_n)):\ \text{\rm trace}(TF)=0, \text{ for any } F\in F^\infty(g)\right\}
$$
is the pre-annihilator    of  $F^\infty(g)$ in  $\cT(F^2(H_n))$.
Since $\cT(F^2(H_n))$ is separable, so is  the closed subspace ${}^\perp F^\infty(g)$.     Now, it is clear that
the quotient  $\cT(F^2(H_n))/_{{}^\perp F^\infty(g)}$ is a separable Banach space and, consequently,  the noncommutative Hardy algebra $F^\infty(g)$ is the dual of a separable Banach space.
Applying  Proposition  20.3 from \cite{Con} in our setting we complete the proof.
\end{proof}

In what follows, we provide a $w^*$-continuous $F^\infty(g)$-functional 
calculus for the completely non-coisometric $n$-tuples in $\cD_{g^{-1}}(\cH)$.

\begin{theorem} \label{cnc2}  Let $g= 1+\sum_{ |\alpha|\geq 1} b_\alpha
Z_\alpha$ be an admissible free holomorphic function with the property that
$
\sup_{\alpha\in \FF_n^+} \frac{b_\alpha}{b_{ \alpha g_i}}<\infty$ for any $ i\in \{1,\ldots, n\},
$
and let ${\bf W}=(W_1,\ldots, W_n)$ be the universal model of the noncommutative domain $\cD_{g^{-1}}$.
Suppose that  $T=(T_1,\ldots, T_n)\in \cD^{cnc}_{g^{-1}}(\cH)$.   
 \begin{enumerate}
\item[(i)] If  $T\in \overline{\cD^{pure}_{g^{-1}}(\cH)}$, then the completely contractive  linear map 
$$
\Phi:\cA(g)\to B(\cH),\qquad \Phi(p({\bf W})):=p(T),
$$ 
where $p$ is any polynomial in $n$  noncommutative indeterminates,
 has a unique extension to a $w^*$-continuous  homomorphism
$\Psi_T:F^\infty(g)\to B(\cH)$ 
 which is  completely  contractive and  sequentially WOT-continuous (resp. SOT-continuous).
 Moreover,  if $\varphi({\bf W})\in F^\infty(g)$ has the Fourier representation $\sum_{\alpha\in \FF_n^+} c_\alpha W_\alpha$, then
 $$
 \Psi_T(\varphi({\bf W}))=\text{\rm SOT-} \lim_{N\to\infty}\sum_{s\in \NN, s\leq N} \left(1-\frac{s}{N+1}\right)\sum_{|\alpha |=s} c_\alpha T_\alpha.
 $$
\item[(ii)]
In particular, if  $T$ and ${\bf W}$ are radially pure with respect to $\cD_{g^{-1}}$ and $\varphi({\bf W})\in F^\infty(g)$ has the Fourier representation $\sum_{\alpha\in \FF_n^+} c_\alpha W_\alpha$, then
   $\varphi_r(T):=\sum_{k=0}^\infty \sum_{|\alpha|=k} r^{|\alpha | } c_\alpha T_\alpha$ converges in the operator norm and 
 \begin{equation*}
 \begin{split}
 \Psi_T(\varphi({\bf W}))&=\text{\rm SOT-}\lim_{r\to 1} \varphi_r(T)
 =\text{\rm SOT-}\lim_{r\to 1}K_{rT}^*[\varphi({\bf W})\otimes I_\cH] K_{rT}.
 \end{split}
 \end{equation*}
 \end{enumerate}
\end{theorem}
\begin{proof}
Let $\varphi({\bf W})\in F^\infty(g)$ and let $\{p_k({\bf W})\}_{k=1}^\infty\subset \cP({\bf W})$ be a sequence of polynomials such that
SOT-$\lim_{k\to\infty}p_k({\bf W})=\varphi({\bf W})$.  Since $T\in \cD_{g^{-1}}(\cH)$,  Corollary \ref{Berez} shows that
\begin{equation}
\label{pKKp}
p_k(T)K_{g,T}^*x=K_{g,T}^*(p_k({\bf W})\otimes I_\cH)x,\qquad x\in F^2(H_n)\otimes \cH.
\end{equation}
Using the fact that the map $Y\mapsto Y\otimes I$ is SOT-continuous  on bounded sets, we can define the operator  $A:\text{\rm range} K_{g,T}^*\to \cH$  by setting
\begin{equation}\label{A}
A(K_{g,T}^*x):=\lim_{k\to\infty}p_k(T)K_{g,T}^*x=K_{g,T}^*(\varphi({\bf W})\otimes I_\cH)x, \qquad x\in F^2(H_n)\otimes \cH.
\end{equation}
Note that $A$ does not depend on the choice of the sequence  $\{p_k({\bf W})\}_{k=1}^\infty\subset \cP(W)$.
Consequently, according to Theorem \ref{analytic},  we can take 
$$
p_k({\bf W}):=\sum_{s\in \NN, s\leq k} \left(1-\frac{s}{N+1}\right)\sum_{|\alpha |=s} c_\alpha W_\alpha,
$$
in which case
$\|p_k({\bf W})\|\leq \|\varphi({\bf W})\|$.
Suppose that  $T\in \overline{\cD^{pure}_{g^{-1}}(\cH)}$ and let $\{T^{(m)}\}_{m\in \NN}\subset \cD^{pure}_{g^{-1}}(\cH)$ be a sequence such that $\|T-T^{(m)}\|\to 0$ as $m\to \infty$.
Due to Theorem \ref{pure}, we have $\|p_k(T^{(m)})\|\leq \|p_k({\bf W})\|\leq \|\varphi({\bf W})\|$ and 
\begin{equation*}
\begin{split}
\|A(K_{g,T}^*x)\|&\leq \sup_k \|p_k(T)\|\|K_{g,T}^*x\|
\leq \sup_{k, m}  \|p_k(T^{(m)})\|\|K_{g,T}^*x\|
\leq  \|\varphi({\bf W})\| \|K_{g,T}^*x\|
\end{split}
\end{equation*}
for any $x\in F^2(H_n)\otimes \cH.$
Since $T$ is completely non-coisometric, Proposition  \ref{cnc} shows that the noncommutative Berezin kernel $K_{g,T}$ is one-to-one, thus $\text{\rm range}\, K_{g,T}^*$ is dense in $\cH$. Consequently, $A$ can be uniquely be extended to a bounded linear operator on $\cH$, which is also denoted by $A$, such that
$\|Ah\|\leq \|\varphi({\bf W})\|\|h\|$ for $h\in \cH$.

In what follows, we show that $Ah=\lim_{k\to\infty}p_k(T)h$ for any $h\in \cH$. To show that the later limit exists, let $\{y_n\}_{n=1}^\infty\subset \text{\rm range} K_{g,T}^*$ be such that $\|y_n-h\|\to 0$ as $n\to \infty$.
Fix $\epsilon>0$ and let $N\in \NN$ be such that$ \|y_N-h\|<\frac{\epsilon}{2\|\varphi({\bf W})\|}$.
Using the fact that $\|p_k(T^{(m)})\|\leq \|p_k({\bf W})\|\leq \|\varphi({\bf W})\|$ for any $m,k\in \NN$, we deduce that
\begin{equation*}
\begin{split}
\|Ah-p_k(T^{(m)})h\| 
&\leq  \|A\|\|h-y_N\|+\|p_k(T^{(m)})\|\|y_N-h\|+\|Ay_N-p_k(T^{(m)})y_N\|\\
&\leq  2\|\varphi({\bf W})\|\|y_N-h\| +\|Ay_N-p_k(T^{(m)})y_N\|\\
&\leq \epsilon +\|Ay_N-p_k(T^{(m)})y_N\|.
\end{split}
\end{equation*}
Taking $m\to \infty$, we obtain
$$
\|Ah-p_k(T)h\|\leq \epsilon +\|Ay_N-p_k(T)y_N\|.
$$
Due to  relation \eqref{pKKp} and taking $k\to \infty$, we get
$\limsup_{k\to\infty} \|Ah-p_k(T)h\|\leq \epsilon$ for any $\epsilon>0$. Hence, 
$\lim_{k\to\infty} \|Ah-p_k(T)h\|=0$, which proves our assertion.

Now, we define the map $\Psi_T: F^\infty(g)\to B(\cH)$  by setting $\psi_T(\varphi({\bf W}))=\varphi(T):=A$.    
Let $\{\varphi_\iota({\bf W})\}_\iota\subset F^\infty(g)$ be a bounded net such that
$\varphi_\iota({\bf W})\to \varphi({\bf W})\in F^\infty(g)$  in the strong operator topology.
Then $\varphi_\iota({\bf W})\otimes I_\cH\to \varphi({\bf W})\otimes I_\cH$ in the same topology.
According to relation \eqref{A}, we have
$$
\varphi_\iota(T)  K_{g,T}^*x=K_{g,T}^*(\varphi_\iota({\bf W})\otimes I_\cH)x\to    K_{g,T}^*(\varphi({\bf W})\otimes I_\cH)x=\varphi(T)K_{g,T}^*x
$$
for any $x\in F^2(H_n)\otimes \cH$. 
Since $\|\varphi_\iota(T)\|\leq\|\varphi_\iota({\bf W})\|$, the net $\{\varphi_\iota(T)\}_\iota$ is also bounded. Since  $\text{\rm range} K_{g,T}^*$ is dense in $\cH$, a standard approximation argument, as above, shows that 
$\text{\rm SOT-}\lim_\iota\varphi_\iota(T)= \varphi(T)$. In a similar manner, one can prove  a similar result when SOT is replaced by WOT or $w^*$-topology. 
Consequently, since   the map $\Psi_T:F^\infty(g)\to B(\cH)$ is sequentially $w^*$-continuous, we can use Lemma \ref{dual}  to  conclude  that $\Psi_T$ is $w^*$-continuous.

 Now, we prove that $\Psi_T$  is  completely  contractive homomorphism.
 Let $[\varphi_{ij}({\bf W})]_{q}\in M_{q\times q}(\CC)\otimes F^\infty(g)$. 
 According to Theorem \ref{analytic}, there a sequence  of matrices $[p_k^{ij}({\bf W})]_{q}\in M_{q\times q}(\CC)\otimes\cP({\bf W})$ such that 
$[p_k^{ij}({\bf W})]_{q}\to  [\varphi_{ij}({\bf W})]_{q}$ strongly as $k\to \infty$ and
$\|[p_k^{ij}({\bf W})]_{q}\|\leq\| [\varphi^{ij}({\bf W})]_{q}\|$. 
Let $h=\oplus_{i=1}^q h_i\in \oplus_{i=1}^qF^2(H_n)$ and  let $\{T^{(m)}\}_{m\in \NN}\subset \cD^{pure}_{g^{-1}}(\cH)$ be a sequence such that $\|T-T^{(m)}\|\to 0$ as $m\to \infty$.
Note that, due to the results above,
$
 [\varphi_{ij}(T)]_{q}h=\lim_{k\to \infty} [p_k^{ij}(T)]_{q}h
 $
and 
\begin{equation*}
\begin{split}
\| [\varphi_{ij}(T)]_{q}h\|&\leq \sup_{k\in \NN}\| [p_k^{ij}(T)]_{q}\|\| h\|
\leq  \sup_{k,m\in \NN}\| [p_k^{ij}(T^{(m)})]_{q}]_{q}\|\| h\|
\leq   \| [\varphi_{ij}({\bf W})]_{q}\|\| h\|.
\end{split}
\end{equation*}
 Consequently,  $\| [\varphi_{ij}(T)]_{q}\|\leq   \| [\varphi_{ij}({\bf W})]_{q}\|$, which proves that   $\Psi_T$  is  a completely  contractive linear map. On the other hand, since  $\Psi_T$ is a homomorphism on the algebra of polynomials
 $\cP({\bf W})$ and $\Psi_T$ is sequentially  WOT-continuous, one can use the WOT-density of $\cP({\bf W})$ in $F^\infty(g)$ to show that  $\Psi_T$ is a homomorphism on $F^\infty(g)$.

  To prove item (ii), let $T$ and ${\bf W}$ be radially pure $n$-tuples  with respect to $\cD_{g^{-1}}$ and  let $\varphi({\bf W})\in F^\infty(g)$ have the Fourier representation $\sum_{\alpha\in \FF_n^+} c_\alpha W_\alpha$.
 According to the proof of Theorem \ref{pure2},
 $\varphi_r({\bf W}):=\sum_{k=0}^\infty \sum_{|\alpha|=k} r^{|\alpha|} c_\alpha W_\alpha$ converges in the operator norm, 
 $\text{\rm SOT-}\lim_{r\to 1} \varphi_r({\bf W})=\varphi({\bf W})$  and $ \|\varphi_r(W)\|\leq \|\varphi(W)\|$
  for  any $r\in (\delta,1)$, where $\delta\in(0,1)$ is chosen such that both $r{\bf W}$ and $rT$ are pure elements. 
  Consequently, using the von Neumann inequality of Corollary  \ref{vN-disc}, we deduce that
  \begin{equation*}
  \begin{split}
\left\|\sum_{k=0}^\infty\sum_{|\alpha|=k} c_\alpha r^{|\alpha|}T_\alpha\right\|
&\leq \sum_{k=0}^\infty r^k\left\|\sum_{ |\alpha|=k} c_\alpha  T_\alpha\right\|
 \leq \sum_{k=0}^\infty r^k\left\|\sum_{ |\alpha|=k} c_\alpha  W_\alpha\right\|
\leq \frac{1}{1-r}\|\varphi({\bf W})\|<\infty.
\end{split}
\end{equation*}
Therefore, $\sum_{k=0}^\infty\sum_{|\alpha|=k} c_\alpha r^{|\alpha|}T_\alpha$ converges in the operator norm.   
  Using the properties of the Berezin kernel,  for any $t, r \in (\delta, 1)$ such that $ rt\in (\delta,1)$, we have
 \begin{equation}
 \label{Kg}
 K_{g,rT}^*(\varphi_t({\bf W})\otimes I_\cH)=\varphi_t(rT)K_{g,rT}^*.
 \end{equation}
On the other hand,  since  $\varphi (rT)$ and $\varphi_t (rT)$ are given by series which are convergent in norm, for any $\epsilon>0$, we can find $N_0\in \NN$, such that
\begin{equation*}
\begin{split}
\|\varphi_t(rT)-\varphi (rT)\|&\leq\left\| \sum_{k=1}^{N_0}\sum_{|\beta|=k} r^k(t^k-1)c_\beta T_\beta \right\| +\epsilon
\end{split}
\end{equation*}
Taking $t\to 1$, we deduce that $\lim_{t\to 1} \varphi_t(rT)=\varphi(rT)$ in the operator norm. Consequently, relation \eqref{Kg} implies 
$
 K_{g,rT}^*(\varphi({\bf W})\otimes I_\cH)=\varphi(rT)K_{g,rT}^*.
 $
Hence, we also deduce that  $\|\varphi(rT)\|\leq \|\varphi({\bf W})\|$ for any $r\in (\delta, 1)$.
Now, since  $T_i K_{g,T}^*=K_{g,T}^* (W_i\otimes I_\cH)$, $i\in \{1,\ldots, n\}$, and using the convergence in norm of the series defining $\varphi(rT)$ and $\varphi(r{\bf W})$, we deduce that
\begin{equation}\label{d}
\varphi(rT)K_{g,T}^*=K_{g,T}^* (\varphi(r{\bf W})\otimes I_\cH)
\end{equation}
for any $r\in(\delta, 1)$.
Taking $r\to 1$ in this relation, we deduce that the map $\tilde A:\text{\rm range}K_{g,T}^*\to \cH$ given by
$\tilde Ay:=\lim_{r\to 1} \varphi(rT)y$ for $y\in \text{\rm range}K_{g,T}^*$,  is well-defined, linear, and
\begin{equation*}
\begin{split}
\tilde A(K_{g,T}^*x)&\leq \limsup_{r\to 1} \| \varphi(rT)\|\|K_{g,T}^*x\|
\leq \|\varphi({\bf W})\| \|K_{g,T}^*x\|.
\end{split}
\end{equation*}
Consequently, $\tilde A$ has a  unique continuous extension  to $\cH$ and 
 $\|\tilde A\|\leq \|\varphi({\bf W})\|$. Now,  using a standard approximation argument,  we can  show that $\tilde A h=\lim_{r\to 1} \varphi(rT)h$ for  every $h\in \cH$.
 Due to the definition of $\tilde A$ and relation \eqref{d}, we have
 $\tilde AK_{g,T}^*=K_{g,T}^* (\varphi({\bf W})\otimes I_\cH)$.  On the other hand,  due to relation \eqref{A},  we have
 $AK_{g,T}^* =K_{g,T}^*(\varphi({\bf W})\otimes I_\cH)$. Consequently, since both operators  $A$ and $\tilde A$ are  bounded and  $\text{\rm range}K_{g,T}^*$ is dense in $\cH$, we conclude that $A=\tilde A$.
 The proof is complete.
\end{proof}

\section{Multipliers and multi-analytic operators}

This section is devoted to the multipliers on Hilbert spaces $F^2(g)$ associated with formal power series $g$ with 
strictly positive coefficients and to the multi-analytic operators with respect to the universal model associated with $g$.

Let  $g= 1+\sum_{ |\alpha|\geq 1} b_\alpha
Z_\alpha$ be  a  formal power series with  $b_\alpha>0$ and let $F^2(g)$ be the Hilbert space of formal power series in indeterminates $Z_1,\ldots, Z_n$ with orthogonal basis $\{Z_\alpha:\ \alpha\in \FF_n^+\}$ such that $\|Z_\alpha\|_g:=\frac{1}{\sqrt{b_\alpha}}$. Note that
$$
F^2(g)=\left\{ \zeta=\sum_{\alpha\in \FF_n^+}c_\alpha Z_\alpha: \   \|\zeta\|_g^2:=\sum_{\alpha\in \FF_n^+}\frac{1}{b_\alpha} |c_\alpha|^2<\infty, \ c_\alpha\in \CC\right\}
$$
and $\left\{\sqrt{b_\alpha}Z_\alpha\right\}_{\alpha\in \FF_n^+}$ is an orthonormal basis for $F^2(g)$.
The {\it left multiplication} operator  $L_{Z_i}$ on $F^2(g)$ is  defined by  $L_{Z_i}\zeta:=Z_i\zeta$ for all $\zeta\in F^2(g)$. We also  use the notation $L_{Z_i}^{(g)}$ when needed.
One can easily see that $L_{Z_i}$ is a bounded operator  if and only if 
\begin{equation}\label{sup1}
\sup_{\alpha\in \FF_n^+} \frac{b_\alpha}{b_{g_i\alpha}}<\infty.
\end{equation}
We remark that  the operator $U:F^2(H_n)\to F^2(g)$,  defined by $U(e_\alpha):=\sqrt{b_\alpha}Z_\alpha$, $\alpha\in \FF_n^+$, is unitary and
 $UW_i=L_{Z_i}U, \  i\in \{1,\ldots,n\},
 $
 where ${\bf W}=(W_1,\ldots, W_n)$ is the universal model associated with $g$.
Similarly, we define the {\it right multiplication} operator  $R_{Z_i}:F^2(g)\to F^2(g)$  by   setting $R_{Z_i}\zeta:=\zeta Z_i$ for all $\zeta\in F^2(g)$. We also  use the notation $R_{Z_i}^{(g)}$ when needed.  Note that $R_{Z_i}$ is a  bounded operator  if and only if 
\begin{equation}\label{sup2}
\sup_{\alpha\in \FF_n^+} \frac{b_\alpha}{b_{\alpha g_i}}<\infty.
\end{equation}
We also have $U\Lambda_i=R_{Z_i}U$ for every $ i\in \{1,\ldots,n\}$, where
 ${\bf \Lambda}:=(\Lambda_1,\ldots, \Lambda_n)$ was defined in Section 2.

Let $\cE_1$ and $\cE_2$ be Hilbert spaces and let $f= 1+\sum_{ |\alpha|\geq 1} b_\alpha^\prime
Z_\alpha$ be  a  formal power series where the coefficients $\{b_\alpha'\}$ satisfy similar properties to  those satisfied by 
$\{b_\alpha\}$.
A formal power series  $\varphi=\sum_{\alpha\in \FF_n^+} Z_\alpha\otimes A_{(\alpha)}$, with operator-valued  coefficients $A_{(\alpha)}\in B(\cE_1,\cE_2)$, is called ({\it twisted}) {\it right multiplier} from $F^2(g)\otimes \cE_1$ to $ F^2(f)\otimes \cE_2$ and  we denote $\varphi\in \cM^r(F^2(g)\otimes \cE_1,F^2(f)\otimes \cE_2)$ if,  for any
$\zeta=\sum_{\beta\in \FF_n^+} Z_\beta \otimes h_{(\beta)}$ in $F^2(g)\otimes \cE_1$, we have
$$
\zeta\varphi:=\sum_{\gamma\in \FF_n^+}  Z_\gamma\otimes \sum_{\alpha,\beta\in\FF_n^+, \beta \alpha =\gamma} A_{(\alpha)}h_{(\beta)}\in F^2(f)\otimes \cE_2.
$$
Due to the closed graph theorem,  the right multiplication operator
$R_\varphi :F^2(g)\otimes \cE_1\to F^2(f)\otimes \cE_2$ defined by $R_\varphi \zeta:=\zeta\varphi$  is a  bounded  linear operator.   The (right) multiplier norm of $\varphi$ is  given by 
$\|\varphi\|^{\cM^r}:=\|R_\varphi\|$.  
  If $g=f$ and $\cE_1=\cE_2=\cE$, we use the notation
$\cM^r(F^2(f)\otimes \cE):=\cM^r(F^2(f)\otimes \cE,F^2(f)\otimes \cE)$. Note that  
$\cM^r(F^2(f)\otimes \cE)$ becomes a Banach algebra with respect to the multiplier norm.
Similarly, we  introduce the the set of   {\it left multipliers} from $F^2(g)\otimes \cE_1$ to $ F^2(f)\otimes \cE_2$ and denote  it by $ \cM^\ell(F^2(g)\otimes \cE_1,F^2(f)\otimes \cE_2)$. In this case, the left multiplication operator
$L_\varphi :F^2(g)\otimes \cE_1\to F^2(f)\otimes \cE_2$   is defined by $L_\varphi \zeta:=\varphi \zeta$   and it is a  bounded  linear operator. The (left) multiplier norm of $\varphi$ is  given by 
$\|\varphi\|^{\cM^\ell}:=\|L_\varphi\|$.

\begin{theorem} \label{right-mult}   Let  $f$ and $g$ be formal power series satisfying the conditions  corresponding to  relations \eqref{sup1} and \eqref{sup2}. Then   the following statements hold.
\begin{enumerate}
\item[(i)]
 A  bounded linear operator $X:F^2(g)\otimes \cE_1\to F^2(f)\otimes \cE_2$ satisfies the relation
$$
X(L_i^{(g)}\otimes I_{\cE_1})=(L_i^{(f)}\otimes I_{\cE_2})X, \qquad i\in \{1,\ldots, n\},
$$
if and only if  there is  $\varphi\in \cM^r(F^2(g)\otimes \cE_1,F^2(f)\otimes \cE_2)$ such that $X=R_\varphi$.
\item[(ii)] A bounded linear operator $X:F^2(g)\otimes \cE_1\to F^2(f)\otimes \cE_2$ satisfies the relation
$$
X(R_i^{(g)}\otimes I_{\cE_1})=(R_i^{(f)}\otimes I_{\cE_2})X, \qquad i\in \{1,\ldots, n\},
$$
if and only if  there is  $\varphi\in \cM^\ell(F^2(g)\otimes \cE_1,F^2(f)\otimes \cE_2)$ such that $X=L_\varphi$.
\end{enumerate}
\end{theorem}
 Since the proof  of the theorem is similar to   the proof  of Theorem 4.3 from \cite{Po-Bergman}, we omit it.

Let ${\bf W}^{(f)}:=(W_1^{(f)},\ldots, W_n^{(f)})$  and ${\bf W}^{(g)}:=(W_1^{(g)},\ldots, W_n^{(g)})$  be the   weighted left creation operators  associated with the formal power series $f$ and $g$, respectively.  We say that  a bounded linear   operator
$\widehat X : F^2(H_n)\otimes \cE_1\to F^2(H_n)\otimes \cE_2$ is multi-analytic with respect to    ${\bf W}^{(g)} $ and ${\bf W}^{(f)}$   if  
$$
\widehat X(W_i^{(g)}\otimes I_{\cE_1})=(W_i^{(f)}\otimes I_{\cE_2}) \widehat X,\qquad
i\in \{1,\ldots, n\}.
$$
If $g= 1+\sum_{ |\alpha|\geq 1} b_\alpha
Z_\alpha$,  we define the unitary operator
  $U_{g}:F^2(H_n)\to F^2(g)$    by 
$
U_{g}(e_\alpha):=\sqrt{b_\alpha} Z_\alpha$,  $\alpha\in \FF_n^+.
$
A consequence of  Theorem \ref{right-mult} is  the following 

\begin{corollary} \label{hatX}     A bounded linear   operator
$\widehat X : F^2(H_n)\otimes \cE_1\to F^2(H_n)\otimes \cE_2$ is multi-analytic with respect to   ${\bf W}^{(g)} $ and ${\bf W}^{(f)}$ if and only if
$$
\widehat X=(U_{f}^*\otimes I_{\cE_2}) R_\varphi (U_{g}\otimes I_{\cE_1})
$$
for some right multiplier $\varphi\in \cM^r(F^2(g)\otimes \cE_1,F^2(f)\otimes \cE_2)$.
\end{corollary}
 
 Combining the   results above  with Theorem  \ref{closure}   and  Proposition  \ref{multi},   one can easily deduce that following

\begin{theorem} \label{ident}  The right multiplier algebra $\cM^r(F^2(f)\otimes \cK)$ is completely isometrically isomorphic to the  noncommutative Hardy algebra $R^\infty(f)\bar\otimes_{min} B(\cK)$. Moreover,  the map
$\Psi: \cM^r(F^2(f)\otimes \cK)\to  R^\infty(f)\bar\otimes_{min} B(\cK)$ defined by
$$\Psi(\varphi):=(U_{f}^*\otimes I_{\cK}) R_\varphi (U_{f}\otimes I_{\cK}),\qquad 
\varphi\in \cM^r(F^2(f)\otimes \cE) 
$$
is a completely  isometric isomorphism.
\end{theorem}

It is quite clear that there is a  version of Theorem \ref{ident}   for the left  multiplier algebra $\cM^\ell(F^2(f)\otimes \cK)$.   We remark that,  Theorem
\ref{ident} can be used to deduce that  
$\cM^r(F^2(f)\otimes\cE_1,F^2(f)\otimes\cE_2)$ is  completely isometrically isomorphic to   $R^\infty(f)\bar\otimes_{min} B(\cE_1, \cE_2)$.

\section{Admissible free holomorphic functions for operator model theory}

 The goal of this section is to introduce
  several classes of admissible free holomorphic functions  for  operator model theory and the associated noncommutative domains. The examples presented  here will be refered to throughout the paper.

\begin{definition}  We say that  the noncommutative  set $\cD_{g^{-1}}$ is a regular domain if $g$ is a  free holomorphic function in a neighborhood of the origin   such that
$
g^{-1} =1+\sum_{\alpha\in \FF_n^+, |\alpha|\geq 1} a_\alpha Z_\alpha
$
for some coefficients $a_\alpha\leq  0$ and $a_{g_i}<0$ for every $i\in \{1,\ldots, n\}$. 
\end{definition}

\begin{example} \label{complete-Nevanlinna}  
Let $g= 1+\sum_{ |\alpha|\geq 1} b_\alpha
Z_\alpha$ be   a free holomorphic function in a neighborhood of the origin such that
$
g^{-1}=1+ \sum_{ |\alpha|\geq 1} a_\alpha
Z_\alpha,
$
where $a_\alpha\leq 0$ for $|\alpha|\geq 1$ and $a_{g_i}<0$ for every $i\in \{1,\ldots, n\}$.  A close look at the results from   \cite{Po-domains} reveals that $g$ is  an admissible free holomorphic function.   The regular domains  were  extensively studied in  \cite{Po-domains}. In this case, we have 
$$
b_\alpha b_\beta\leq  b_{\alpha\beta},\qquad \alpha,\beta\in \FF_n^+.
$$
We remark  that the converse is not true.  A weaker converse of this result  was considered in \cite{APo1} and will be extended  in Theorem  \ref{sub-cN}. 
\end{example}

The next result provides sufficient conditions on a free holomorphic functions $g$  to ensure  that $\cD_{g^{-1}}$  is a regular domain. This result  extends the corresponding one  from \cite{APo1} and the proof is similar.  

\begin{theorem} \label{sub-cN}  Let $g= 1+\sum_{ |\alpha|\geq 1} b_\alpha
Z_\alpha$ be a free holomorphic function in a neighborhood of the origin,

which satisfies   the conditions
\begin{enumerate}
\item[(i)] $b_{g_0}=1$, $b_\alpha>0$ for $|\alpha|\geq 1$;
\item[(ii)]   for any $\beta\in \FF_n^+$ and $i,j\in \{1,\ldots, n\}$,
$
\frac{b_{\beta g_i}}{b_{g_j \beta g_i}}\leq \frac{b_{\beta }}{b_{g_j \beta }}.
$
  \end{enumerate}
Then  $g$ is  an admissible free holomorphic function and  $\cD_{g^{-1}}$ is a regular domain. 
\end{theorem}

Let $g= 1+\sum_{ |\alpha|\geq 1} b_\alpha
Z_\alpha$  be  a formal power series  with $b_\alpha=b_{\beta}$ for any $\alpha,\beta \in \FF_n^+$ with $|\alpha|=|\beta|$  such that 
$b_0=1$ and $b_{\alpha}>0$ for every $\alpha\in \FF_n^+$. In this case, the Hilbert space
$$
F^2(g)=\left\{ \zeta=\sum_{\alpha\in \FF_n^+}c_\alpha Z_\alpha: \   \|\zeta\|_g^2:=\sum_{\alpha\in \FF_n^+}\frac{1}{b_{\alpha}} |c_\alpha|^2<\infty, \ c_\alpha\in \CC\right\}
$$
is called {\it unitarily invariant Hilbert space}.  If, in  addition,  
$
\lim_{k\to \infty} \frac{b_k}{b_{k+1}}=1,
$
 where $b_k:=b_{|\alpha|}$ for any $\alpha\in \FF_n^+$ with $|\alpha|=k$,  we call $F^2(g)$  {\it normalized unitarily invariant Hilbert space}.
  
  A simple consequence of Theorem \ref{sub-cN} is the following 
  \begin{corollary}  Let $g= 1+\sum_{ |\alpha|\geq 1} b_{|\alpha|}
Z_\alpha$ be a free holomorphic function in a neighborhood of the origin such that  $b_{|\alpha|}>0$  and  $\left\{\frac{b_{|\alpha|}}{b_{|\alpha|+1}}\right\}$ is a decreasing sequence.
   Then  $g$ is an admissible  free holomorphic function  and 
  $\cD_{g^{-1}}$  is a regular domain. 
  \end{corollary}

Some classes of  unitarily invariant Hilbert spaces are  given in the following examples.

\begin{example}  \label{Bergman}
 For each $s\in (0,\infty)$, consider  the formal power series
$$
g_s:=1+\sum_{k=1}^\infty\left(\begin{matrix} s+k-1 \\k \end{matrix}\right)(Z_1+\cdots  +Z_n)^k.
$$ 
The associated Hilbert  space $F^2(g_s)$ has orthonormal basis 
$\left\{ \sqrt{\left(\begin{matrix} s+k-1 \\k \end{matrix}\right) }Z_\alpha\right\}_{\alpha\in \FF_n^+}$. It is easy to see that $\frac{b_k}{b_{k+1}}=\frac{k+1}{s+k}$ and, consequently, $\lim_{k\to \infty} \frac{b_k}{b_{k+1}}=1$.
Note  that the sequence $\frac{b_k}{b_{k+1}}$ is decreasing if and only if $s\in (0,1]$. In this case, 
Theorem \ref{sub-cN} implies that  $\cD_{g_s^{-1}}$ is a regular domain.
We remark  that if $s=1$, then the corresponding subspace $F^2(g_1)$ is   the full Fock space on $n$ generators.
On the other hand, if $s=n$ (resp. $s=n+1$) we obtain the noncommutative Hardy (resp. Bergman)  space over the unit ball
 $[B(\cH)^n]_1$. In the particular case when $s\in \NN$, we obtain the noncommutative weighted Bergman space over $[B(\cH)^n]_1$, which was extensively studied in  \cite{Po-Berezin}, \cite{Po-domains}, and \cite{Po-Bergman}.
 The  scale of  generalized Bergman spaces    $\{F^2(g_s):  s\in(0,\infty)\}$ was never considered and studied  in the  multivariable noncommutative setting, except the case when $s\in \NN$.    We remark that when $s\in (0,1]$,   the Hilbert  spaces $F^2(g_s)$  are  noncommutative generalizations of the classical   Besov-Sobolev spaces on the unit ball $\BB_n\subset \CC^n$ with representing kernel $\frac{1}{(1-\left<z,w\right>)^s}$, $z,w\in \BB_n$.
 \end{example}

\begin{example} \label{Dirichlet}
Let $s\in \RR$ and let
$$
\xi_s:=1+\sum_{\alpha\in \FF_n^+, |\alpha|\geq 1} (|\alpha|+1)^s Z_\alpha.
$$
The associated Hilbert  space $F^2(\xi_s)$ has orthonormal basis $\left\{ \sqrt{(|\alpha|+1)^s}Z_\alpha\right\}_{\alpha\in \FF_n^+}$. Since   
$$\frac{b_k}{b_{k+1}}=\frac{(k+1)^s}{(k+2)^s}\to 1,\quad \text{as }\ k\to \infty,
$$
$F^2(\xi_s)$ is  a normalized  unitarily invariant Hilbert space.
 On the other hand,  it is clear that
the sequence $\frac{b_k}{b_{k+1}}$ is decreasing if and only if $s\leq 0$. In this case, Theorem \ref{sub-cN}  shows that  $\cD_{\xi_s^{-1}}$ is a regular domain. We remark that the scale of spaces $\{F^2(\xi_s):  s\in \RR\}$ contains  the noncommutative Dirichlet space  $F^2(\xi_{-1})$over the unit ball
 $[B(\cH)^n]_1$. When $s=0$, we recover again the full Fock space with $n$ generators.
\end{example}

\begin{theorem}  \label{general}  Let    $g= 1+\sum_{ |\alpha|\geq 1} b_\alpha
Z_\alpha$   be  a free holomorphic function in a neighborhood of the origin        such that
the following conditions are satisfied.
\begin{enumerate}
\item[(a)]  $b_\alpha>0$  and   
$\sup_{\alpha\in \FF_n^+} \frac{b_\alpha}{b_{g_i \alpha}}<\infty$  for every $i\in \{1,\ldots, n\}$.
\item[(b)]    If $g^{-1}=\sum_{\alpha\in \FF_n^+} a_\alpha Z_\alpha $ is the inverse of $g$,  there is $N_0\geq 2$ such that $a_\alpha\geq 0$ $($or $a_\alpha\leq 0) $ for every $\alpha\in \FF_n^+$ with $|\alpha|\geq N_0$.
\end{enumerate}
Then $g$ is an admissible free holomorphic function.
\end{theorem}

\begin{proof} Assume that $a_\alpha\leq 0$  for every $\alpha\in \FF_n^+$ with $|\alpha|\geq N_0$.
 We prove that there is $C>0$ such that
 \begin{equation}
 \label{C}
 \left| \sum_{{0\leq |\beta|\leq N}\atop{\beta\gamma=\alpha}}\frac{a_\beta b_\gamma}{b_\alpha} 
\right|<C
\end{equation}
for any $N>N_0$ and  any $\alpha \in \FF_n^+$ with $|\alpha|\geq N_0$.
Due to Lemma \ref{lem1}, we have 
 $a_{g_0}=1$ and $a_{g_i}=-b_{g_i}<0$ for $i\in \{1,\ldots, n\}$.
Given $\alpha, \beta\in \FF_n^+$,  we  say that $\beta\leq\alpha$   if $\alpha=\beta\gamma$ for some $\gamma\in \FF_n^+$. In this case, we set $\alpha\backslash \beta:=\gamma$.
According to relation \eqref{sum1}, we have 
$
\sum_{ \beta\in \FF_n^+, \beta\leq\alpha}a_\beta b_{\alpha\backslash \beta}=0$
  if  $ |\alpha|\geq 1.
$
Consequently,
\begin{equation}\label{sums}
\sum_{ \beta\in \FF_n^+, \beta\leq\alpha, |\beta|\leq N}a_\beta b_{\alpha\backslash \beta} +
\sum_{ \beta\in \FF_n^+, \beta\leq\alpha, |\beta|>N}a_\beta b_{\alpha\backslash \beta}=0.
\end{equation}
Since  $a_\beta\leq 0$  for every $\beta\in \FF_n^+$ with $|\beta|\geq N>N_0$, the second sum in the relation above is negative.  Now, relation \eqref{sums} implies 
$
\sum_{ \beta\in \FF_n^+, \beta\leq\alpha, |\beta|\leq N}a_\beta b_{\alpha\backslash \beta}\geq 0
$
for any  $\alpha\in \FF_n^+$ with $|\alpha|>N_0$ and $N>N_0$. Using the later relation and the fact that
$\sum_{ \beta\in \FF_n^+, \beta\leq\alpha, |\beta|> N_0}\frac{a_\beta b_{\alpha\backslash \beta}}{b_\alpha}\leq 0,
$
we deduce that
\begin{equation*}
\begin{split}
\left| \sum_{ \beta\in \FF_n^+, \beta\leq\alpha, |\beta|\leq N}\frac{a_\beta b_{\alpha\backslash \beta}}{b_\alpha} 
\right| &\leq 
\sum_{ \beta\in \FF_n^+, \beta\leq\alpha, |\beta|\leq N_0}\frac{a_\beta b_{\alpha\backslash \beta}}{b_\alpha}\\
&\leq  \left(\sup_{{\beta\in \FF_n^+, \beta\leq\alpha, |\beta|\leq N_0}\atop{\alpha\in \FF_n^+}}
\frac{b_{\alpha\backslash \beta}}{b_\alpha}\right)
 \left(\sum_{ \beta\in \FF_n^+, |\beta|\leq N_0}|a_\beta|\right).
\end{split}
\end{equation*}
On the other hand, since $\sup_{\alpha\in \FF_n^+} \frac{b_\alpha}{b_{g_i \alpha}}<\infty$  for   $i\in \{1,\ldots, n\}$, we obtain
$\sup_{{\beta\in \FF_n^+, \beta\leq\alpha, |\beta|\leq N_0}\atop{\alpha\in \FF_n^+}}
\frac{b_{\alpha\backslash \beta}}{b_\alpha}<\infty.
$
Indeed,  let $\beta=g_{i_1}\cdots g_{i_k}\in \FF_n^+$  be such that $k\leq N_0$ and let $\alpha=\beta\gamma$, where  $\gamma\in \FF_n^+$.  Let $M>0$ be such that $\sup _{\alpha\in \FF_n^+} \frac{b_\alpha}{b_{g_i \alpha}}<M$ for any $i\in \{1,\ldots, n\}$. Note that 
$$
\frac{b_{\alpha\backslash \beta}}{b_\alpha} =\frac{b_{\gamma}}{b_{g_{i_k}\gamma}}
\frac{b_{g_{i_k}\gamma}}{b_{g_{i_{k-1}}g_{i_k}\gamma}}\cdots
\frac{b_{g_{i_2}\cdots g_{i+k}\gamma}}{b_{g_{i_1}\cdots g_{i+k}\gamma}}<M^{|\beta|}
$$
for any  $\gamma\in \FF_n^+$ and  any $\beta\in \FF_n^+$ with $|\beta|\leq N_0$.
Consequently, $\sup_{{\beta\in \FF_n^+, \beta\leq\alpha, |\beta|\leq N_0}\atop{\alpha\in \FF_n^+}}
\frac{b_{\alpha\backslash \beta}}{b_\alpha}<\infty$ and relation \eqref{C} holds.
Using Proposition \ref{abb}, we  deduce that $g$ is an admissible free holomorphic function.
The case when  $a_\alpha\geq 0$    for every $\alpha\in \FF_n^+$ with $|\alpha|\geq N_0$, can be treated in a similar manner. The proof is complete.
 \end{proof}

\begin{example} \label{Bergman-cont} For each $s\in (0,\infty)$, consider  the formal power series
$$
g_s=1+\sum_{k=1}^\infty\left(\begin{matrix} s+k-1 \\k \end{matrix}\right)(Z_1+\cdots +Z_n)^k.
$$ 
and the  associated Hilbert  space $F^2(g_s)$ (see Example \ref{Bergman}).  In this case, we have
$$
g_s^{-1}=\left[1-(Z_1+\cdots +Z_n)\right]^s=\sum_{k=0}^\infty (-1)^k\left(\begin{matrix} s \\k \end{matrix}\right)(Z_1+\cdots + Z_n)^k.
$$
Therefore 
$$a_\alpha=(-1)^{|\alpha|} \left(\begin{matrix} s \\|\alpha| \end{matrix}\right)=(-1)^{|\alpha|}\frac{s(s-1)\cdots (s-|\alpha|+1)}{|\alpha| !},\qquad \alpha\in \FF_n^+, |\alpha|\geq 1.
$$
Note that there is $N_s\geq 2$ such that $a_\alpha\geq 0$ $($or $a_\alpha\leq 0)$ for every $\alpha\in \FF_n^+$ with $|\alpha|>N_s$.
Applying Theorem \ref{general}, we deduce that $g_s$ is an admissible free holomorphic function.
\end{example}

\begin{proposition}  Let  $s\in [1,\infty)$ and let $\varphi=\sum_{|\alpha|\geq 1} d_\alpha Z_\alpha$ be a formal power series such that $d_\alpha \geq 0$ and $d_{g_i}>0$ for $i\in \{1,\ldots, n\}$.  Consider   the  formal power series $\psi_s$  with inverse 
$$\psi_s^{-1}=(1-\varphi)^s:= \sum_{k=0}^\infty (-1)^k\left(\begin{matrix} s \\k \end{matrix}\right) \varphi^k=1+\sum_{|\alpha|\geq 1}a_\alpha  Z_\alpha
$$
and let $\psi_s=1+\sum_{|\alpha|\geq 1}b_\alpha  Z_\alpha$. Then  the coefficients $\{b_\alpha\}$ satisfy the inequalities
$
\frac{b_\alpha}{b_{g_i\alpha}}\leq \frac{1}{d_{g_i}} $ and $ \frac{b_\alpha}{b_{\alpha g_i}}\leq \frac{1}{d_{g_i}}
$
for any $\alpha\in \FF_n^+$ and $i\in \{1,\ldots, n\}$.
Moreover,  if $\varphi$  is a noncommutative polynomial, then 
 $\psi_s$ is an admissible free holomorphic function for every  $s\geq 1$.
\begin{proof}
Note that 
\begin{equation*}
\begin{split}
\psi_s &= 1+\sum_{k=1}^\infty\left(\begin{matrix} s+k-1 \\k \end{matrix}\right) \varphi^k
=1+\sum_{\alpha\in \FF_n^+, |\alpha|\geq 1}\left(\sum_{j=1}^{|\alpha|} \left(\begin{matrix} s+j-1 \\j \end{matrix}\right)
\sum_{{\gamma_1\cdots \gamma_j=\alpha }\atop {|\gamma_1|\geq 1,\ldots,
 |\gamma_j|\geq 1}} d_{\gamma_1}\cdots d_{\gamma_j} \right)  
  Z_\alpha.
\end{split}
\end{equation*}
Hence,  we deduce that
\begin{equation}\label{b_alpha}
b_{g_0}=1 \quad \text{ and }\quad b_\alpha= \sum_{j=1}^{|\alpha|}  \left(\begin{matrix} s+j-1 \\j \end{matrix}\right)
\sum_{{\gamma_1\cdots \gamma_j=\alpha }\atop {|\gamma_1|\geq
1,\ldots, |\gamma_j|\geq 1}} d_{\gamma_1}\cdots d_{\gamma_j}   \quad
\text{ if } \ |\alpha|\geq 1.
\end{equation}
Assume that $s\geq 1$ and note that $ \left(\begin{matrix} s+j \\j+1\end{matrix}\right)\geq  \left(\begin{matrix} s+j-1 \\j \end{matrix}\right)$ for any  $j\geq 1$. Using relation \eqref{b_alpha}, for any  $\alpha\in \FF_n^+$ with $|\alpha|\geq 1$,  we obtain
\begin{equation*}
\begin{split}
b_{g_i\alpha}&= \sum_{t=1}^{|\alpha|+1}  \left(\begin{matrix} s+t-1 \\t\end{matrix}\right)
\sum_{{\gamma_1\cdots \gamma_t=g_i\alpha }\atop {|\gamma_1|\geq
1,\ldots, |\gamma_t|\geq 1}} d_{\gamma_1}\cdots d_{\gamma_t}
\geq 
\sum_{t=2}^{|\alpha|+1}  \left(\begin{matrix} s+t-1 \\t\end{matrix}\right)
\sum_{{g_i\gamma_2\cdots \gamma_t=g_i\alpha }\atop {|\gamma_2|\geq
1,\ldots, |\gamma_t|\geq 1}} d_{g_i}d_{\gamma_2}\cdots d_{\gamma_t}\\
&=
\sum_{j=1}^{|\alpha|}  \left(\begin{matrix} s+j \\j+1\end{matrix}\right)
\sum_{{g_i\gamma_1\cdots \gamma_j=g_i\alpha }\atop {|\gamma_1|\geq
1,\ldots, |\gamma_j|\geq 1}} d_{g_i}d_{\gamma_1}\cdots d_{\gamma_j}
=
d_{g_i} \left(\sum_{j=1}^{|\alpha|}  \left(\begin{matrix} s+j \\j+1\end{matrix}\right)
\sum_{{\gamma_1\cdots \gamma_j=\alpha }\atop {|\gamma_1|\geq
1,\ldots, |\gamma_j|\geq 1}} d_{\gamma_1}\cdots d_{\gamma_j}\right)\\
&= d_{g_i} b_\alpha.
\end{split}
\end{equation*}
 Note also that $b_{g_i}= s d_{g_i}$ for every $i\in \{1,\ldots, n\}$. Consequently, we have
 $\frac{b_\alpha}{b_{g_i\alpha}}\leq \frac{1}{d_{g_i}}$ for any $\alpha\in \FF_n^+$.
We remark that, we can show in a similar manner  that $\frac{b_\alpha}{b_{\alpha g_i}}\leq \frac{1}{d_{g_i}}$ for any $\alpha\in \FF_n^+$.

Now, assume that $\varphi$  is a noncommutative polynomial. Note that there is $N_s\geq 2$ such that $a_\alpha\geq 0$ $($or $a_\alpha\leq 0)$ for every $\alpha\in \FF_n^+$ with $|\alpha|>N_s$.  
  Theorem \ref{general} shows that      $\psi_s$ is an admissible free holomorphic function whenever $s\geq 1$.  The proof is complete.
\end{proof}
\end{proposition}

\section{$C^*$-algebras associated with noncommutative domains and classification}

We associate with each universal model ${\bf W}=(W_1,\ldots, W_n)$ the $C^*$-algebra  $C^*({\bf W})$ generated by $W_1,\ldots, W_n$ and identity. The main goal of this section is to obtain  Wold decompositions for any unital $*$-representations of $C^*({\bf W})$. Under certain conditions on ${\bf W}$, we obtain some classification results and  an exact sequence of $C^*$-algebras, generalizing the corresponding results obtained by Coburn and Cuntz. All the results of this section hold, in particular,  for  the admissible free holomorphic functions provided by   Example \ref{scale}.

\begin{proposition}\label{norm-conv}  Let $g= 1+\sum_{ |\alpha|\geq 1} b_\alpha
Z_\alpha$ be a free holomorphic function in a neighborhood of the origin such that  $b_\alpha>0$ and 
$\sup_{\alpha\in \FF_n^+} \frac{b_\alpha}{b_{g_i \alpha}}<\infty$ for each $i\in \{1,\ldots, n\}$.
 If $g^{-1}=1+\sum_{ |\gamma|\geq 1} a_\gamma
Z_\gamma$, then the following statements are equivalent.
\begin{enumerate}
\item[(i)] The series $\sum_{k=0}^\infty\sum_{|\alpha|=k} a_\alpha W_\alpha W_\alpha^*$ converges in the operator norm topology, where  $(W_1,\ldots, W_n)$ are  the  weighted left creation operators associated with $g$.

\item[(ii)]  If  $q,N\in \NN$ with $q\leq N$, then 
$$\lim_{{q\to\infty}\atop{N\to \infty}}\sup_{\alpha\in \FF_n^+, |\alpha|\geq q} 
\left|  \sum\limits_{{\beta\gamma=\alpha}\atop { q\leq |\beta|\leq N}}
\frac{a_\beta b_\gamma}{b_\alpha}
 \right|=0.
 $$
\end{enumerate}

\end{proposition}
\begin{proof} Since 
$
W_\beta W_\beta^* e_\alpha =\begin{cases}
\frac {{b_\gamma}}{{b_{\alpha}}} e_\alpha & \text{ if } \alpha=\beta\gamma\\
0& \text{ otherwise}
\end{cases}
$,
the operator  $\sum_{q\leq|\beta|\leq N} a_\beta W_\beta  W_\beta^*$ is diagonal and 
$$
\left\|\sum_{q\leq|\beta|\leq N} a_\beta W_\beta  W_\beta^*\right\|=\sup_{\alpha\in \FF_n^+, |\alpha|\geq q} 
\left|  \sum\limits_{{\beta\gamma=\alpha}\atop { q\leq |\beta|\leq N}}
\frac{a_\beta b_\gamma}{b_\alpha}
 \right|
$$
for any  $q,N\in \NN$ with $q\leq N$.
Now, one can easily complete the proof.
\end{proof}

Note that the conditions in Proposition \ref{norm-conv}  are satisfied, in particular,   when $g^{-1}$ is a polynomial.

\begin{proposition} \label{Ex}  Let $g= 1+\sum_{ |\alpha|\geq 1} b_\alpha
Z_\alpha$ be a free holomorphic function in a neighborhood of the origin such that 
$\sup_{\gamma, \beta}\frac{b_\gamma}{b_{\beta \gamma}}<\infty$.
 If $g^{-1}=1+\sum_{ |\gamma|\geq 1} a_\gamma
Z_\gamma$ 
and 
$
\sum_{j=1}^\infty \max\{|a_\alpha|:\  |\alpha|=j\}
$
is convergent, then  the series $\sum_{\alpha\in \FF_n^+} a_\alpha W_\alpha W_\alpha^*$ converges in the operator norm topology, where   $(W_1,\ldots, W_n)$ are  the corresponding  weighted left creation operators associated with $g$.   In particular, the result holds  if $g^{-1}$ is a polynomial.
\end{proposition}

\begin{proof} Fix $q,N\in \NN$ with $q\leq N$. For each $\alpha\in \FF_n^+$ with $|\alpha|\geq q$ and each
$p\in \NN$ such that $q\leq p\leq N$, there is at most one  word $\beta\in \FF_n^+$ such that  $|\beta|=p$ and $\beta\gamma=\alpha$. Consequently,  setting 
$M:=\sup_{\gamma, \beta}\frac{b_\gamma}{b_{\beta \gamma}}<\infty$, we have
$$
\left|  \sum\limits_{{\beta\gamma=\alpha}\atop { q\leq |\beta|\leq N}}
\frac{a_\beta b_\gamma}{b_\alpha}
 \right|
 \leq M\sum_{j=q}^N \max\{|a_\alpha|:\  |\alpha|=j\}
$$
for any  $\alpha\in \FF_n^+$ with $|\alpha|\geq q$.   Hence, we deduce that
$$\sup_{\alpha\in \FF_n^+, |\alpha|\geq q} 
\left|  \sum\limits_{{\beta\gamma=\alpha}\atop { q\leq |\beta|\leq N}}
\frac{a_\beta b_\gamma}{b_\alpha}
 \right|
\leq  M\sum_{j=q}^N \max\{|a_\alpha|:\  |\alpha|=j\}.
$$
Since the series $
\sum_{j=1}^\infty \max\{|a_\alpha|:\  |\alpha|=j\}
$
is convergent, we use  Proposition \ref{norm-conv} to complete the proof.
\end{proof}

\begin{example} \label{scale}   For each $s\in [1,\infty)$, consider  the formal power series
$$
g_s:=1+\sum_{k=1}^\infty\left(\begin{matrix} s+k-1 \\k \end{matrix}\right)(Z_1+\cdots +Z_n)^k.
$$ 
If $\alpha\in \FF_n^+$ and $|\alpha|=k$, then $b_k:=b_\alpha=\left(\begin{matrix} s+k-1 \\k \end{matrix}\right)$ and $\frac{b_k}{b_{k+1}}=\frac{k+1}{s+k}\leq 1$. Consequently,  we obtain $\sup_{\gamma, \beta}\frac{b_\gamma}{b_{\beta \gamma}}\leq 1$. On the other hand, 
since 
$$
g_s^{-1}=\left[1-(Z_1+\cdots +Z_n)\right]^s=\sum_{k=0}^\infty (-1)^k\left(\begin{matrix} s \\k \end{matrix}\right)(Z_1+\cdots + Z_n)^k,
$$
we have $a_\alpha=(-1)^k\left(\begin{matrix} s \\k \end{matrix}\right)$ if $|\alpha|=k$.  Since 
$
\sum_{k=0}^\infty \left|\left(\begin{matrix} s \\k \end{matrix}\right)\right|<\infty \text{ for any} \ s>0,
$
 we can apply Proposition \ref{Ex}, 
  to deduce that  the series $\sum_{\alpha\in \FF_n^+} a_\alpha W_\alpha W_\alpha^*$ converges in the operator norm topology.
\end{example}

We denote by $C^*({\bf W})$ the $C^*$-algebra generated by $W_1,\ldots W_n$ and the  identity. 
We also consider the lexicographic order on the unital free semigroup   $\FF_n^+$.

\begin{theorem}  \label{irreducible} Let $g= 1+\sum_{ |\alpha|\geq 1} b_\alpha
Z_\alpha$ be a free holomorphic function in a neighborhood of the origin such that  $b_\alpha>0$ and 
$
\sup_{\alpha\in \FF_n^+} \frac{b_\alpha}{b_{g_i \alpha}}<\infty$ for every  $i\in \{1,\ldots, n\},
$ 
and 
let   ${\bf W}=(W_1,\ldots, W_n)$ be  the   weighted left creation operators associated with $g$.
Then  the following statements hold.
\begin{enumerate}
\item[(i)] The $C^*$-algebra $C^*({\bf W})$ is irreducible.
\item[(ii)] The $n$-tuples $(W_1\otimes I_\cH,\ldots,
W_n\otimes I_\cH)$ and  $(W_1\otimes I_{\cH'},\ldots, W_n\otimes
I_{\cH'})$ 
  are unitarily equivalent if and only if   $\dim \cH=\dim \cH'$.
\item[(iii)] If  there is $i\in \{1,\ldots, n\}$ such that
$$
\lim_{\gamma\to \infty} \left(\frac{b_{g_p\gamma}}{b_{g_ig_p\gamma}}-\frac{b_\gamma}{b_{g_p\gamma}}\right) =0\quad  \text{for any} \quad  p\in \{1,\ldots, n\},
$$
then  the $C^*$-algebra $C^*({\bf W})$ contains all the compact operators in $B(F^2(H_n))$.  \end{enumerate}
\end{theorem}
\begin{proof}
Let $A\in B(F^2(H_n))$  be  commuting with each operator in $C^*({\bf W})$.
Since $A(1)\in F^2(H_n)$, we have
$A(1)=\sum_{\beta\in \FF_n^+} c_{
\beta}\frac{1}{\sqrt{b_{ \beta}}} e_{ \beta}$ for some
coefficients $\{c_\beta\}_{\FF_n^+}\subset \CC$ with $\sum_{\beta\in\FF_n^+}
|c_\beta|^2 \frac{1}{b_\beta}<\infty$. On the other hand, since $
AW_i=W_iA$ for any  $i\in \{1,\ldots,n\}$, relation \eqref{WbWb}
  implies
\begin{equation*}
\begin{split}
Ae_\alpha &=\sqrt{b_\alpha}AW_{\alpha}(1)=\sqrt{b_\alpha}W_{\alpha}
A(1)
=\sum_{\beta\in \FF_n^+} c_{ \beta}
\frac{\sqrt{b_\alpha}}{\sqrt{b_{ \alpha\beta}}} e_{
 \alpha\beta}. 
 \end{split}
\end{equation*}
Similarly, since  $A^*W_i=W_iA^*$ for  any $i\in\{1,\ldots,n\}$, there are  some scalars
 $\{c_\gamma'\}_{\gamma\in \FF_n^+}\subset \CC$ with $\sum_{\gamma\in\FF_n^+}
|c_\gamma' |^2 \frac{1}{b_\gamma}<\infty$   such that 
$$
A^* e_\sigma=\sum_{\gamma\in \FF_n^+} c_{\gamma}'
\frac{\sqrt{b_\sigma}}{\sqrt{b_{ \sigma\gamma}}} e_{
 \sigma\gamma}.
 $$
Since $\left< Ae_\alpha, e_\sigma\right>=\left< e_\alpha, A^* e_\sigma\right>$, the relations above imply
$$
\left< \sum_{\beta\in \FF_n^+} c_{ \beta}
\frac{\sqrt{b_\alpha}}{\sqrt{b_{ \alpha\beta}}} e_{
 \alpha\beta}, e_\sigma\right>=\left< e_\alpha, \sum_{\gamma\in \FF_n^+} c_{\gamma}'
\frac{\sqrt{b_\sigma}}{\sqrt{b_{ \sigma\gamma}}} e_{
 \sigma\gamma}\right>
$$
for any $\alpha,\sigma\in \FF_n^+$.
Assume that $\sigma=\alpha\beta_0$ where $\beta_0\in \FF_n^+$ and $|\beta_0|\geq 1$.
The relation above becomes
$$
c_{\beta_0}=\left< e_\alpha, \sum_{\gamma\in \FF_n^+} c_{\gamma}'
\frac{\sqrt{b_{\alpha\beta_0}}}{\sqrt{b_{ \alpha\beta_0\gamma}}} e_{
 \alpha\beta_0\gamma}\right>=0.
$$
Hence $c_{\beta_0}=0$ for any  $\beta_0\in \FF_n^+$ with $|\beta_0|\geq 1$.
Therefore, $Ae_\alpha=c_{g_0}e_\alpha$ for any $\alpha\in \FF_n^+$, thus $A=c_{g_0}I$. This shows that 
$C^*({\bf W})$ is irreducible.

 To prove  item (ii), note that 
 one implication is trivial. Assume that $U:F^2(H_n)\otimes \cH\to
F^2(H_n)\otimes \cH'$ is a unitary operator such that
$
U(W_i\otimes I_\cH)=(W_i\otimes I_{\cH'})U$ for  every $i\in\{1,\ldots, n\}$.
Then, we have
$
U(W_i^*\otimes I_\cH)=(W_i^*\otimes I_{\cH'})U$ for  $i\in \{1,\ldots, n\}$,
and due to the fact that the $C^*$-algebra $C^*({\bf W})$ is irreducible, we
must have $U=I_{F^2(H_n)}\otimes A$, where $A\in B(\cH,\cH')$ is a
unitary operator . Therefore, $\dim \cH=\dim \cH'$.
To prove item (iii), note that
\begin{equation*}
  W_iW_i^*
e_\alpha =\begin{cases} \frac
{b_\gamma}{b_{g_i\gamma}}\,e_{g_i\gamma}& \text{ if
}
\alpha=g_i\gamma, \ \, \gamma\in \FF_{n}^+ \\
0& \text{ otherwise,}
\end{cases}
\end{equation*}
and  $W_i^*W_i e_\alpha=  \frac{b_\alpha}{b_{g_i\alpha}} e_\alpha$ for all $\alpha\in \FF_n^+$. Hence, we deduce that  
  $W_i^*W_i-\sum_{j=1}^n W_jW_j^*$ is  a diagonal operator and
$$\left( W_i^*W_i-\sum_{j=1}^n W_jW_j^*\right)
e_\alpha=\left(\frac
{b_\alpha}{b_{g_i\alpha}}-\frac
{b_\gamma}{b_{g_p\gamma}}\right)\,e_{\alpha}$$
for any $\alpha\in \FF_n^+$ with $\alpha=g_p\gamma$ for some $\gamma\in \FF_n^+$ and $p\in \{1,\ldots,n \}$.
Since
$$
\lim_{\gamma\to \infty}\left(\frac{b_{g_p\gamma}}{b_{g_ig_p\gamma}}-
\frac{b_\gamma}{b_{g_p\gamma}}\right)=0,\quad \text{for any} \quad p\in \{1,\ldots, n\},
$$
it is clear that the diagonal operator $W_i^*W_i-\sum_{j=1}^n W_jW_j^*$ is compact. 
Since the $C^*$-algebra $C^*({\bf W})$ is irreducible, we conclude that it  contains all the compact operators in $B(F^2(H_n))$.
This completes the proof.
\end{proof}

We remark that Theorem \ref{irreducible} holds, in particular, for  the admissible free holomorphic functions $g$ such that $F^2(g)$ is a normalized unitarily invariant Hilbert space. In particular, it holds for  the free holomorphic functions provided by Example \ref{Bergman} and Example \ref{Dirichlet}.

\begin{theorem} \label{compact}  Let $g= 1+\sum_{ |\alpha|\geq 1} b_\alpha
Z_\alpha$ be a free holomorphic function in a neighborhood of the origin such that  $b_\alpha>0$ and 
$
\sup_{\alpha\in \FF_n^+} \frac{b_\alpha}{b_{g_i \alpha}}<\infty$   for every $i\in \{1,\ldots, n\},
$ 
and 
let  ${\bf W}=(W_1,\ldots, W_n)$ be  the  weighted left creation operators associated with $g$.
  Assume that   the series $\sum_{k=0}^\infty\sum_{|\alpha|=k} a_\alpha W_\alpha W_\alpha^*$ is convergent in the operator norm topology, where $g^{-1}=1+\sum_{|\alpha|\geq 1} a_\alpha Z_\alpha$. Then  the following statements hold.
  \begin{enumerate}
  \item[(i)]
$\boldsymbol\cK(F^2(H_n))\subset \overline{\text{\rm span}}\{W_\alpha W_\beta^*:\ \alpha,\beta\in \FF_n^+\},
$
 where  $\boldsymbol\cK(F^2(H_n))$ stands for the ideal of all compact
operators  in $B(F^2(H_n))$. 
\item[(ii)] There is a unique minimal nontrivial  two-sided ideal  $\cJ$ of $C^*({\bf W})$. Moreover, 
$\cJ$ coincides  with $\boldsymbol\cK(F^2(H_n))$.

\item[(iii)] If, in addition,
$$
\lim_{\gamma \to \infty} \left(\frac{b_{g_p\gamma}}{b_{g_ig_p\gamma}}-\frac{b_\gamma}{b_{g_p\gamma}}\right) =0
$$
for every $i,p\in \{1,\ldots, n\}$,  then 
$$
C^*({\bf W})=\overline{\text{\rm span}}\{W_\alpha W_\beta^*:\ \alpha,\beta\in \FF_n^+\}.
$$
\end{enumerate}
\end{theorem}
\begin{proof} 
 Since $\sum_{k=0}^\infty \sum_{|\alpha|=k} a_\alpha W_\alpha W_\alpha^*$ is convergent in the operator norm topology, Proposition \ref{W} implies  
\begin{equation*}
 \sum_{k=0}^\infty\sum_{|\alpha|= k} a_\alpha W_\alpha W_\alpha^*
 =P_{\CC 1}\in  \overline{\text{\rm span}}\{W_\alpha W_\beta^*:\ \alpha,\beta\in \FF_n^+\},
\end{equation*}
where $ P_{\CC1}$ is the orthogonal projection of $F^2(H_n)$ onto
$\CC$.
For any  polynomial 
$g({\bf W}):=\sum\limits_{|\alpha|\leq m} d_\alpha W_\alpha$
 and   any $ \xi:=\sum\limits_{\beta\in \FF_n^+} c_\beta
e_\beta\in  F^2(H_n)$, we have
 $
P_{\CC1} g({\bf W})^*\xi = \left<
\xi,g({\bf W})(1)\right>
$
and, consequently,
\begin{equation}\label{rankone}
q({\bf W})P_{\CC1} g({\bf W})^*\xi= \left<
\xi,g({\bf W})(1)\right>q({\bf W})(1)
\end{equation}
for any polynomial $q({\bf W})$.
Hence, we deduce that  $q({\bf W})P_{\CC1}
g({\bf W})^*$
 Is  a rank-one  operator  which belongs to 
the operator space $\overline{\text{\rm span}}\{W_\alpha W_\beta^*:\
\alpha,\beta\in \FF_n^+\}$. Using the definition of the weighted left creation operators $W_1,\ldots, W_n$, one can see that  the   set
$$\cE:=\left\{\left(\sum\limits_{|\alpha|\leq m}d_\alpha
W_\alpha\right)(1):\ m\in \NN, d_\alpha\in \CC\right\}$$ is  dense
in $F^2(H_n)$, Consequently,    relation \eqref{rankone} shows   that all compact
operators  in $B(F^2(H_n))$ are included in the operator space
$\overline{\text{\rm span}}\{W_\alpha W_\beta^*:\ \alpha,\beta\in
\FF_n^+\}$. This poves part (i).

To prove item (ii),  let  $\cJ$ be  a nontrivial  two-sided ideal   of $C^*({\bf W})$. Then there exists 
$A \in \cJ$ such that $A e_\alpha\neq 0$ for some $\alpha\in \FF_n^+$. 
Note that
 \begin{equation*}
P_{\CC 1} W_\beta^* e_\alpha =\begin{cases}
\frac {1}{\sqrt{b_{\beta}}}   & \text{ if } \alpha=\beta\\
0& \text{ otherwise,}
\end{cases}
\end{equation*}
 implies 
$ b_\alpha W_\alpha P_{\CC 1}
W_\alpha^*= P_{\CC e_\alpha}\in \overline{\text{\rm span}}\{W_\alpha W_\beta^*:\ \alpha,\beta\in \FF_n^+\}$,
where $P_{\CC e_\alpha}$ is the orthogonal projection of  the full Fock space $F^2(H_n)$ onto the one-dimensional subspace
 $\CC e_\beta$.   This implies that
$
 P_{\CC e_\alpha}A^*A P_{\CC e_\alpha}\in \cJ.
 $
Note that, for any $x\in F^2(H_n)$, we have
\begin{equation*}
\begin{split}
 P_{\CC e_\alpha}A^*A P_{\CC e_\alpha}x
 &=\left< x, P_{\CC e_\alpha}A^*A e_\alpha\right>e_\alpha=\|A e_\alpha\|^2 P_{\CC e_\alpha} x.
 \end{split}
\end{equation*}
Since $\|A e_\alpha\|^2 P_{\CC e_\alpha}=P_{\CC e_\alpha}A^*\Gamma P_{\CC e_\alpha}$ and $A e_\alpha\neq 0$, we deduce that $P_{\CC e_\alpha}\in\cJ$. Now, it is clear  that $W_\alpha^*P_{\CC e_\alpha} W_\alpha\in \cJ$  and, consequently, 
    $b_\alpha W_\alpha^* P_{\CC e_\alpha}
W_\alpha=P_{\CC 1}\in \cJ$. As in the proof of item (i), one can show that $q({\bf W})P_{\CC 1}
g({\bf W})^*\in \cJ$
 is  a rank-one  operator  for any polynomials  $q({\bf W})$ and $g({\bf W})$. Moreover, this can be used to show that
   any rank-one operator in $B(F^2(H_n))$ is in the ideal $\cJ$. Consequently,  we deduce that
   $\boldsymbol\cK(F^2(H_n))\subset \cJ$. Since $\cJ$ is a minimal  two-sided ideal of $C^*({\bf W})$, we must have $\cJ=\boldsymbol\cK(F^2(H_n))$.

Now, we prove item (iii).
We saw in the proof of Theorem \ref{irreducible}, part (iii), that   if 
$$
\lim_{\gamma\to \infty}\left(\frac{b_{g_p\gamma}}{b_{g_ig_p\gamma}}-
\frac{b_\gamma}{b_{g_p\gamma}}\right)=0\quad  \text{for every}\quad i, p\in \{1,\ldots, n\},
$$
then $W_i^*W_i-\sum_{j=1}^n W_jW_j^*$ is a compact operator in  $B(F^2(H_n))$. 
Consequently, using the fact that   $W_i^*W_j=0$ if $i, j\in \{1,\ldots, n\}$ with $i\neq j$, we deduce that\begin{equation*}
W_i^*W_iW_{i_1}\cdots W_{i_p}\in \sum_{j=1}^nW_{i_1}\cdots W_{i_p}W_jW_j^* + \boldsymbol\cK(F^2(H_n))
\end{equation*}
  for  any 
 $i, i_1,\ldots, i_p$ in $\{1,\ldots, n\}$. 
Inductively, one can easily prove that
$$
W_\beta^* W_\alpha \in \text{\rm span}\left\{W_\gamma W_\sigma^*:\ \gamma,\sigma\in \FF_n^+ \text{ with } |\gamma|\leq |\alpha|, |\sigma|\leq |\beta|\right\} +\boldsymbol\cK(F^2(H_n))
$$
 for  any $\alpha, \beta\in \FF_n^+$.
 On the other hand, according to the first part of the theorem, we have
$$
\boldsymbol\cK(F^2(H_n))\subset \overline{\text{\rm span}}\left\{W_\alpha W_\beta^*:\ \alpha,\beta\in \FF_n^+ \right\}.
$$
Combining these results, we deduce that 
$C^*({\bf W})= \overline{\text{\rm span}}\{W_\alpha W_\beta^*:\ \alpha,\beta\in \FF_n^+\}.
$
 This completes the proof.
\end{proof}
We remark that Theorem \ref{compact} holds, in particular, for  the admissible free holomorphic functions of Example \ref{scale}.
Here is  a $C^*$-algebra  version  of the Wold decomposition for  the unital 
$*$-representations  of  the $C^*$-algebra $C^*({\bf W})$. 

\begin{theorem}  \label{wold} Let $g$ be an admissible free holomorphic function  and 
let  ${\bf W}=(W_1,\ldots, W_n)$ be  the  weighted left creation operators associated with $g$.
  If the series $\sum_{k=0}^\infty\sum_{|\alpha|=k} a_\alpha W_\alpha W_\alpha^*$ is convergent in the operator norm topology, where $g^{-1}=1+\sum_{|\alpha|\geq 1} a_\alpha Z_\alpha$, and   
$\pi:C^*({\bf W})\to B(\cK)$ is  a unital
$*$-representation  of $C^*({\bf W})$ on a separable Hilbert
space  $\cK$, then $\pi$ decomposes into a direct sum
$$
\pi=\pi_0\oplus \pi_1 \  \text{ on  } \ \cK=\cK_0\oplus \cK_1,
$$
where $\pi_0$ and  $\pi_1$  are disjoint representations of
$C^*({\bf W})$ on the Hilbert spaces
\begin{equation*}
\begin{split}
\cK_0:&=\overline{\text{\rm span}}\left\{\pi(W_\beta)
\Delta_{g^{-1}}(\pi({\bf W}),\pi({\bf W})^*)\cK:\ \beta\in
\FF_n^+\right\}\quad \text{ and } \\
 \cK_1:&=\cK\ominus \cK_0,
\end{split}
\end{equation*}
 respectively, where $\pi({\bf W}):=(\pi(W_1),\ldots, \pi(W_n))$. Moreover, up to an isomorphism,
\begin{equation*}
\cK_0\simeq F^2(H_n)\otimes \cG, \quad  \pi_0(X)=X\otimes I_\cG \quad
\text{ for } \  X\in C^*({\bf W}),
\end{equation*}
 where $\cG$ is a Hilbert space with
$$
\dim \cG=\dim \left[\text{\rm range}\,\Delta_{g^{-1}}(\pi({\bf W}),\pi({\bf W})^*)\right],
$$
 and $\pi_1$ is a $*$-representation  which annihilates the compact operators  in $C^*({\bf W})$  and
$$
\Delta_{g^{-1}}(\pi_1({\bf W}),\pi_1({\bf W})^*)=0.
$$
If $\pi'$ is another $*$-representation  of $C^*({\bf W})$  on a separable Hilbert space $\cK'$, then $\pi$ is unitarily equivalent to  $\pi'$ if and only if $\dim \cG=\dim \cG'$ and $\pi_1$ is unitarily equivalent to $\pi_1'$.
 \end{theorem}
 \begin{proof}
 According to Theorem \ref{compact},   all the
compact operators $ \boldsymbol\cK(F^2(H_n))$ in $B(F^2(H_n))$ are contained in the
$C^*$-algebra $C^*({\bf W})$.
 Due to  standard theory of
representations of  $C^*$-algebras \cite{Arv-book},
the representation $\pi$ decomposes into a direct sum
$\pi=\pi_0\oplus \pi_1$ on $ \cK=\cK_0\oplus \cK_1$,
where $\pi_0$, $\pi_1$  are disjoint representations of
$C^*({\bf W})$ on the Hilbert spaces
$$\cK_0:=\overline{\text{\rm span}}\{\pi(X)\cK:\ X\in  \boldsymbol\cK(F^2(H_n))\}
\quad \text{ and  }\quad  \cK_1:=\cK_0^\perp,
$$
respectively,
such that
 $\pi_1$ annihilates  the compact operators in $B(F^2(H_n))$ and
  $\pi_0$ is uniquely determined by the action of $\pi$ on the
  ideal $\boldsymbol \cK(F^2(H_n))$ of compact operators.
Since every representation of $  \boldsymbol\cK(F^2(H_n))$ is equivalent to a
multiple of the identity representation, we deduce
 that
\begin{equation*}
\cK_0\simeq\cN_J\otimes \cG, \quad  \pi_0(X)=X\otimes I_\cG, \quad
X\in C^*({\bf W}),
\end{equation*}
 for some Hilbert space $\cG$.
 Using  the proof of Theorem \ref{compact},  part (i),  one can
 see
that
\begin{equation*}\begin{split}
\cK_0&:=\overline{\text{\rm span}}\{\pi(X)\cK:\ X\in \boldsymbol\cK(F^2(H_n)))\}\\
&=\overline{\text{\rm span}}\{\pi(W_\beta P_{\CC 1} W_\alpha^*)\cK:\
 \alpha, \beta\in \FF_n^+\}\\
&= \overline{\text{\rm span}}\left\{\pi(W_\beta)\Delta_{g^{-1}}(\pi({\bf W}),\pi({\bf W})^*)\cK:\ \beta\in
\FF_n^+\right\}.
\end{split}
\end{equation*}
Since $\Delta_{g^{-1}}({\bf W},{\bf W}^*)=P_{\CC 1}\in C^*({\bf W})$, we have
$
\Delta_{g^{-1}}(\pi_1({\bf W}),\pi_1({\bf W})^*)=0$
  and
$
\dim \cG=\dim \left[\text{\rm range}\,\pi(P_{\CC 1})\right].
$
To prove the uniqueness, note that
according to the standard theory of representations of  $C^*$-algebras,
$\pi$ and $\pi'$ are unitarily equivalent if and only if
 $\pi_0$ and $\pi_0'$ (resp.~$\pi_1$ and $\pi_1'$) are unitarily equivalent. Using Theorem \ref{irreducible},  part (ii), we deduce that $\dim\cG=\dim\cG'$.
The proof is complete.
 \end{proof}

Here is our geometric version  of the Wold decomposition for unital 
$*$-representation  of  the $C^*$-algebra $C^*({\bf W})$.  This extends the corresponding result from 
\cite{Po-wold} to our more general setting.

\begin{theorem} \label{wold2}  Let $g$ be an admissible free holomorphic function  and 
let  ${\bf W}=(W_1,\ldots, W_n)$ be  the  weighted left creation operators associated with $g$ such 
   that the series $\sum_{k=0}^\infty\sum_{|\alpha|=k} a_\alpha W_\alpha W_\alpha^*$ is convergent in the operator norm topology, where $g^{-1}=1+\sum_{|\alpha|\geq 1} a_\alpha Z_\alpha$.
Let   $\pi$ be    a unital
$*$-representation  of the $C^*$-algebra $C^*({\bf W})$  on a separable Hilbert
space  $\cK$ and
   set
$V_i:=\pi(W_i)$. Then the  noncommutative Berezin kernel ${ K_{g,V}}$   is a   partial isometry. Setting
$$
\cK^{(0)}:=\text{\rm range}\, K_{g,V}^* \ \text{ and }  \ \cK^{(1)}:= \ker K_{g,V},
$$
the orthogonal decomposition  
 $\cK=\cK^{(0)}\bigoplus \cK^{(1)}$  has  the following properties.
\begin{enumerate}
\item[(i)] $\cK^{(0)}$ and $ \cK^{(1)}$ are reducing  subspaces for each operator $V_i$.

\item[(ii)]  $V|_{\cK^{(0)}}:=(V_1|_{\cK^{(0)}},\ldots, V_n|_{\cK^{(0)}})$ is  a pure tuple in
 $\cD_{g^{-1}}(\cK^{(0)})$.

\item[(iii)]   $V|_{\cK^{(1)}}:=(V_1|_{\cK^{(1)}},\ldots, V_n|_{\cK^{(1)}})$  is a Cuntz tuple in
$\cD_{g^{-1}}(\cK^{(1)})$.

\end{enumerate}
Moreover, 
 the  Berezin kernel
   ${ K_{g,V}}|_{\cK^{(0)}}$ is a unitary operator satisfying relation
     $$V_{i}|_{\cK^{(0)}}=\left({ K_{g,V}}|_{\cK^{(0)}}\right)^* \left(W_{i}\otimes I_{\cD}\right)\left({ K_{g,V}}|_{\cK^{(0)}}\right),\qquad i\in \{1,\ldots, n\},  $$
where $\cD:=\text{\rm range}\, \Delta_{g^{-1}}( V, V^*)$ and
$\Delta_{g^{-1}}( V, V^*)
$
is an orthogonal projection. 

In addition,  the orthogonal decomposition of $\cK$ is uniquely determined by   the properties (i), (ii), and (iii)  and we have 
$$
\cK^{(0)}=\bigoplus_{\alpha\in \FF_n^+}  V_{\alpha} (\cD),\qquad  \text{where} \quad
\cD=\cK\ominus\left(\bigoplus_{i=1}^n  \overline{V_i\cK}\right),
$$
and  $$\cK^{(1)}=\bigcap_{s=0}^\infty\left(\bigoplus_{\alpha\in \FF_n^+, |\alpha|=s}   \overline{V_{\alpha} (\cK)}\right).
$$
\end{theorem}
\begin{proof} Due to Theorem \ref{wold},  the Hilbert space $\cK$ admits an   orthogonal decomposition
$\cK=\cK_0\oplus \cK_1$ such that  $\cK_0$ and $ \cK_1$ are reducing subspaces for each operator $V_i$ and  there is a unitary operator $U:\cK_0\to F^2(H_n)\otimes \cG$ with the property that
$V_i|_{\cK_0}=U^*(W_i\otimes I_\cG)U$ for every $i\in \{1,\ldots, n\}$. 
With respect to the    orthogonal decomposition $\cK=\cK_0\oplus \cK_1$, we have
$\Delta_{g^{-1}}( V, V^*)|_{ \cK_1}=0$,
\begin{equation*}
 V_i=\left(\begin{matrix} V_i|_{\cK_0}&0\\0&V_i|_{\cK_1}
\end{matrix}\right)\quad \text{ and }\quad  \Delta_{g^{-1}}( V, V^*)=\left(\begin{matrix}\Delta_{g^{-1}}( V, V^*)|_{\cK_0}&0\\0&0
\end{matrix}\right).
\end{equation*}
Taking into account that   the series $\sum_{k=0}^\infty\sum_{|\alpha|=k} a_\alpha W_\alpha W_\alpha^*$ is convergent in the operator norm topology, 
$$\Delta_{g^{-1}}( {\bf W}, {\bf W}^*)=\sum_{k=0}^\infty\sum_{|\alpha|=k} a_\alpha W_\alpha W_\alpha^*=P_{\CC 1}
$$ is an orthogonal projection in $C^*({\bf W})$. Since
$\pi$ is a $*$-representation of $C^*({\bf W})$, we deduce that    the operators
$\Delta_{g^{-1}}( V, V^*)$ and   $\Delta_{g^{-1}}( V, V^*)|_{\cK_0}$ are orthogonal projections.
Consequently,  we have
\begin{equation*}
\cD:=\Delta_{g^{-1}}( V, V^*)\cK=\Delta_{g^{-1}}( V, V^*)\cK_0.
\end{equation*}
 Since $
  b_\beta W_\beta P_{\CC 1}
W_\beta^* =P_{\CC e_\beta}$  and  $\Delta_{g^{-1}}( {\bf W}, {\bf W}^*) =P_{\CC 1}
\in C^*({\bf W})$, and using the fact that  $\pi$ is a $*$-representation of $C^*({\bf W})$, we deduce that the operators 
$b_\beta^{(m)}V_\beta  \Delta_{g^{-1}}(V,V^*) V_\beta^*$,   where
$\beta\in \FF_n^+$, are pairwise orthogonal projections and 
\begin{equation*}
K_{g,V}^* K_{g,V}
=\sum_{\beta\in \FF_n^+} b_\beta V_\beta
\Delta_{g^{-1}}( V, V^*) V_\beta^*
\end{equation*}
is an orthogonal projection. This shows, in particular, that $V=(V_1,\ldots, V_n)\in \cD_{g^{-1}}(\cK)$.
On the other hand, with respect to   orthogonal decomposition
$\cK=\cK_0\oplus \cK_1$, 
 the noncommutative Berezin kernel $K_{g,V}$ has the representation
$$
K_{g,V}=\left[ K_{g,V|_{\cK_0}} \ 0\right]:\cK_0\oplus \cK_1\to
\cD.
$$
Since $V|_{\cK_0}$  is a pure $n$-tuple in $\cD_{g^{-1}}(\cK_0)$, the   Berezin kernel
   ${ K_{g,V|_{\cK_0}}}$ is an isometry satisfying relation
     $$V_{i}|_{\cK_0}=\left({ K_{g,V|_{\cK_0}}}\right)^* \left(W_{i}\otimes I_{\cD}\right)\left({ K_{g,V}}|_{\cK_0}\right),\qquad i\in \{1,\ldots, n\}.$$
Moreover,  since $V_i|_{\cK_0}=U^*(W_i\otimes I_\cG)U$ for every $i\in \{1,\ldots, n\}$ and $K_{g, {\bf W}\otimes I_\cG}$ is a unitary operator, so is  ${ K_{g,V|_{\cK_0}}}$.
Since  $K_{g,V}$ is a partial isometry, we have
$$
\cK^{(0)}:=\text{\rm range}\, K_{g,V}^* =\cK_0\ \text{ and }  \ \cK^{(1)}:= \ker K_{g,V}=\cK_1.
$$
The fact that the orthogonal decomposition of $\cK$ is uniquely determined by   the properties (i), (ii), and (iii)
is a consequence of Theorem \ref{wold}.
Using the definition of  the  Berezin kernel $K_{g,V}$, we obtain
$$
 \cK^{(0)} =\left\{ \xi\in \cK:\ \sum_{\beta\in \FF_n^+} b_\beta V_\beta
\Delta_{g^{-1}}( V, V^*)V_\beta^*\xi=\xi\right\}
$$
  and
\begin{equation}
\label{K1}
\cK^{(1)}  =\left\{ \xi\in \cK:\ \Delta_{g^{-1}}( V, V^*)V_\alpha^*\xi=0 \ \text{ for every } \  \alpha\in \FF_n^+\right\}.
 \end{equation}
Since the operators 
$b_\beta V_\beta  \Delta_{g^{-1}}(V,V^*) V_\beta^*$,  
$\beta\in \FF_n^+$, are pairwise orthogonal projections,  we deduce that $\sqrt{b_\beta} V_\beta  \Delta_{g^{-1}}(V,V^*)$ is a partial isometry with range equal to $V_\beta \cD$   and $V_\alpha \cD \perp V_\beta \cD$ if $\alpha\neq \beta$. On the other hand, due to Theorem \ref{wold},  
 we have  $ \cK_0=\overline{\text{\rm span}} \{V_\alpha \cD:\ \alpha\in \FF_n^+\}$. 
  Now, it is clear that
  \begin{equation}
  \label{K0}
  \cK^{(0)}=\bigoplus_{\alpha\in \FF_n^+}   V_{\alpha} (\cD).
  \end{equation}
Due to relation \eqref{K1}, if $\xi\in \cK_1$ then $\Delta_{g^{-1}}( V, V^*)\xi=0$, which implies 
$\xi\in \overline{\Span}_{ |\alpha|=1}V_\alpha\cK^{(1)}$. Since  $ \cK_1$ are reducing subspaces for each operator $V_i$, we obtain  $\cK^{(1)}= \overline{\Span}_{ |\alpha|=1}V_\alpha\cK^{(1)}$.
Note that the later relation implies 
  $\cK^{(1)}=  \Span_{ |\alpha|=s} V_\alpha \cK^{(1)}$ for any $s\in \NN$ and, consequently, we have 
\begin{equation}
\label{sp}\cK^{(1)}=\bigcap_{s=0}^\infty \Span_{ |\alpha|=s} V_\alpha \cK^{(1)}.
\end{equation}
Using the orthogonal decomposition
$\cK=\cK_0\oplus \cK_1$ and that $\cK_0$ and $ \cK_1$ are reducing subspaces for each operator $V_i$,
one can  use the relation \eqref{sp} to  prove that 
$$
\bigcap_{s=0}^\infty\overline{\Span}_{ |\alpha|=s} V_\alpha \cK
=\left(\bigcap_{s=0}^\infty\overline{\Span}_{ |\alpha|=s} V_\alpha \cK^{(0)}\right)\bigoplus \cK^{(1)}.
$$
Due to relation \eqref{K0}, we have  $\bigcap_{s=0}^\infty\overline{\Span}_{ |\alpha|=s} V_\alpha \cK^{(0)}=\{0\}$. Therefore,
$
\bigcap_{s=0}^\infty\overline{\Span}_{ |\alpha|=s} V_\alpha \cK=\cK^{(1)}.
$
Since $V_i^* V_j=0$  for any $i,j\in \{1,\ldots, n\}$ with $i\neq j$, 
we have
 $
\overline{\Span}_{ |\alpha|=s} V_\alpha \cK=\bigoplus_{|\alpha|=s} \overline{ V_\alpha \cK}.
$
Note that the orthogonal decomposition  \eqref{K0} and the fact that  $\cK^{(1)}= \overline{\Span}_{ |\alpha|=1}V_\alpha\cK^{(1)}$  can be used  to show that 
$\cD=\cK\ominus\left(\bigoplus_{i=1}^n  \overline{V_i\cK}\right)$.
This completes the proof.
\end{proof}
We remark that Theorem \ref{wold} holds, in particular, for  the admissible free holomorphic functions provided by  Example \ref{scale}.

\bigskip

 \begin{remark}  The  subspaces $\cK^{(0)}$ and $ \cK^{(1)}$   in Theorem \ref{wold2} satisfy the relations
   $$
 \cK^{(0)} =\left\{ \xi\in \cK:\ \sum_{\beta\in \FF_n^+} b_\beta V_\beta
\Delta_{g^{-1}}( V, V^*)V_\beta^*\xi=\xi\right\}
$$
  and
$$
\cK^{(1)}  =\left\{ \xi\in \cK:\ \Delta_{g^{-1}}( V, V^*)V_\alpha^*\xi=0 \ \text{ for every } \  \alpha\in \FF_n^+\right\}. 
 $$
 \end{remark}

Let $\pi$ be a $*$-representation of $C^*({\bf W})$  on a Hilbert space $\cK$ and  set $V:=(V_1,\ldots, V_n)$, where 
$V_i:=\pi(W_i)$ for $i\in\{1,\ldots, n\}$.
 We saw in the proof of Theorem \ref{wold2} that if  the series $\sum_{k=0}^\infty\sum_{|\alpha|=k} a_\alpha W_\alpha W_\alpha^*$ is convergent in the operator norm topology,  then $\Delta_{g^{-1}}( V, V^*)$  is an orthogonal projection and 
  the operators 
$b_\beta V_\beta  \Delta_{g^{-1}}(V,V^*) V_\beta^*$,  
$\beta\in \FF_n^+$, are pairwise orthogonal projections. Hence,
$
\sum_{\beta\in \FF_n^+} b_\beta V_\beta  \Delta_{g^{-1}}(V,V^*) V_\beta^*
$
is an orthogonal projection  and  $V\in \cD_{g^{-1}}(\cK)$.  

A unital $*$-representation $\pi$  of the $C^*$-algebra  $C^*({\bf W})$ such that 
  $
  \Delta_{g^{-1}}(\pi({\bf W}),\pi({\bf W}^*))=0,
  $
is called {\it Cuntz type representation}. 

 \begin{definition} Let $g$ be an admissible free holomorphic function  and 
let  ${\bf W}=(W_1,\ldots, W_n)$ be  the  weighted left creation operators associated with $g$ such 
   that the series $\sum_{k=0}^\infty\sum_{|\alpha|=k} a_\alpha W_\alpha W_\alpha^*$ is convergent in the operator norm topology, where $g^{-1}=1+\sum_{|\alpha|\geq 1} a_\alpha Z_\alpha$.
  The algebra  $\cO(g)$ is   the universal $C^*$-algebra  generated by $\pi(W_1),\ldots, \pi(W_n)$ and the identity, where 
  $\pi$ is a unital $*$-representation of the $C^*$-algebra  $C^*({\bf W})$ such that 
  $
  \Delta_{g^{-1}}(\pi({\bf W}),\pi({\bf W}^*))=0,
  $
  where  $\pi({\bf W}):=(\pi(W_1),\ldots, \pi(W_n))$.
  \end{definition}
  
  Note that if $g^{-1}=1-(Z_1+\cdots + Z_n)$, $n\geq 2$,  then $\cO(g)$  coincides with the Cuntz algebra $\cO_n$ (see \cite{Cu}).

   \begin{lemma}  \label{Cuntz}
    Let $\{\pi_\omega\}_\omega $ be the collection  of all    Cuntz type $*$-representations of 
$C^*({\bf W})$ such that $C^*(\pi_\omega({\bf W}))$ is irreducible. The universal algebra  $\cO(g)$   is  $*$-isomorphic  to  the $C^*$-algebra 
   $C^*(\widetilde\pi({\bf W}))$, where  $\widetilde\pi:=\bigoplus_\omega \pi_\omega$  is the direct sum  representation of $C^*({\bf W})$.
   \end{lemma}
   
 \begin{proof} 
 Note that $\widetilde\pi$ is a Cuntz type $*$-representation of $C^*({\bf W})$ and set $\widetilde V:=(\widetilde V_1,\ldots, \widetilde V_n)$, where 
$\widetilde V_i:=\widetilde \pi(W_i)$ for $i\in\{1,\ldots, n\}$.  We prove that  $C^*(\widetilde V)$ is  $*$-isomorphic  to
 $\cO(g)$. 
 Let $\pi$ be an arbitrary Cuntz type  $*$-representation  of 
 $C^*({\bf W})$ and   set $V:=(V_1,\ldots, V_n)$, where 
$V_i:=\pi(W_i)$ for $i\in\{1,\ldots, n\}$. We need to prove that there is a surjective  $*$-homomorphism $\Omega:C^*(\widetilde V)\to C^*(V)$ such that $\Omega(\widetilde V_i)=V_i$.
   Due to the GNS construction, given a polynomial  $q ({\bf Z}, {\bf Y}\})$  in noncommutative indeterminates ${\bf Z}=\{Z_1,\ldots, Z_n\}$ and ${\bf Y}=\{Y_1,\ldots, Y_n\}$, there is an irreducible representation $\gamma$ of  $C^*( V)$ such that 
\begin{equation*}
 \|\gamma\left(q(V, V^*)\right)\|=  \|q(V,  V^*)\|.
\end{equation*}
Note that $\gamma\circ\pi$ is  an irreducible Cuntz type $*$-representation of $C^*({\bf W})$. Consequently,
$$
\|q((\gamma\circ\pi)({\bf W}), (\gamma\circ\pi)({\bf W})^*)\|\leq \|q(\widetilde V, \widetilde V^*)\|.
$$
 Combining these relations, we obtain
   $
 \|q(V, V^*)\|\leq  \|q(\widetilde V, \widetilde V^*)\|,
$
for all  polynomials  $q ({\bf Z}, {\bf Y}\})$, which proves that there is a surjective  $*$-homomorphism $\Omega:C^*(\widetilde V)\to C^*(V)$ such that $\Omega(\widetilde V_i)=V_i$. The proof is complete.
 \end{proof}
 
 The following result is an analogue of Coburn's theorem   \cite{Co}  for  the $C^*$-algebra generated by  the unilateral  shift  on the Hardy space $H^2$ of the disk  $\DD=\{z\in \CC: |z|<1\}$ and Cuntz' extension \cite{Cu} for the $C^*$-algebra generated by the left creation operators on the full Fock space with $n$ generators.
 
\begin{theorem} \label{exact1}   Let $g$ be an admissible free holomorphic function  with $g^{-1}=1+\sum_{|\alpha|\geq 1} a_\alpha Z_\alpha$ and 
let  ${\bf W}=(W_1,\ldots, W_n)$ be  the  weighted left creation operators associated with $g$.
  If  $\sum_{k=0}^\infty\sum_{|\alpha|=k} a_\alpha W_\alpha W_\alpha^*$ is convergent in the operator norm topology, then the sequence
of $C^*$-algebras
$$
0\to \boldsymbol\cK(F^2(H_n))\to C^*({\bf W})\to \cO(g)\to 0
$$
is exact,  where  $\boldsymbol\cK(F^2(H_n))$ stands for the ideal of all compact
operators  in $B(F^2(H_n))$.
\end{theorem}
\begin{proof}

 Due to Theorem \ref{compact},  the two-sided ideal   $\cJ$  generated by the orthogonal projection
 $\Delta_{g^{-1}}({\bf W}, {\bf W})$ in the $C^*$-algebra $C^*({\bf W})$ coincides with the ideal  $\boldsymbol\cK(F^2(H_n))$ of all compact operators in $B(F^2(H_n))$.
 According to Lemma \ref{Cuntz}, it is enough to  prove that  $C^*(\widetilde V)$ is $*$-isomorphic to the $C^*$-algebra $C^*({\bf W})/_{\cJ}$.
 
 Since  the  Cuntz type $*$-representation $\widetilde \pi:C^*({\bf W})\to C^*(\widetilde V)$ has the property that $\widetilde\pi(\cJ)=0$, it  induces a $*$-representation 
 $\sigma:C^*({\bf W})/_{\cJ}\to C^*(\widetilde V)$
 such that 
 $
 \sigma\left(q({\bf W},{\bf W}^*)+\cJ\right)=q(\widetilde V, \widetilde V^*)
 $
 for every polynomial  $q ({\bf Z}, {\bf Y})$  in noncommutative indeterminates ${\bf Z}=\{Z_1,\ldots, Z_n\}$ and ${\bf Y}=\{Y_1,\ldots, Y_n\}$.
 Therefore, $\sigma$ is surjective and 
 \begin{equation}
 \label{qq}
 \|q(\widetilde V, \widetilde V^*))\|\leq \|q({\bf W},{\bf W}^*)+\cJ\|.
 \end{equation}
 Let $\rho:C^*({\bf W})\to C^*({\bf W})/_{\cJ}$ be the canonical quotient map and note that, due to the fact that
 $\cJ$  is generated by the orthogonal projection
 $\Delta_{g^{-1}}({\bf W}, {\bf W})$,  we have  
     $ \Delta_{g^{-1}}(\rho({\bf W}),\rho({\bf W}^*))=0, $ 
     where    $\rho({\bf W}):=(\rho(W_1),\ldots, \rho(W_n))$. Therefore, $\rho$  is a Cuntz type $*$-representation of $C^*({\bf W})$ and 
 $$
\|q(\rho({\bf W}),\rho({\bf W}^*))\| \leq  \|q(\widetilde V, \widetilde V^*)\|.
$$
Hence and using relation  \eqref{qq}, we obtain
$
\|q(\widetilde V, \widetilde V^*)\|= \|q({\bf W},{\bf W}^*)+\cJ\|.
$
 Now, the map   defined by
 $$
 \Omega(q(\widetilde V, \widetilde V^*)):=q({\bf W},{\bf W}^*)+\cJ
 $$
 for all noncommutative polynomials  $q ({\bf Z}, {\bf Y})$ is  a  contractive $*$-homomorphism on $C^*(\widetilde V)$ which extends by continuity to a $*$-homomorphism  of $C^*(\widetilde V)$ onto $C^*({\bf W})/_\cJ$. It is clear now that $\Omega$ is a $*$-isomorphism of $C^*$-algebras.
  The proof is complete.
\end{proof}

\begin{theorem} \label{C*}    Let $\pi$ be   a  unital
$*$-representation  of $C^*({\bf W})$ on a separable Hilbert
space  $\cK$ and let   $V:=(V_1,\ldots, V_n)$, where $V_{i}:=\pi({W}_{i})$.
If  $\pi$ is not a Cuntz type  $*$-representation, then  the $C^*$-algebras $C^*({\bf W})$  and   $C^*({V})$ are $*$-isomorphic.
\end{theorem}

\begin{proof} According to  Theorem \ref{wold2}, we have the Wold decomposition $\cK=\cK^{(0)}\oplus \cK^{(1)}$ and 
$
V_{i}=({W}_{i}\otimes I_\cD)\bigoplus V_{i}'$, $i\in  \{1,\ldots, k\},
$
 where   $V'_{i}:=V_{i}|_{\cK^{(1)}}$  and 
$ \cD$ is a Hilbert space with $\dim \cD\geq 1$.
Moreover,   ${V}':=(V_1', \ldots, V_k')$ is     a Cuntz   $k$-tuple  in 
$\cD_{g^{-1}}(\cK^{(1)})$.
  As a consequence, we deduce that
 $$
 q(V, V^*)=\left(q({\bf W}, {\bf W}^*)\otimes I_\cD\right)
 \bigoplus q(V',
  V'^*)
 $$
 for any noncommutative polynomial  $q ({\bf Z}, {\bf Y})$.   Since   $\pi|_{\cK^{(1)}}$ is a $*$-representation,  we have
$
\|q(V', V'^*)\|\leq  \|q({\bf W}, {\bf W}^*)\|.
$
On the other hand, since  $\cD\neq \{0\}$, we have
$$
\| q(V, V^*)\|=\max\left\{ \| q({\bf W}, {\bf W}^*))\|,
 \|  q(V',V'^*))\|\right\}=\| q({\bf W}, {\bf W})\|.
$$
 Due to the fact that  $\pi: C^*({\bf W})\to C^*({V})$ is surjective,  we have  
$\|\pi(g)\|=\|g\|$ for all $g\in  C^*({\bf W})$. Therefore,   $\pi$ is a  $*$-isomorphism of $C^*$-algebras.
The proof is complete.
\end{proof}

\begin{theorem}\label{ideal2}  Let $\pi_c:C^*({\bf W})\to B(\cK)$ be   a Cuntz type
$*$-representation  of $C^*({\bf W})$ on a  Hilbert
space  $\cK$, set   $C_{i}:=\pi_c({W}_{i})$ and $V_{i}:={W}_{i}\oplus C_{i}$ for $i\in \{1,\ldots, n\}$.  
Then there is a unique minimal nontrivial  two-sided ideal  ${\cJ_\Delta}$ of $C^*({ V})$  and  
$${\cJ_\Delta}=\widetilde{\boldsymbol\cK} \cap C^*({V})= \boldsymbol\cK(F^2(H_n))\oplus 0,
 $$
 where   $\widetilde{\boldsymbol\cK}$ is the ideal of all  compact operators in $B(F^2(H_n)\oplus \cK)$.
Moreover,  if $V:=(V_1,\ldots, V_n)$, 
then the  sequence
of $C^*$-algebras
$$
0\to {\cJ_\Delta} \to C^*({V})\to \cO(g)\to 0
$$
is exact.
    
  \end{theorem}
\begin{proof}  
Set  ${\bf C}:=(C_1,\ldots, C_n)$, and note that
$$
\Delta_{g^{-1}}(V, V^*)=\Delta_{g^{-1}}({\bf W}, {\bf W}^*)\bigoplus \Delta_{g^{-1}}(C, C^*)=
P_{\CC1}\oplus 0.
$$
 This shows that    $\widetilde{\boldsymbol\cK}\cap C^*({V})$ is a nontrivial ideal in $C^*({V})$. Let $\cJ$  be an arbitrary nontrivial ideal of $C^*({V})$ and let $E\oplus F\in \cJ$ with $E\oplus F\neq 0$.
Then there is a sequence   $q_s\in \CC\left< {\bf Z}, {\bf Y}\right>$, $s\in \NN$,  of polynomials such that
$q_s({\bf W}, {\bf W}^*)\to E$ and $q_s(C, C^*)\to F$ in the operator norm as
 $s\to \infty$.
Due to the fact that 
$$
\|q_s({\bf W}, {\bf W}^*)\bigoplus q_s(C, C^*)\|\leq \|q_s({\bf W}, {\bf W}^*)\|,
$$
we deduce that $\|E\oplus F\|\leq \|E\|$. Since $E\oplus F\neq 0$, we deduce that $E\neq 0$. Consequently, there is $\alpha\in \FF_n^+$ such that $Ee_{\alpha}\neq 0$.
As in the proof of Theorem \ref{compact},  part (ii), one can use  the fact that ${C}$  is a Cuntz  tuple  
 to show that  $ \boldsymbol\cK(F^2(H_n))\oplus 0\subset \cJ$. 
 This shows that $\boldsymbol\cK(F^2(H_n))\oplus 0$ is the minimal nontrivial ideal of $C^*({V})$. 
 Thus ${\cJ_\Delta}=\boldsymbol\cK(F^2(H_n))\oplus 0$.
Since $\boldsymbol\cK(F^2(H_n))\oplus 0\subset  \widetilde{\boldsymbol{\cK}}\cap C^*(V)$, it remains to prove the reverse inclusion.
Indeed,   if we fix an operator  $E\oplus F\in \widetilde{\boldsymbol{\cK}}\cap C^*(V)\subset B(F^2(H_{n})\oplus \cK)$, then $E\in \boldsymbol\cK(F^2(H_{n}))$ and $F$ is a compact operator in $C^*({C})\subset B(\cK)$.
Let  $q_s\in \CC\left< {\bf Z}, {\bf Y}\right>$, $s\in \NN$,  be a sequence of polynomials such that
$q_s({\bf W}, {\bf W}^*)\to E$ and $q_s(C, C^*)\to F$ in the operator norm as
 $s\to \infty$. According to the proof of Theorem \ref{exact1}, we have
 $$
 \|q_s(C, C^*)\|\leq \|q_s({\bf W}, {\bf W}^*)+\boldsymbol\cK(F^2(H_n))\|
 $$
 which implies  $\|F\|\leq \|\rho(E)\|$, where
  $\rho: C^*({\bf W})\to C^*({\bf W})/_{\boldsymbol\cK(F^2(H_n))}$ is the quotient map.  Since $E$ is a compact operator, we have $\|\rho(E)\|=0$, which implies $F=0$. Consequently, we have
$$\boldsymbol\cK(F^2(H_n))\oplus 0=\widetilde{\boldsymbol{\cK}}\cap C^*(V).$$
 The proof of the exactness of the sequence  is similar to that of Theorem \ref{exact1}.   We leave it to the reader.
\end{proof}

   Due to the Wold decomposition of Theorem \ref{wold}, the result of Theorem \ref{ideal2} carries over
  to  $C^*$-algebras   $C^*(\pi({\bf W}))$, where $\pi:C^*({\bf W})\to B(\cK)$ is an arbitrary unital  $*$-representation on a Hilbert space $\cK$  which is not a Cuntz type $*$-representation. We remark that Theorem \ref{exact1}  and Theorem \ref{ideal2} hold, in particular, for  the admissible free holomorphic function provided by  Example \ref{scale}.

      \bigskip

       %

      \end{document}